\documentclass[11pt, notitlepage,reqno]{amsart}

\usepackage{graphicx}
\usepackage{amsmath,amsthm}
\usepackage{amsfonts}
\usepackage{amssymb}
\usepackage{tikz}
\usepackage{subfig}
\usetikzlibrary{matrix,arrows}
\usepackage{bbm}

\title{Quantum Matrices by paths}

\newtheorem{thm}{Theorem}[subsection]
\newtheorem{lem}[thm]{Lemma}

\newtheorem{prop}[thm]{Proposition}
\newtheorem{cor}[thm]{Corollary}

\newtheorem{claim2}{Claim}

\theoremstyle{definition}
\newtheorem{defn}[thm]{Definition}

\newtheorem{rem}[thm]{Remark}

\newtheorem{ex}[thm]{Example}
\newtheorem{note}[thm]{Note}
\newtheorem{notn}[thm]{Notation}

\newcommand{\qmatrix}{
\C{O}_q(\C{M}_{m,n}(\B{K}))
}

\newcommand{\qtorus}{
\mathcal{O}_q((\B{K}^{\times})^{m\times n})
}

\newcommand{\fgraph}{
G_B^{m\times n}
}

\newcommand{\ImidJ}{I\,|\,J}

\newcommand{\lt}{
\ell t}

\newcommand{\qaffine}{
\mathcal{O}_q(\B{K}^{m\times n})
}

\newcommand{\C}{\mathcal}
\newcommand{\B}{\mathbb}

\newcommand{\spec}{\textnormal{Spec}}
\newcommand{\hspec}{\C{H}\textnormal{-Spec}}

\newcommand{\vect}{\boldsymbol}
\newcommand{\mat}{\C{M}_{m,n}(\B{Z})}
\newcommand{\matnonneg}{\C{M}_{m,n}(\B{Z}_{\geq 0})}

\author{Karel Casteels}
\address{
School of Mathematics, Statistics and Actuarial Science \\
University of Kent \\
Canterbury, United Kingdom\\
CT2 7NF
}

\begin{document}

\begin{abstract}

We study, from a combinatorial viewpoint, the \emph{quantized coordinate ring of $m\times n$ matrices} over an infinite field $\B{K}$, $\qmatrix$ (often simply called \emph{quantum matrices}).The first part of this paper shows that $\qmatrix$, which is traditionally defined by generators and relations, can be seen as a subalgebra of a quantum torus by using paths in a certain directed graph. Roughly speaking, we view each generator of $\qmatrix$ as a sum over paths in the graph, each path being assigned an element of the quantum torus. The $\qmatrix$ relations then arise naturally by considering intersecting paths. This viewpoint is closely related to Cauchon's deleting derivations algorithm.

The second part of this paper applies the above to the theory of torus-invariant prime ideals of $\qmatrix$. We prove a conjecture of Goodearl and Lenagan that all such prime ideals, when the quantum parameter $q$ is a non-root of unity, have generating sets consisting of quantum minors. Previously, this result was known to hold only when $\textnormal{char}(\B{K})=0$ and with $q$ transcendental over $\B{Q}$. Our strategy is to prove the stronger result that the quantum minors in a given torus-invariant ideal form a Gr\"obner basis. 

\end{abstract}
\maketitle
\section{Introduction}

The purpose of this paper is to introduce a ``combinatorial model'' of $\qmatrix$, the quantized coordinate ring of $m\times n$ matrices over a field $\B{K}$ (simply called \emph{quantum matrices}). We demonstrate the utility of this model by using it to study the prime spectrum of $\qmatrix$. 

Quantum matrices have generated a good deal of interest since their discovery during the initial development of quantum group theory in the 1980's. This is because not only do quantum matrices underlie many of the traditional quantum groups such as the quantum special and general linear groups, but there are also interesting connections with topics such as braided tensor categories and knot theory. See~\cite{take} for a brief survey. More recently, it has been observed~\cite{gll2, gll, gll3} that the prime spectrum of quantum matrices is deeply related to the theory of totally nonnegative matrices and the \emph{totally nonnegative grassmannian} in the sense of Postnikov~\cite{postnikov}. 

Since the late 1990's, much effort has been expended toward understanding the structure of the prime and primitive spectra of various quantum algebras. Quantum matrices have received particular attention since, while this algebra has a seemingly simple structure (for example, it is an iterated Ore extension over the field $\B{K}$), many problems have proven difficult to resolve. In particular, the machineries employed to analyze $\spec(\qmatrix)$ have tended to use fairly sophisticated viewpoints from noncommutative ring theory and representation theory and even then often require extra restrictions on the base field $\B{K}$ and choice of quantum parameter $q$.

The $\C{H}$-stratification theory of Goodearl and Letzter~\cite{GL} (see also ~\cite{bg}) is an important advancement toward understanding the prime and primitive spectra of some quantum algebras. Briefly, many noncommutative rings support a rational action of a torus $\C{H}$ which allows one to partition the prime spectrum of the ring into finitely many \emph{$\C{H}$-strata}, each $\C{H}$-stratum homeomorphic (with respect to the usual Zariski topology) to the prime spectrum of a Laurent polynomial ring in finitely many commuting indeterminates, and each containing a unique $\C{H}$-invariant prime ideal. Moreover, the primitive ideals of the algebra are precisely those that are maximal within their $\C{H}$-stratum. For these reasons, an important first step towards understanding the prime and primitive spectra is to first study the $\C{H}$-invariant prime ideals called \emph{$\C{H}$-primes.} 

The deleting derivations algorithm of Cauchon~\cite{cauchon1,cauchon2} has also proven quite useful. Roughly speaking, this procedure shows that when the $\C{H}$-stratification theory applies to a given quantum algebra, one can often embed the set of $\C{H}$-primes into the set of $\C{H}$-primes of a \emph{quantum affine space}. This is convenient since quantum affine spaces are typically easy to handle thanks to results of Goodearl and Letzter~\cite{gl6}. The strategy then is to reverse the deleting derivations procedure in order to transfer (more easily obtained) information about the quantum affine space back to information about the quantum algebra.

The $\C{H}$-stratification and the deleting derivations theories both apply to quantum matrices in the generic case, i.e., when the parameter $q$ is a non-root of unity, and so a natural problem is to find generating sets for the $\C{H}$-primes. For $2\times 2$ quantum matrices, this problem is fairly straightforward, yet even the $3\times 3$ case required a significant amount of work by Goodearl and Lenagan~\cite{gl3,gl4}. However, in all cases their generating sets consisted of \emph{quantum minors} and so it was conjectured that this held true in general. Launois~\cite{launois2,launois3} was the first to prove this conjecture under the constraints $\B{K}=\B{C}$ and $q$ transcendental over $\B{Q}$. This was later extended to any $\B{K}$ of characteristic zero~\cite{gll2}. 


An important part of Cauchon's results is a parametrization of the $\C{H}$-primes of quantum matrices using what are now known in the quantum algebra community as \emph{Cauchon diagrams}. It turns out that a Cauchon diagram encodes fundamental information about the corresponding $\C{H}$-stratum. For example, the Krull dimension can be easily calculated from the Cauchon diagram using the main result of~\cite{bcl}. Launois also described an algorithm to find the generators of a given $\C{H}$-prime from its Cauchon diagram, but the calculations involved very quickly become unwieldy. A graph theoretic interpretation of Launois' algorithm provided in~\cite{casteels} forms the starting point for some of the results presented below. In fact, much of Section~\ref{pathsection} may be seen as a combinatorial interpretation of the deleting derivations algorithm.

It is notable that Cauchon diagrams arose independently in work of Postnikov~\cite{postnikov} in his investigations of the totally nonnegative Grassmannian. In this context, Cauchon diagrams are called \reflectbox{L}-diagrams (also Le-diagrams) and have been investigated by several authors (see Lam and Williams~\cite{LW} and Talaska~\cite{tal} in particular). The connections between these two areas and Poisson geometry have been explored by Goodearl, Launois and Lenagan~\cite{gll,gll2}.

Finally, let us also mention that Yakimov~\cite{yakimov, yakimov3} has developed representation theoretic methods with great success. In particular, he has independently verified (and generalized) Goodearl and Lenagan's conjecture, but again, only under the constraint that $\textnormal{char}(\B{K})=0$ and $q$ transcendental over $\B{Q}$. Furthermore, the generating sets obtained are actually smaller than Launois' in general. It is unclear how Yakimov's work relates to the viewpoint presented in this paper, however, recent work of Geiger and Yakimov~\cite{gy} explore the connections between Yakimov's work and Cauchon's, and so there is quite possibly a close relationship. 

As will be reviewed in Section~\ref{basicssection}, the usual description of $\qmatrix$ is by generators and relations. Our approach to $\qmatrix$ is the focus of Section~\ref{pathsection0} where we begin by giving a directed graph and assign elements (``weights'') of a quantum torus to directed paths. We then discuss various subalgebras of the quantum torus generated by sums over path weights. In particular, Corollary~\ref{isocor} shows that quantum matrices can be so obtained. One nice aspect of this is that the quantum matrix relations naturally arise by considering intersecting paths (see the proofs of Theorem~\ref{pathcommutation} and Theorem~\ref{generatorrelations}).

While at first it may appear that the description of quantum matrices ``by paths'' is a mere curiosity, it is in fact an indispensable tool in the bulk of this paper, Section~\ref{gensection}. Here, the Goodearl-Lenagan conjecture is an immediate corollary to a stronger result, Theorem~\ref{genthm}, which states that for \emph{any} infinite field $\B{K}$ and  non-root of unity $q\in \B{K}^*$, the quantum minors in a given $\C{H}$-prime form a Gr\"obner basis with respect to a certain term ordering. The difficulty with this approach is that for a given $\C{H}$-prime of $\qmatrix$, \emph{a priori} we do not know any generating sets at all to which we can apply Buchberger's algorithm, so we must check that the minors form a Gr\"obner basis by direct verification of the definition. The way we do this is by using the strategy noted above for the deleting derivations algorithm. That is, we transfer an (easily obtained) Gr\"obner basis for an $\C{H}$-prime in a quantum affine space to a Gr\"obner basis for an $\C{H}$-prime in quantum matrices.

Finally, many nonstandard terms and notation have been invented for use in this paper. An combined index and glossary is provided in an appendix to assist the reader in more easily locating the definitions should the need arise.

\section{Quantum Matrices} \label{basicssection}

 
Let us first set some data, notation and conventions that are to be used throughout this paper. 
\begin{itemize}
\item Fix: an infinite field $\B{K}$, integers $m,n\geq 2$, and a nonzero, non-root of unity $q\in \B{K}$. 
\item For a positive integer $k$, we set $[k]=\{1,2,\ldots, k\}.$
\item The set of $m\times n$ matrices with integer entries is denoted by $\mat$. The set of $m\times n$ matrices with non-negative integer entries is denoted by $\matnonneg$. 
\item The $(i,j)$-entry of $N\in\mat$ is denoted by $(N)_{i,j}$, and $(i,j)$ is called the \emph{coordinate} of this entry. In view of this, the elements of $[m]\times [n]$ are called coordinates.
\item We often describe relative positions of coordinates using the usual meaning of terms such as north, northwest etc. For example, $(i,j)$ is \emph{northwest} of $(r,s)$ if $i<r$ and $j<s$, and \emph{north} if $i<r$ and $j=s$.
\end{itemize}

The restriction $m,n\geq 2$ is made simply to avoid some inconveniences in various definitions that would occur if $m=1$ or $n=1$. Fortunately, it is already known that all results presented in this paper hold when $m=1$ or $n=1$ since in these cases, all algebras in this paper reduce to quantum affine spaces, and such algebras can be dealt with using results of~\cite{gl6}.

\subsection{The Algebras $R^{(t)}$} \label{Rtsection}

\begin{defn} \label{lexorderdef}
The \emph{lexicographic order} on $[m]\times [n]$ is the total order $<$ obtained by setting 
\begin{align*} (i,j)<(k,\ell) & \Leftrightarrow \textnormal{$i<k$, or, $i=k$ and $j<\ell$.}\end{align*}
If $(i,j)\in [m]\times [n]$, then $(i,j)^-$ denotes the largest element less than $(i,j)$ with respect to the lexicographic order.
\end{defn}


\begin{note} Any reference in this paper relating to an ordering of the coordinates $[m]\times[n]$ is with respect to the lexicographic order.
\end{note}



The algebras in the next definition each have a set of generators indexed by $[m]\times [n]$. It is natural to place these generators as the entries of an $m\times n$ matrix that we call the \emph{matrix of generators}.
\begin{defn} \label{Rtdef}
Let $t\in [mn]$ and set $(r,s)$ to be the $t^{\textnormal{th}}$ smallest coordinate. Define $R^{(t)}$ to be the $\B{K}$-algebra with the $m\times n$ matrix of generators $X=[x_{i,j}]$ subject to the following relations. If $$\begin{bmatrix} a&b\\c&d\end{bmatrix}$$ is any $2\times 2$ submatrix of $X$, then:

\begin{enumerate}
\item $ab=qba$, $cd=qdc$;
\item $ac=qca$, $bd=qdb$;
\item $bc=cb$;
\item $ad=\begin{cases} da, &\textnormal{if $d=x_{k,\ell}$ and $(k,\ell)>(r,s)$;}\\
da+(q-q^{-1})bc, & \textnormal{if $d=x_{k,\ell}$ and $(k,\ell) \leq(r,s)$.}\end{cases}$

\end{enumerate}

\end{defn}

\begin{ex} If $m=2$, $n=3$ and $t=5$, then $(r,s)=(2,2)$ and $R^{(5)}$ has matrix of generators $$\begin{bmatrix} x_{1,1} & x_{1,2} & x_{1,3} \\ x_{2,1} & x_{2,2} & x_{2,3}\end{bmatrix}.$$ The relations corresponding to Part 4 of Definition~\ref{Rtdef} are \begin{align*} x_{1,1}x_{2,2} &= x_{2,2}x_{1,1} + (q-q^{-1})x_{1,2}x_{2,1} \\ x_{1,1}x_{2,3} &= x_{2,3}x_{1,1} \\ x_{1,2}x_{2,3}&=x_{2,3}x_{1,2}.\end{align*}
\end{ex}

The two extremities in the collection of $R^{(t)}$ are of the most interest to us.

\begin{notn}
With respect to the notation in Definition~\ref{Rtdef}:

\begin{enumerate}
\item If $t=1$, then in Part 4 of Definition~\ref{Rtdef} we always have $$ad=da.$$ We call this algebra \emph{$m\times n$ quantum affine space}, denoted $\qaffine$. The entries of the matrix of generators of $\qaffine$ will often be labeled by $t_{i,j}$ for $(i,j)\in [m]\times [n]$. \\

\item If $t=mn$, then in Part 4 of Definition~\ref{Rtdef} we always have $$ad=da+(q-q^{-1})bc.$$ This algebra is the \emph{quantized coordinate ring of $m\times n$ matrices over $\B{K}$}, denoted by $\qmatrix$ and simply referred to as  ($m\times n$) \emph{quantum matrices.}\\

\item The localization of $R^{(1)}=\qaffine$ with respect to the multiplicative set generated by the standard generators $t_{i,j}$ is called the ($m\times n$)\emph{ quantum torus} $\qtorus$. 

\item Two elements $y,z\in R^{(t)}$ will be said to \emph{$q^*$-commute} if there is an integer $r$ such that $yz=q^rzy$. Note that commuting elements $q^*$-commute.

\end{enumerate}
\end{notn}

In later sections, we work intimately with monomials in the generators of $R^{(t)}$, so we here set some notation in this respect. For the remainder of this section, fix $t\in [mn]$ and let $[x_{i,j}]$ be the matrix of generators for $R^{(t)}$.

\begin{notn} \label{monomialnotation}
If $N\in\matnonneg$, then we write $$\vect{x}^N = x_{1,1}^{(N)_{1,1}}x_{1,2}^{(N)_{1,2}}\cdots x_{m,n}^{(N)_{m,n}} \in R^{(t)},$$ written so that the indices obey the lexicographic order from smallest to largest as one goes from left to right. We call such a monomial a \emph{lexicographic term}. Similar notation will be used both for the quantum torus (where $N\in\mat$), and, if $(r,s)$ is the $t^\textnormal{th}$ smallest coordinate, for $R^{(t)}[x_{r,s}^{-1}]$ (where all entries of $N$ are non-negative except possibly the $(r,s)$-entry).


\end{notn}

It is not difficult to check that each $R^{(t)}$ may be written as an iterated Ore extension which immediately yields the following.

\begin{thm} \label{Rtproperties}
The following properties hold for every $t\in[mn]$.
\begin{enumerate}
\item $R^{(t)}$ is a Noetherian domain. 
\item As a $\B{K}$-vector space, $R^{(t)}$ has a basis consisting of the lexicographic terms $\vect{x}^N$ with $N\in \matnonneg$. The same properties also hold for the $m\times n$ quantum torus (but with $N\in \mat$).\qed
\end{enumerate}
\end{thm}

\begin{defn} \label{lexexpressiondef}
The \emph{lexicographic expression} of $a\in R^{(t)}$ is the unique linear combination $a=\sum_{N\in\matnonneg}\alpha_N\vect{x}^N$ of distinct lexicographic terms with $\alpha_N\neq 0$. A lexicographic term in this expression will be called a \emph{lex term} of $a$. 
\end{defn}

For $R^{(1)}=\qaffine$, we will require a slight extension of Theorem~\ref{Rtproperties}. Observe that any monomial $\vect{t}=t_{i_1,j_1}t_{i_2,j_2}\cdots t_{i_\ell,j_\ell}$ in the standard generators of $R^{(1)}$ may be written as $\vect{t}=q^\ell\vect{t}^{M^\textnormal{lex}}$ for some integer $\ell$ and lexicographic term $\vect{t}^{M^\textnormal{lex}}$. Since $q^\ell\neq 0$, the next result follows easily.

\begin{prop} \label{LID} 

For any coordinate $(r,s)$, the set of lexicographic monomials of $\qaffine$ involving only $t_{i,j}$ with $(i,j)>(r,s)$ is linearly independent over the subalgebra generated by the $t_{i,j}$ with $(i,j)\leq (r,s)$. 
Moreover, for a set $\{\vect{t}_1,\vect{t}_2,\ldots,\vect{t}_\ell\}$ of monomials in the standard generators of $\qaffine$, the following are equivalent.
\begin{enumerate}
\item The set $\{\vect{t}_1,\vect{t}_2,\ldots,\vect{t}_\ell\}$ is linearly independent over $\B{K}$.
\item The set $\{\vect{t}_1^{M_1^\textnormal{lex}},\vect{t}_2^{M_2^\textnormal{lex}},\ldots,\vect{t}_\ell^{M_\ell^\textnormal{lex}}\}$ is linearly independent over $\B{K}$.
\item The matrices $M_1^\textnormal{lex},\ldots, M_\ell^\textnormal{lex}$ are distinct.
\end{enumerate}
A similar set of statements hold for the $m\times n$ quantum torus. \qed
\end{prop}

We conclude this section by noting that $R^{(t)}$ has a natural $\B{Z}_{\geq 0}^{m+n}$-grading that will be very much exploited in the proof of Theorem~\ref{genthm}.  If $$\vect{s}=(r_1,r_2,\ldots,r_m,c_1,c_2,\ldots,c_n)\in (\B{Z}_{\geq 0})^{m+n},$$ then the homogeneous component of degree $\vect{s}$ is the subspace of $R^{(t)}$ spanned by the lexicographic monomials of the form $\vect{x}^N$, where $N$ satisfies \begin{align*} \sum_{j=1}^n (N)_{i,j} & = r_i, \textnormal{ for all $i\in [m]$, and} \\
\sum_{i=1}^m (N)_{i,j} & = c_j, \textnormal{ for all $j\in [n]$.} \\
\end{align*}
In other words, the sum of all entries in row $i$ of $N$ equals $r_i$, and the sum of all entries in column $j$ of $N$ equals $c_j$. All references in this paper to a grading on $R^{(t)}$ will be with respect to this grading.

\subsection{The Deleting Derivations Algorithm} \label{DDsection}

The relationship between $R^{(t)}$ and $R^{(t-1)}$ has been studied by Cauchon~\cite{cauchon2} as a special case of the more general theory developed in~\cite{cauchon1}. Here, we review his results as they apply to these algebras. For each result in this section, we fix $t\in[mn]$ with $t\neq1$, let $(r,s)$ denote the $t^\textnormal{th}$ smallest coordinate, and let $[x_{i,j}]$ be the matrix of generators of $R^{(t)}$ and $[y_{i,j}]$ the matrix of generators for $R^{(t-1)}.$

\begin{thm}[Cauchon~\cite{cauchon1}, Lemme 2.1 and Th\'eor\`eme 3.2.1] \label{ddtheorem}
\hspace{5cm}
\begin{enumerate}
\item The multiplicative set generated by $x_{r,s}$ is a left and right Ore set for $R^{(t)}$, and the multiplicative set generated by $y_{r,s}$ is a left and right Ore set for $R^{(t-1)}$. 

\item There is an injective homomorphism $$\overrightarrow{\cdot}: R^{(t-1)}\to R^{(t)}\left[x_{r,s}^{-1}\right]$$ defined on the standard generators by $$\overrightarrow{y_{i,j}} = \begin{cases} x_{i,j}- x_{i,s}x_{r,s}^{-1}x_{r,j}, & \textnormal{ if $i<r$ and $j<s$;} \\ 
x_{i,j} & \textnormal{ otherwise.}
\end{cases}$$

\item There is an injective homomorphism $$\overleftarrow{\cdot}: R^{(t)}\to R^{(t-1)}\left[y_{r,s}^{-1}\right]$$ defined on the standard generators by $$\overleftarrow{x_{i,j}} = \begin{cases} y_{i,j} + y_{i,s}y_{r,s}^{-1}y_{r,j}, & \textnormal{ if $i<r$ and $j<s$;} \\ 
y_{i,j} & \textnormal{ otherwise.}
\end{cases}$$

\item $R^{(t)}\left[x_{r,s}^{-1}\right]= R^{(t-1)}\left[y_{r,s}^{-1}\right]$.\qed
\end{enumerate}
\end{thm}

The homomorphism in Theorem~\ref{ddtheorem}~(2) is called the \emph{deleting derivations map}. We call the homomorphism in Theorem~\ref{ddtheorem}~(3) the  \emph{adding derivations map}. (This map is called the ``reverse deleting derivations map'' in~\cite{launois2}, and a step of the ``restoration'' algorithm in~\cite{gll}.) 

The strategy of Cauchon's theory is to use these maps to iteratively transfer information between $R^{(1)}=\qaffine$ and $R^{(mn)}=\qmatrix$. For example, to embed the prime spectrum of the latter algebra into the prime spectrum of the former. 

As usual, for an algebra $A$, denote by $\spec(A)$ the set of prime ideals, equipped with the Zariski topology. We may partition $\spec(R^{(t)})$ as $$\spec(R^{(t)})= \spec^{\not\in}(R^{(t)})\cup \spec^\in(R^{(t)}),$$ where $$\spec^{\not\in}(R^{(t)})=\{P\in\spec(R^{(t)})\mid x_{r,s}\not\in P\},$$ and $$\spec^\in(R^{(t)})=\{P\in\spec(R^{(t)})\mid x_{r,s}\in P\}.$$

\begin{thm}[Cauchon~\cite{cauchon2}, Section 3.1] \label{Cauchonmap}
There exists an injective map $$\phi_t: \spec(R^{(t)})\to\spec(R^{(t-1)})$$ satisfying the following properties.
\begin{enumerate}
\item Restricted to $\spec^{\not\in}(R^{(t)})$, $\phi_t$ is bijective, sending $P\in\spec^{\not\in}(R^{(t)})$ to $$\phi_t(P)=\overleftarrow{P}[y_{r,s}^{-1}]\cap R^{(t-1)}.$$ If $Q\in\spec^{\not\in}(R^{(t-1)}),$ then $$\phi_t^{-1}(Q)=\overrightarrow{Q}[x_{r,s}^{-1}]\cap R^{(t)}.$$
\item Restricted to $\spec^{\in}(R^{(t)})$, $\phi_t$ is injective, sending $P\in\spec^\in(R^{(t)})$ to $$\phi_t(P) = g^{-1}(P/\langle x_{r,s}\rangle),$$ where $g:R^{(t-1)}\to R^{(t)}/\langle x_{r,s}\rangle$ is the unique homomorphism that maps the standard generators as $y_{i,j}\mapsto x_{i,j}+\langle x_{r,s}\rangle.$\qed
\end{enumerate}
\end{thm}

\subsection{$\C{H}$-Stratification} \label{Hstratsection}

For many quantum algebras, including the $R^{(t)}$, the structure of the prime spectrum may be understood by first understanding the prime ideals that are invariant under a rational action of an algebraic torus $\C{H}$. For $R^{(t)}$ with matrix of generators $[x_{i,j}]$, let $\C{H}=(\B{K}^*)^{m+n}$ and note that every $h=(\rho_1,\ldots,\rho_m,\gamma_1,\ldots,\gamma_n)\in \C{H}$ induces an automorphism of $R^{(t)}$ by $$h\cdot x_{i,j}= \rho_i\gamma_jx_{i,j}.$$ 

\begin{defn}
An \emph{$\C{H}$-prime} is a prime ideal $K\in\spec(R^{(t)})$ such that $h\cdot K=K$ for all $h\in \C{H}$. The set of all $\C{H}$-primes of $R^{(t)}$ is denoted $\hspec(R^{(t)})$. The \emph{$\C{H}$-stratum} associated to an $\C{H}$-prime $K$ is the set $$\spec_K(R^{(t)}) = \{P\in\spec(R^{(t)})\mid \bigcap_{h\in \C{H}} h\cdot P = K\}.$$
\end{defn}

\begin{thm}[Goodearl-Letzter~\cite{GL} (or see~\cite{bg}, Part II)] \label{Hstrattheorem}
For every $t\in[mn]$, there are finitely many $\C{H}$-primes in $\hspec(R^{(t)})$, and 
$$\spec(R^{(t)}) = \bigsqcup_{K\in \hspec(R^{(t)})} \spec_K(R^{(t)}).$$\qed
\end{thm}

\begin{rem}
Theorem~\ref{ddtheorem} and Theorem~\ref{Hstrattheorem} are where it is necessary to require $q$ to be a nonzero, non-root of unity. We also note here that the $\C{H}$-primes are well-known to be homogeneous ideals.
\end{rem}

The $\C{H}$-primes of $R^{(1)}=\qaffine$ have generating sets of a simple form.

\begin{thm}[Goodearl-Letzter~\cite{gl6}, Section 2.1(ii)] \label{grobner1}
A prime ideal $K\in\spec(R^{(1)})$ is an $\C{H}$-prime if and only if there exists a $B\subseteq [m]\times [n]$ such that $$K=\langle t_{i,j}\mid (i,j)\in B\rangle.$$\qed
\end{thm}

It is convenient to describe these $\C{H}$-primes by using diagrams.

\begin{defn} \label{diagramdef}
An $m\times n$ \emph{diagram} is an $m\times n$ grid of squares, each square colored either black or white.
\end{defn}

We index the squares of a diagram as one would the entries of an $m\times n$ matrix. If $$K=\langle t_{i,j}\mid (i,j)\in B\rangle\in\hspec(R^{(1)})$$ for some $B\subseteq [m]\times [n]$, then the diagram corresponding to $K$ as that in which the black squares are precisely those $(i,j)\in B$. Conversely, any diagram defines a subset $B\subseteq[m]\times [n]$ corresponding to the indices of the black squares, and therefore a corresponding $K\in\hspec(R^{(1)})$. We henceforth identify a diagram with the corresponding subset $B\subseteq [m]\times [n]$. Figure~\ref{diagramex} presents two diagrams, the left one corresponding to the $\C{H}$-prime $\langle t_{1,1},t_{2,1},t_{2,3}\rangle\in \hspec(\C{O}_q(\B{K}^{3\times 4})).$.

\begin{figure}[htbp]
\begin{center}
\begin{tikzpicture}[xscale=1.1, yscale=1.1]

\draw [fill=black] (0,1) rectangle (0.5,1.5);
\draw [fill=black] (0,0.5) rectangle (0.5,1);
\draw [fill=black] (1,0.5) rectangle (1.5,1);
\draw [color=gray] (0,0) rectangle (0.5,0.5);
\draw [color=gray] (1,0) rectangle (1.5,0.5);
\draw [color=gray] (1.5,0) rectangle (2,0.5);
\draw [color=gray] (1.5,0.5) rectangle (2,1);
\draw [color=gray] (0,1) rectangle (0.5,1.5);
\draw [color=gray] (1,1) rectangle (1.5,1.5);
\draw [color=gray] (0.5,0) rectangle (1,0.5);
\draw [color=gray] (0.5,0.5) rectangle (1,1);
\draw [color=gray] (0.5,1) rectangle (1,1.5);
\draw [color=gray] (1.5,1) rectangle (2,1.5);

\draw [color=gray, fill=black] (3.5,1) rectangle (4,1.5);
\draw [color=gray,fill=black] (3.5,0) rectangle (4,0.5);
\draw [color=gray,fill=black] (3.5,0.5) rectangle (4,1);
\draw [color=gray,fill=black] (3,0.5) rectangle (3.5,1);
\draw [color=gray,fill=black] (4,0.5) rectangle (4.5,1);
\draw [color=gray,fill=black] (4.5,1) rectangle (5,1.5);
\draw [color=gray] (3,0) rectangle (3.5,0.5);
\draw [color=gray] (4,0) rectangle (4.5,0.5);
\draw [color=gray] (4.5,0) rectangle (5,0.5);
\draw [color=gray] (4.5,0.5) rectangle (5,1);
\draw [color=gray] (3,1) rectangle (3.5,1.5);
\draw [color=gray] (4,1) rectangle (4.5,1.5);
\end{tikzpicture}
\caption{Two $3\times 4$ diagrams.}
\label{diagramex}
\end{center}
\end{figure}
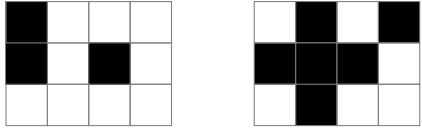

The deleting derivations map behaves nicely with respect to $\C{H}$-primes.
\begin{thm}[Cauchon~\cite{cauchon2}, Section 3.1] \label{Cauchonmap1}
For every $t\in [mn]$, $t\neq 1$,  the map $\phi_t$ injects $\hspec(R^{(t)})$ into $\hspec(R^{(t-1)})$. Consequently, the composition $$\phi=\phi_{2}\circ \cdots \phi_{mn}$$ is an injection of $\hspec(\qmatrix)$ into $\hspec(\qaffine)$.\qed
\end{thm}

In view of the strategy mentioned in Section~\ref{DDsection}, a natural problem is to identify the diagrams of those $\C{H}$-primes in $\hspec(R^{(1)})$ that are the image of an $\C{H}$-prime in $\hspec(R^{(mn)})$ under $\phi$. We call these \emph{Cauchon diagrams}

\begin{defn} \label{cauchondiagramdef}
A diagram is a Cauchon diagram if, for any given black square, either every square to the left or every square above is also black.
\end{defn}

The right diagram in Figure~\ref{diagramex} is an example of a Cauchon diagram, while the left is not a Cauchon diagram since the black square in position $(2,3)$ has a white square both above and to its left.

\begin{thm}[Cauchon~\cite{cauchon2}, The\'eor\`eme 3.2.2] \label{Cauchondiagramtheorem}
A diagram is a Cauchon diagram if and only if the corresponding $\C{H}$-prime in $\hspec(R^{(1)})$ is the image under $\phi$ of an $\C{H}$-prime in $\hspec(R^{(mn)}).$ \qed
\end{thm}


\section{Quantum Matrices by paths} \label{pathsection0}
\subsection{Graphs and Paths} \label{pathsection}
Let $B$ be a Cauchon diagram and, by Theorem~\ref{Cauchondiagramtheorem}, consider the corresponding $\C{H}$-prime $K$ of $\qmatrix$. With the notation of Section~\ref{Hstratsection}, the image of $K$ under the composition $\phi_{t+1}\circ\cdots \circ \phi_{mn}$ is an $\C{H}$-prime $K_t$ of $R^{(t)}$. The goal of this section is to explain how $R^{(t)}/K_t$ is isomorphic to a subalgebra $A_B^{(t)}$ of the quantum torus $\qtorus$ defined by considering paths in a directed graph that is defined using $B$. In particular, when $B=\emptyset$, we obtain a combinatorial description of $\qmatrix$. 


\begin{defn} \label{FactorGraph}
To a Cauchon diagram $B$ construct a directed graph $G_B^{m\times n}$ called the \emph{Cauchon graph}\footnote{``Cauchon graphs'' already appear in~\cite{postnikov} where they are called $\Gamma$-graphs. We here call these Cauchon graphs to be consistent with the Cauchon diagrams from which they derive.} as follows. The vertex set consists of \emph{white vertices} $$W=\left([m]\times [n]\right)\setminus B,$$ together with \emph{row vertices} $R=[m]$, and \emph{column vertices}\footnote{There is ambiguity between labels of the row and column vertices, but the type of vertex we mean will always be explicitly stated.} $C=[n]$.
The set of directed edges $E$ consists precisely of those in the following list.
\begin{itemize}
\item[(1)]  If $(i,j), (i,j^\prime) \in W$ are distinct white vertices with $j>j^\prime$ and such that there is no white vertex $(i,j^{\prime\prime})$ for any $j^\prime<j^{\prime\prime}<j$, then we make an edge from $(i,j)$ to $(i,j^\prime)$;
\item[(2)]  If $(i,j), (i^\prime,j) \in W$ are distinct white vertices with $i<i^\prime$ such that there is no white vertex $(i^{\prime\prime},j)$ for any $i<i^{\prime\prime}<i^\prime$, then we make an edge from $(i,j)$ to $(i^\prime,j)$;
\item[(3)]  For $i\in R,$ we make an edge from $i$ to $(i,j)$, where $j$ is the largest integer such that $(i,j)\in W$ (if such a $j$ exists);
\item[(4)]  For $j\in C,$ we make an edge from $(i,j)$ to $j$ where $i$ is the largest integer such that $(i,j)\in W$ (if such an $i$ exists).
\end{itemize}

\end{defn}

\begin{note} 
There is a natural way to embed a Cauchon graph in the plane by placing it ``on top'' of  the Cauchon diagram $B$ as follows. The white vertices are placed at the center of the corresponding white squares, the row vertices to the right of the corresponding diagram row, and the column vertices underneath the corresponding diagram column. An example is illustrated in Figure~\ref{factorgraphex}. We call this the \emph{standard embedding} and always assume a given Cauchon graph is equipped with it. Hence, without confusion we can refer to aspects of a Cauchon graph using common directional or geometric terms\footnote{For example, horizontal, vertical, above, below, northwest, etc.} That a diagram is a Cauchon diagram easily implies that the corresponding Cauchon graph has the following important property.
\end{note}
\begin{figure}[htbp]
\begin{center}

\begin{tikzpicture}[xscale=1.5, yscale=1.5]

\draw [fill=gray] (7,2) rectangle (8,3);
\draw [fill=gray] (9,1) rectangle (10,3);

\draw [help lines, darkgray] (7,0) grid   (10,3);

\node at (10.7,2.52) {$\bullet$ 1};
\node at (10.7,1.52) {$\bullet$ 2};
\node at (10.7,0.52) {$\bullet$ 3};

\node at (7.505,-0.46) {$\hspace{0.32cm}\bullet$ 1};
\node at (8.505,-0.46) {$\hspace{0.32cm}\bullet$ 2};
\node at (9.505, -0.46) {$\hspace{0.32cm}\bullet$ 3};

\node[scale=0.7]  at (8.5, 2.7) {$(1,2)$};
\node at (8.5, 2.5) {$\bullet$};

\node[scale=0.7]  at (7.5, 1.7) {$(2,1)$};
\node at (7.5, 1.5) {$\bullet$};

\node at (8.5, 1.5) {$\bullet$};
\node[scale=0.7]  at (8.75, 1.7) {$(2,2)$};

\node[scale=0.7]  at (7.75, 0.7) {$(3,1)$};
\node at (7.5, 0.5) {$\bullet$};

\node[scale=0.7] at (8.75, 0.7) {$(3,2)$};
\node at (8.5, 0.5) {$\bullet$};

\node[scale=0.7]  at (9.5, 0.7) {$(3,3)$};
\node at (9.5, 0.5) {$\bullet$};

\draw [<-, thick, black] (8.6,2.5)--(10.5,2.5);
\draw [<-, thick, black] (8.6,1.5)--(10.5,1.5);
\draw [<-, thick, black] (7.6,1.5)--(8.35,1.5);
\draw [<-, thick, black] (7.6,0.5)--(8.35,0.5);
\draw [<-, thick, black] (8.6,0.5)--(9.35,0.5);
\draw [<-, thick, black] (9.6,0.5)--(10.5,0.5);

\draw [->, thick, black] (7.5,0.4)--(7.5,-0.4);
\draw [->, thick, black] (7.5,1.4)--(7.5,0.6);
\draw [->, thick, black] (8.5,0.4)--(8.5,-0.4);
\draw [->, thick, black] (8.5,1.4)--(8.5,0.6);
\draw [->, thick, black] (8.5,2.4)--(8.5,1.6);
\draw [->, thick, black] (9.5,0.4)--(9.5,-0.4);
\end{tikzpicture}
\caption{The graph $G^{3\times 3}_B$, embedded on top of the $3\times 3$ Cauchon diagram $B={\{(1,1),(1,3),(2,3)\}}$.}
\label{factorgraphex}
\end{center}
\end{figure}
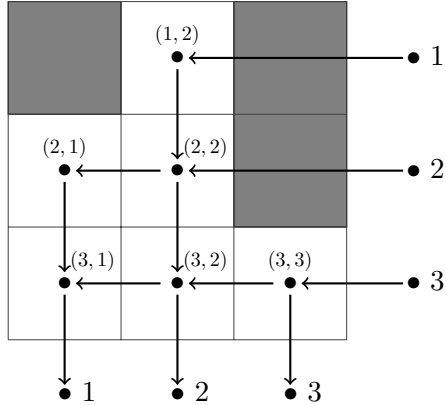
\begin{prop}
The standard embedding of a Cauchon graph is planar. \qed
\end{prop}


\begin{defn} A \emph{path} in $\fgraph$ is a sequence $P=(v_0,v_1,\ldots,v_{k})$ of distinct vertices such that\footnote{strictly speaking, we are defining a \emph{directed} path, but we will never have use for non-directed paths in this paper.}  for all $i\in [k]$, there exists an edge in $\fgraph$ directed from $v_{i-1}$ to $v_{i}$. Naturally, we say that $P$ \emph{starts} at $v_0$ and \emph{ends} at $v_k$ and write $P\colon v_0\to v_k$.
\end{defn}

We consider a directed edge $e$ from $v$ to $w$ to be a path and write $e\colon v\to w$. If $e$ is the edge between two consecutive vertices in a path $P$, then we abuse notation by writing $e\in P$. Finally, if $P\colon u\to v$, $Q\colon v\to w$, then we write $P\cup Q$ to denote the concatenation of $P$ and $Q$. To a path in a Cauchon graph we will assign an element of the quantum torus as follows.




\begin{defn} \label{weightdefn}
Let $\fgraph$ be a Cauchon graph. Define the function $$w\colon E \rightarrow \qtorus$$ as follows, where the numbering and notation correspond to the edge types of Definition~\ref{FactorGraph}:
\begin{itemize}
\item[(1)] $w(e\colon(i,j)\rightarrow(i,j^\prime))=t_{i,j}^{-1}t_{i,j^\prime}$;
\item[(2)] $w(e\colon(i,j)\rightarrow(i^\prime,j))=1$;
\item[(3)] $w(e\colon i\rightarrow (i,j))=t_{i,j}$;
\item[(4)] $w(e\colon(i,j)\rightarrow j)=1$.
\end{itemize}
The image $w(e)$ of an edge $e$ is called the \emph{weight} of $e$. 

If $P=(v_0,v_1,\ldots v_{k})$ is a path, and $e_i\colon v_{i-1}\to v_i$, then the weight of $P$ is defined to be $$w(P)=w(e_1)w(e_2)\cdots w(e_k).$$
\end{defn}

\begin{figure}[htbp]
\begin{center}

\begin{tikzpicture}[xscale=1.9, yscale=1.9]

\node at (10.7,2.52) {$\bullet$ 1};
\node at (10.7,1.52) {$\bullet$ 2};
\node at (10.7,0.52) {$\bullet$ 3};

\node at (7.505,-0.46) {$\hspace{0.32cm}\bullet$ 1};
\node at (8.505,-0.46) {$\hspace{0.32cm}\bullet$ 2};
\node at (9.505, -0.46) {$\hspace{0.32cm}\bullet$ 3};

\node at (8.5, 2.5) {$\bullet$};

\node at (7.5, 1.5) {$\bullet$};

\node at (8.5, 1.5) {$\bullet$};

\node at (7.5, 0.5) {$\bullet$};

\node at (8.5, 0.5) {$\bullet$};

\node at (9.5, 0.5) {$\bullet$};

\draw [<-, thick, black] (8.6,2.5)--(10.5,2.5);
\node[scale=0.75] at (9.55, 2.65) {$t_{1,2}$};

\draw [<-, thick, black] (8.6,1.5)--(10.5,1.5);
\node[scale=0.75] at (9.55, 1.65) {$t_{2,2}$};

\draw [<-, thick, black] (7.6,1.5)--(8.35,1.5);
\node[scale=0.75] at (7.975, 1.65) {$t_{2,2}^{-1}t_{2,1}$};

\draw [<-, thick, black] (7.6,0.5)--(8.35,0.5);
\node[scale=0.75] at (7.975, 0.65) {$t_{3,2}^{-1}t_{3,1}$};

\draw [<-, thick, black] (8.6,0.5)--(9.35,0.5);
\node[scale=0.75] at (8.975, 0.65) {$t_{3,3}^{-1}t_{3,2}$};

\draw [<-, thick, black] (9.6,0.5)--(10.5,0.5);
\node[scale=0.75] at (10.05, 0.65) {$t_{3,3}$};

\draw [->, thick, black] (7.5,0.4)--(7.5,-0.4);
\node[scale=0.75] at (7.3, 0) {$1$};

\draw [->, thick, black] (7.5,1.4)--(7.5,0.6);
\node[scale=0.7] at (7.3, 1) {$1$};

\draw [->, thick, black] (8.5,0.4)--(8.5,-0.4);
\node[scale=0.7] at (8.3, 0) {$1$};

\draw [->, thick, black] (8.5,1.4)--(8.5,0.6);
\node[scale=0.7] at (8.3, 1) {$1$};

\draw [->, thick, black] (8.5,2.4)--(8.5,1.6);
\node[scale=0.7] at (8.3, 2) {$1$};

\draw [->, thick, black] (9.5,0.4)--(9.5,-0.4);
\node[scale=0.7] at (9.3, 0) {$1$};

\end{tikzpicture}
\caption{The graph $G^{3\times 3}_B$, with $B={\{(1,1),(1,3),(2,3)\}}$, and edges labeled by their weights. (Labels of white vertices omitted.)}
\label{factorgraphex2}
\end{center}
\end{figure}
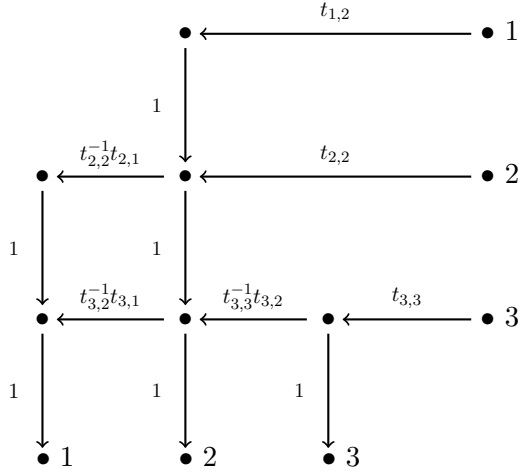

\begin{ex} \label{pathweightex}
 Figure~\ref{factorgraphex2} illustrates the graph of Figure~\ref{factorgraphex} with edges labeled by their weights. The weight of the path $$P=(1, (1,2), (2,2),(2,1),(3,1),1)$$ is \begin{align*} w(P) & =(t_{1,2})(1)(t_{2,2}^{-1}t_{2,1})(1)(1)\\ &= t_{1,2}t_{2,2}^{-1}t_{2,1}.\end{align*}
\end{ex}
It is convenient to observe that for a row vertex $i$ and a column vertex $j$, the weight of a path $P\colon i\to j$ can be computed by looking at the sequence of ``turns''. 

\begin{defn}
Let $P=(v_0,v_1,\ldots ,v_{k-1},v_k)$ be a path in a Cauchon graph starting from row vertex $i=v_0$ and ending at column vertex $j=v_k$. \begin{itemize}\item A \emph{$\Gamma$-turn} in $P$ is a white vertex $v_i\in P$ such that the edge from $v_{i-1}$ to $v_i$ is horizontal, and the edge from $v_i$ to  $v_{i+1}$ is vertical. \item A \reflectbox{L}\emph{-turn} in $P$ is a white vertex $v_i\in P$ such that the edge from $v_{i-1}$ to $v_i$ is vertical and the edge from $v_i$ to $v_{i+1}$ is horizontal. 
\end{itemize}
\end{defn}

The next proposition follows easily using the definitions of edge and path weights.
\begin{prop} \label{turnsprop}
Let $P\colon i\to j$ be a path in a Cauchon graph where $i$ is a row vertex and $j$ is a column vertex.  If $(v_{i_1},v_{i_2},\ldots,v_{i_t})\subset P$ is the subsequence consisting of all $\Gamma$-turns and \textnormal{\reflectbox{L}}-turns, then $$w(P)=t_{v_{i_1}}t^{-1}_{v_{i_2}}t_{v_{i_3}}\cdots t_{v_{i_{t-1}}}^{-1}t_{v_{i_t}}.$$
\end{prop}

\begin{ex}

For the path $P$ in Example~\ref{pathweightex}, the vertex $(1,2)$ is a $\Gamma$-turn, $(2,2)$ is a $\reflectbox{L}$-turn, and $(2,1)$ is a $\Gamma$-turn, so that $w(P)=(t_{1,2})(t_{2,2}^{-1})(t_{2,1}).$ This, of course, agrees with Example~\ref{pathweightex}.

\end{ex}

Parts 1 and 2 of the next result are Lemmas 3.5 and 3.6 respectively in ~\cite{casteels}. Part 3 is proven similarly. 

\begin{lem} \label{pathscommonvertex}
In a Cauchon graph $\fgraph$, let $(a,b)$ be a white vertex,  $i$ and $k$ row vertices with $i<k$, and $j$ and $\ell$ column vertices with $j<\ell$. \begin{enumerate}
\item \label{pathscommonvertex3} If $P\colon i\rightarrow (a,b)$ and $Q\colon (a,b)\rightarrow \ell$ are paths in $\fgraph$ with only $(a,b)$ in common, then
$$w(P)w(Q) = 
\begin{cases}
w(Q)w(P), & \mbox{if $b=\ell$, i.e., $Q$ has only vertical edges,}\\
q^{-1}w(Q)w(P), & \mbox{otherwise.}\end{cases}$$
\item \label{pathscommonvertex2}If $P\colon (a,b)\rightarrow j$ and $Q\colon (a,b)\rightarrow \ell$ are paths in $\fgraph$ with only $(a,b)$ in common, then
$$w(P)w(Q) = 
\begin{cases}
w(Q)w(P), & \mbox{if $b=\ell$, i.e., $Q$ has only vertical edges,}\\
qw(Q)w(P), & \mbox{otherwise.}\end{cases}$$
\item \label{pathscommonvertex1}If $P\colon i\rightarrow (a,b)$ and $Q \colon k\rightarrow (a,b)$ are paths in $\fgraph$ with only $(a,b)$ in common, then $$w(P)w(Q)= qw(Q)w(P).$$

\end{enumerate}
\end{lem}

For the remainder of this section, fix $t\in [mn]$ and let $(r,s)$ be the $t^\textnormal{th}$ smallest coordinate. 

\begin{notn} \label{setofpaths}

For a row vertex $i$ and a column vertex $j$ of $\fgraph$, let $\Gamma^{(t)}_B(i,j)$ denote the set of all  paths $P\colon i\to j$ in $\fgraph$ for which no vertex larger than $(r,s)$ is a $\reflectbox{L}$-turn.

\end{notn}

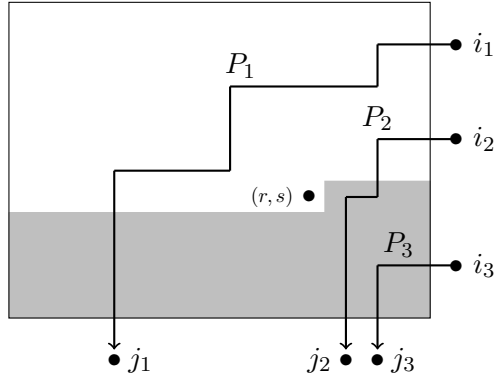
\begin{figure}[htbp]
\begin{center}

\begin{tikzpicture}[xscale=1.4, yscale=1.4]
\draw[color=lightgray, fill=lightgray] (0,0) rectangle (4,1);
\draw[color=lightgray, fill=lightgray] (3,1) rectangle (4,1.3);
\draw (0,0) rectangle (4,3);
\node at (2.85, 1.15) {$\bullet$};
\node[scale=0.7] at (2.5, 1.15) {$(r,s)$};
\node at (4.4, 2.6) {$\bullet$ $i_1$};
\node at (2.2, 2.4) {$P_1$};
\draw [thick, black] (4.2,2.6)--(3.5,2.6);
\draw [thick, black] (3.5,2.6)--(3.5,2.2);
\draw [thick, black] (3.5,2.2)--(2.1,2.2);
\draw [thick, black] (2.1,2.2)--(2.1,1.4);
\draw [thick, black] (2.1,1.4)--(1,1.4);
\draw [->, thick, black] (1,1.4)--(1,-0.3);
\node at (1, -0.4) {$\bullet$};
\node at (1.25, -0.4) {$j_1$};
\node at (4.4, 1.7) {$\bullet$ $i_2$};
\node at (3.5, 1.9) {$P_2$};
\draw [thick, black] (4.2,1.7)--(3.5,1.7);
\draw [thick, black] (3.5,1.7)--(3.5,1.15);
\draw [thick, black] (3.5,1.15)--(3.2,1.15);
\draw [->, thick, black] (3.2,1.15)--(3.2,-0.3);
\node at (3.2, -0.4) {$\bullet$};
\node at (2.95, -0.4) {$j_2$};
\node at (4.4, 0.5) {$\bullet$ $i_3$};
\node at (3.7, 0.7) {$P_3$};
\draw [thick, black] (4.2,0.5)--(3.5,0.5);
\draw [->, thick, black] (3.5,0.5)--(3.5,-0.3);
\node at (3.5, -0.4) {$\bullet$};
\node at (3.75, -0.4) {$j_3$};
\end{tikzpicture}
\caption{The shaded area represents all white vertices greater than the $t^\textnormal{th}$ smallest coordinate $(r,s)$. (This convention will be repeated in later illustrations.) In this example, $P_1\in \Gamma^{(t)}_B(i_1,j_1)$, $P_3\in\Gamma^{(t)}_B(i_3,j_3)$ but $P_2\not\in \Gamma^{(t)}_B(i_2,j_2)$.}
\label{Gammaexample}
\end{center}
\end{figure}

Figure~\ref{Gammaexample} is meant to clarify Notation~\ref{setofpaths}, and while we have drawn a vertex $(r,s)$ in this figure, it will not exist if $(r,s)\in B$. The main theorem of this section is the following. 

\begin{thm} \label{pathcommutation}
Let $\fgraph$ be a Cauchon graph, let $i,k$ be row vertices with $i<k$, and let $j,\ell$ be column vertices.

\begin{enumerate}

\item If $j<\ell$, then there exists a permutation of $\Gamma_B^{(t)}(i,j)\times \Gamma_B^{(t)}(i,\ell)$ sending $(P,Q)\mapsto (\tilde{P},\tilde{Q})$ where $$w(P)w(Q)=qw(\tilde{Q})w(\tilde{P}).$$

\item If $j=\ell$, then there exists a permutation of $\Gamma_B^{(t)}(i,j)\times \Gamma_B^{(t)}(k,j)$ sending $(P,Q)\mapsto (\tilde{P},\tilde{Q})$ where $$w(P)w(Q)=qw(\tilde{Q})w(\tilde{P}).$$

\item If $j>\ell$, then there exists a permutation of $\Gamma_B^{(t)}(i,j)\times \Gamma_B^{(t)}(k,\ell)$ sending $(P,Q)\mapsto (\tilde{P},\tilde{Q})$ where $$w(P)w(Q)=w(\tilde{Q})w(\tilde{P}).$$

\item If $j<\ell$, then:
\begin{enumerate} \item If $P\in \Gamma_B^{(t)}(i,j)$, $Q\in \Gamma_B^{(t)}(k,\ell)$ and $P\cap Q=\emptyset$, then $$w(P)w(Q)=w(Q)w(P);$$
\item There exists a bijective function from the subset of $\Gamma_B^{(t)}(i,j)\times \Gamma_B^{(t)}(k,\ell)$ consisting of those $(P,Q)$ with $P\cap Q\neq \emptyset$, to $\Gamma_B^{(t)}(i,\ell)\times \Gamma_B^{(t)}(k,j)$ sending $(P,Q)$ to $(\tilde{P},\tilde{Q})$ where $$w(P)w(Q)=qw(\tilde{Q})w(\tilde{P}).$$

\end{enumerate}
\end{enumerate}
\end{thm}

\begin{proof}

\noindent\emph{Part 1:} Let $(P,Q)\in \Gamma_B^{(t)}(i,j)\times \Gamma_B^{(t)}(i,\ell)$. Since $j<\ell$, $P$ and $Q$ have a last (white) vertex in common, say $(a,b)$. See Figure~\ref{tailswitchex}. Therefore, we may write $P=P_1\cup P_2$ where $P_1\colon i\to (a,b)$ and $P_2\colon(a,b)\to j$, and $Q=Q_1\cup Q_2$ where $Q_1\colon k\rightarrow (a,b)$ and $Q_2\colon (a,b)\rightarrow \ell$. Define $\tilde{P}=Q_1\cup P_2$ and $\tilde{Q}=P_1\cup Q_2$. We have $(\tilde{P}, \tilde{Q})\in \Gamma_B^{(t)}(i,j)\times \Gamma_B^{(t)}(i,\ell)$ and that $\tilde{\tilde{P}}=P$ and $\tilde{\tilde{Q}}=Q$, i.e., the map $(P,Q)\mapsto (\tilde{P},\tilde{Q})$ is an involution and so a permutation.

Finally, we apply Lemma~\ref{pathscommonvertex} to make our final conclusion as follows. If $Q_2$ has only vertical edges, then \begin{align*} 
w(P)w(Q) & =  w(P_1)w(P_2)w(Q_1)w(Q_2)\\
& =  qw(P_1)w(Q_1)w(P_2)w(Q_2) \mbox{ (Lemma~\ref{pathscommonvertex} (1)),}\\
& =  qw(P_1)w(Q_2)w(Q_1)w(P_2) \mbox{ (Lemma~\ref{pathscommonvertex} (1)\,\&\,(3))}\\
&=  qw(\tilde{Q})w(\tilde{P}).
\end{align*} 
If $Q_2$ has a horizontal edge, then \begin{align*} 
w(P)w(Q) & =  w(P_1)w(P_2)w(Q_1)w(Q_2)\\
& =  q^{-1}qw(P_1)w(Q_2)w(P_2)w(Q_1) \mbox{ (Lemma~\ref{pathscommonvertex} (1))}\\
& =  qw(P_1)w(Q_2)w(Q_1)w(P_2) \mbox{ (Lemma~\ref{pathscommonvertex} (1)\,\&\,(3))} \\
&= qw(\tilde{Q})w(\tilde{P}).
\end{align*}

\begin{figure}[htbp]
\begin{center}

\begin{tikzpicture}[xscale=1.4, yscale=1.4]

\draw (0,0) rectangle (3,2.8);

\node at (3.5, 2.3) {$\bullet$ $i=k$};
\node at (1.08, -0.3) {$\bullet$ $j$};
\node at (1.7, -0.3) {$\bullet$ $\ell$};

\node at (0.8, 1.2) {$P$};
\draw [thick, black] (3.1,2.3)--(2.8,2.3);
\draw [thick, black] (2.8,2.3)--(2.8,1.8);
\draw [thick, black] (2.8,1.8)--(2,1.8);
\draw [thick, black] (2,1.8)--(2, 1.1);
\draw [thick, black] (2,1.1)--(1, 1.1);
\draw [->, thick, black] (1,1.1)--(1, -0.2);

\draw [thick, dashed] (3.1,2.25)--(2.6,2.25);
\draw [thick, dashed] (2.6,2.25)--(2.6,1.6);
\draw [thick, dashed] (2.6,1.6)--(1.6,1.6);
\draw [->, thick, dashed] (1.6,1.6)--(1.6,-0.2);
\node at (1.8, 0.6) {$Q$};
\node at (1.6, 1.1) {$\bullet$};
\node[scale=0.7]  at (1.34, 1.25) {$(a,b)$};

\draw (5,0) rectangle (8,2.8);

\node at (8.5, 2.3) {$\bullet$ $i=k$};
\node at (6.08, -0.3) {$\bullet$ $j$};
\node at (6.7, -0.3) {$\bullet$ $\ell$};

\node at (5.8, 1.2) {$\tilde{P}$};
\draw [thick, dashed] (8.1,2.3)--(7.8,2.3);
\draw [thick, dashed] (7.8,2.3)--(7.8,1.8);
\draw [thick, dashed] (7.8,1.8)--(7,1.8);
\draw [thick, dashed] (7,1.8)--(7, 1.1);
\draw [thick, dashed] (7,1.8)--(7, 1.1);

\draw [thick, dashed] (7,1.1)--(6.6, 1.1);
\draw [thick, black] (6.6,1.1)--(6, 1.1);
\draw [->, thick, black] (6,1.1)--(6, -0.2);

\draw [thick, black] (8.1,2.25)--(7.6,2.25);
\draw [thick, black] (7.6,2.25)--(7.6,1.6);
\draw [thick, black] (7.6,1.6)--(6.6,1.6);
\draw [thick, black] (6.6,1.6)--(6.6,1.1);
\draw [->, thick, dashed] (6.6,1.1)--(6.6,-0.2);

\node at (6.8, 0.6) {$\tilde{Q}$};
\node at (6.6, 1.1) {$\bullet$};
\node[scale=0.7]  at (6.34, 1.25) {$(a,b)$};

\end{tikzpicture}
\caption{Illustration of Part 1 in the proof of Theorem~\ref{pathcommutation}. The left figure shows paths $P$ (solid) and $Q$ (dashed). Right figure shows paths $\tilde{P}$ (solid) and $\tilde{Q}$ (dashed).}
\label{tailswitchex}
\end{center}
\end{figure}
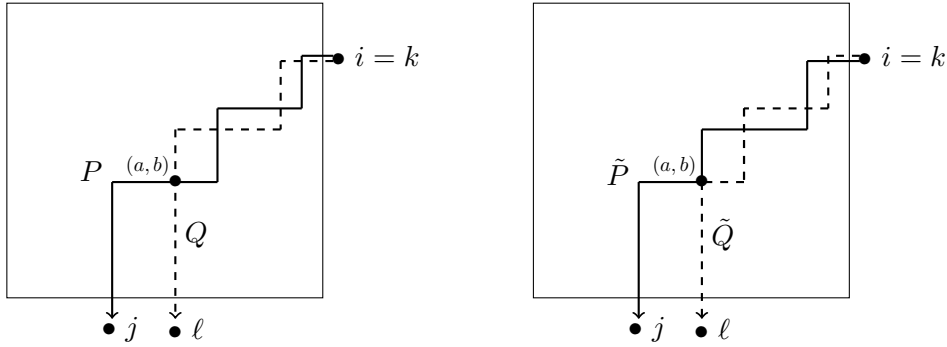

\noindent\emph{Part 2:} Let $(P,Q)\in \Gamma_B^{(t)}(i,j)\times \Gamma_B^{(t)}(k,j)$. In this case, $P$ and $Q$ have a first common vertex, say $(a,b)$. Therefore, we may write  $P=P_1\cup P_2$ where $P_1\colon i\to (a,b)$ and $P_2\colon(a,b)\to j$, and $Q=Q_1\cup Q_2$ where $Q_1\colon k\rightarrow (a,b)$ and $Q_2\colon (a,b)\rightarrow \ell$. Define $\tilde{P}=P_1\cup Q_2$ and $\tilde{Q}=Q_1 \cup P_2$. We again have $(\tilde{P}, \tilde{Q})\in \Gamma_B^{(t)}(i,j)\times \Gamma_B^{(t)}(k,j)$ and that the map $(P,Q)\mapsto (\tilde{P},\tilde{Q})$ is a permutation. The remainder of the proof for Part 2 proceeds as in Part 1 and by using Lemma~\ref{pathscommonvertex}, Parts 1 and 2. 

\noindent\emph{Part 3:} Let $(P,Q)\in \Gamma_B^{(t)}(i,j)\times \Gamma_B^{(t)}(k,\ell)$ where $i<k$ and $j>\ell$. In this case, $P$ and $Q$ have a first common vertex $(a,b)$ and a last common vertex $(a^\prime,b^\prime)$. We can write $P=P_1\cup P_2\cup P_3$ where $P_1\colon i\rightarrow (a,b)$, $P_2\colon (a,b)\rightarrow (a^\prime,b^\prime)$ and $P_3\colon (a^\prime,b^\prime)\to j$. Similarly $Q=Q_1\cup Q_2\cup Q_3$ where $Q_1\colon k\to (a,b)$, $Q_2\colon (a,b)\to (a^\prime,b^\prime)$ and $Q_3\colon (a^\prime,b^\prime) \to \ell$. Define $\tilde{P}=P_1\cup Q_2\cup P_3$ and $\tilde{Q}=Q_1\cup P_2\cup Q_3$. 

We again have $(\tilde{P}, \tilde{Q})\in \Gamma_B^{(t)}(i,j)\times \Gamma_B^{(t)}(k,\ell)$ and that the map $(P,Q)\mapsto (\tilde{P},\tilde{Q})$ is a permutation. To prove the final conclusion concerning the weights relation, we must consider several possibilities according to whether or not any of $P_2,P_3$ and $Q_2$ consists only of vertical edges, or no edges at all (the other paths here always have a horizontal edge). We here only discuss the case that  $P_2,P_3$ and $Q_2$ each have a horizontal edge, the other possibilities being dealt with similarly. Before we begin, we should mention that, strictly speaking, $P_2$ and $Q_2$ do not begin nor end at a row or column vertex, and so Lemma~\ref{pathscommonvertex} does not directly apply. In order to use the lemma, we identify $P_2$ and $Q_2$ respectively with the paths obtained by adding the vertical path from $(a^\prime,b^\prime)$ to $b^\prime$ and the horizontal path from $a$ to $(a,b)$. We can do this since In either case, these latter paths have the same weight as $w(P_2)$ or $w(P_3)$ respectively, by Proposition~\ref{turnsprop}.

We have
 \begin{align*} 
w(P)w(Q) & = w(P_1)w(P_2)w(P_3)w(Q_1)w(Q_2)w(Q_3) \\
& = q w(P_1)w(P_2)w(Q_1)w(Q_2)w(P_3)w(Q_3)\mbox{ (Lemma~\ref{pathscommonvertex} (1))} \\
&= w(P_1)w(P_2)w(Q_1)w(Q_2)w(Q_3)w(P_3)\mbox{ (Lemma~\ref{pathscommonvertex} (2))}\\
&= q^{-1} w(P_1)w(P_2)w(Q_1)w(Q_3)w(Q_2)w(P_3) \mbox{ (Lemma~\ref{pathscommonvertex} (1))} \\
&= w(P_1)w(Q_1)w(P_2)w(Q_3)w(Q_2)w(P_3) \mbox{ (Lemma~\ref{pathscommonvertex} (1))}  \\
&= w(Q_1)w(P_2)w(P_1)w(Q_3)w(Q_2)w(P_3) \mbox{ (Lemma~\ref{pathscommonvertex} (3)\,\&\,(1)),} 
\end{align*} 
where the second line is applying the cited lemma to $P_2$ and $Q_1\cup Q_2$. That the last line is equal to $w(\tilde{Q})w(\tilde{P})$ is now implied by the fact that $w(P_1)$ and $w(Q_3)$ commute. Indeed, we have
\begin{align*}
w(P_1)w(Q_3) & = w(P_1)w(Q_2)^{-1}w(Q_2)w(Q_3) \\
&= q w(Q_2)^{-1}w(P_1)w(Q_2)w(Q_3) \mbox{ (Lemma~\ref{pathscommonvertex} (1))} \\
&= w(Q_2)^{-1}w(Q_2)w(Q_3)w(P_1) \mbox{ (Lemma~\ref{pathscommonvertex} (1))} \\
&= w(Q_3)w(P_1),
\end{align*}
where the third line is applying the cited lemma to $P_3$ and $Q_2\cup Q_3$. 

%
%

\noindent\emph{Part 4a:}  Lemma 3.4 in~\cite{casteels} shows that the weight of any edge not sharing a vertex with $Q$ commutes with $w(Q)$. Since this is the case for all edges of $P$ we immediately have $w(P)w(Q)=w(Q)w(P)$.

\noindent\emph{Part 4b:} As in Part 1, we let $(a,b)$ be the last common vertex in a non-disjoint pair of paths $(P,Q)\in \Gamma_B^{(t)}(i,j)\times \Gamma_B^{(t)}(k,\ell)$. We then ``switch'' the tails of $P$ and $Q$ at $(a,b)$ to obtain a $\tilde{P}\colon i\to \ell$ and a $\tilde{Q}\colon k\to j$. The remainder of the proof is as in Part 1.

\end{proof}

\subsection{The Algebras $A_B^{(t)}$} \label{ABsection}

In this section we introduce, for each $t\in [mn]$ and Cauchon diagram $B$, a subalgebra $A_B^{(t)}$ of $\qtorus$. When $B=\emptyset$, we will see that $A_\emptyset^{(t)}\simeq R^{(t)}$. Throughout this section we fix $t\in [mn]$ and let $(r,s)$ be the $t^\textnormal{th}$ smallest coordinate. 

\begin{defn}\label{ASt}
We define $A_B^{(t)}$ to be the subalgebra of $\qtorus$ with the $m\times n$ matrix of generators $[x_{i,j}]$ where, for each coordinate $(i,j)$,
$$x_{i,j}=\sum_{P\in \Gamma^{(t)}_B(i,j)}  w(P).$$ When $B=\emptyset$ we write $A^{(t)}=A_\emptyset^{(t)}$. 
\end{defn}

\begin{ex} \label{ABtex}
Consider the $2\times 3$ Cauchon diagram $B=\{ (1,1)\}$. Figure~\ref{Aex} presents two copies of the corresponding Cauchon graph, where we continue our illustrative convention that no path may contain a $\reflectbox{L}$-turn in the shaded region. For each $t\in [6]$, we denote by $[x_{i,j}^{(t)}]$ the matrix of generators for $A_B^{(t)}$.

The left graph of Figure~\ref{Aex} corresponds to $t=1$. In this case, any path from row vertex $1$ to column vertex $1$ necessarily contains a $\reflectbox{L}$-turn in the shaded region. Therefore, $A_B^{(1)}$ has the matrix of generators $$\begin{bmatrix} x_{1,1}^{(1)} & x_{1,2}^{(1)}  & x^{(1)} _{1,3} \\ x^{(1)}_{2,1} & x^{(1)}_{2,2} & x^{(1)}_{2,3}\end{bmatrix}=\begin{bmatrix} 0 & t_{1,2} & t_{1,3} \\ t_{2,1} & t_{2,2} & t_{2,3} \end{bmatrix}.$$

One may check that $A_B^{(1)}=A_B^{(2)}=A_B^{(3)} =A_B^{(4)}$. For $t=5$, the Cauchon graph is illustrated on the right in Figure~\ref{Aex}. In this case, there exists a unique path in $\Gamma_B^{(5)}(1,1)$, so that the matrix of generators for $A_B^{(5)}$ is $$\begin{bmatrix} x^{(5)}_{1,1} & x^{(5)}_{1,2} & x^{(5)}_{1,3} \\ x^{(5)}_{2,1} & x^{(5)}_{2,2} & x^{(5)}_{2,3}\end{bmatrix} = \begin{bmatrix} t_{1,2}t_{2,2}^{-1}t_{2,1} & t_{1,2} & t_{1,3} \\ t_{2,1} & t_{2,2} & t_{2,3} \end{bmatrix}.$$ 

Finally, one may check that $A_B^{(6)}$ has matrix of generators 
$$\begin{bmatrix} x^{(6)}_{1,1} & x^{(6)}_{1,2} & x^{(6)}_{1,3} \\ x^{(6)}_{2,1} & x^{(6)}_{2,2} & x^{(6)}_{2,3}\end{bmatrix} =\begin{bmatrix} t_{1,2}t_{2,2}^{-1}t_{2,1} + t_{1,3}t_{2,3}^{-1}t_{2,1} & t_{1,2}+t_{1,3}t_{2,3}^{-1}t_{2,2} & t_{1,3} \\ t_{2,1} & t_{2,2} & t_{2,3} \end{bmatrix}.$$

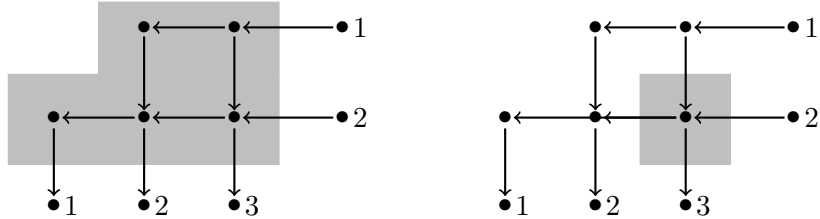
\begin{figure}[htbp]
\begin{center}

\begin{tikzpicture}[xscale=1.2, yscale=1.2]

\draw[color=lightgray, fill=lightgray] (7,0) rectangle (10,1);
\draw[color=lightgray, fill=lightgray] (8,1) rectangle (10,1.8);

\node at (10.7,1.52) {$\bullet$};
\node at (10.7,0.52) {$\bullet$};
\node at (10.9,1.52) {$1$};
\node at (10.9,0.52) {$2$};

\node at (7.505,-0.46) {$\bullet$};
\node at (8.505,-0.46) {$\bullet$};
\node at (9.505, -0.46) {$\bullet$};
\node at (7.7,-0.46) {$1$};
\node at (8.7,-0.46) {$2$};
\node at (9.7, -0.46) {$3$};

\node at (8.505, 1.52) {$\bullet$};
\node at (9.505,1.52) {$\bullet$};

\node at (8.505, 0.52) {$\bullet$};
\node at (9.505,0.52) {$\bullet$};
\node at (7.505, 0.52) {$\bullet$};

\draw [->, thick, black] (10.6,1.52)--(9.6,1.52);
\draw [->, thick, black] (10.6,0.52)--(9.6,0.52);
\draw [->, thick, black] (9.4,1.52)--(8.6,1.52);
\draw [->, thick, black] (9.4,0.52)--(8.6,0.52);
\draw [->, thick, black] (8.4,0.52)--(7.6,0.52);

\draw [->, thick, black] (9.505,1.42)--(9.505,0.6);
\draw [->, thick, black] (8.505,1.42)--(8.505,0.6);
\draw [->, thick, black] (9.505, 0.4)--(9.505,-0.35);
\draw [->, thick, black] (8.505, 0.4)--(8.505,-0.35);
\draw [->, thick, black] (7.505, 0.4)--(7.505,-0.35);

\draw[color=lightgray, fill=lightgray] (14,0) rectangle (15,1);

\node at (15.7,1.52) {$\bullet$};
\node at (15.7,0.52) {$\bullet$};
\node at (15.9,1.52) {$1$};
\node at (15.9,0.52) {$2$};

\node at (12.505,-0.46) {$\bullet$};
\node at (13.505,-0.46) {$\bullet$};
\node at (14.505, -0.46) {$\bullet$};
\node at (12.7,-0.46) {$1$};
\node at (13.7,-0.46) {$2$};
\node at (14.7, -0.46) {$3$};

\node at (13.505, 1.52) {$\bullet$};
\node at (14.505,1.52) {$\bullet$};

\node at (13.505, 0.52) {$\bullet$};
\node at (14.505,0.52) {$\bullet$};
\node at (12.505, 0.52) {$\bullet$};

\draw [->, thick, black] (15.6,1.52)--(14.6,1.52);
\draw [->, thick, black] (15.6,0.52)--(14.6,0.52);
\draw [->, thick, black] (14.4,1.52)--(13.6,1.52);
\draw [->, thick, black] (14.4,0.52)--(13.6,0.52);
\draw [->, thick, black] (14.4,0.52)--(12.6,0.52);

\draw [->, thick, black] (14.505,1.42)--(14.505,0.6);
\draw [->, thick, black] (13.505,1.42)--(13.505,0.6);
\draw [->, thick, black] (14.505, 0.4)--(14.505,-0.35);
\draw [->, thick, black] (13.505, 0.4)--(13.505,-0.35);
\draw [->, thick, black] (12.505, 0.4)--(12.505,-0.35);

\end{tikzpicture}
\caption{Two copies of the graph $G^{2\times 3}_{\{(1,1)\}}$ referred to in Example~\ref{ABtex}. The left picture is shaded to assist the definition of $A_B^{(1)}$, the right picture for $A_B^{(5)}$.}
\label{Aex}
\end{center}
\end{figure}

\end{ex}

Theorem~\ref{pathcommutation} implies some commutation relations between the generators of $A_B^{(t)}$.

\begin{thm}[cf. Definition~\ref{Rtdef}] \label{generatorrelations}
If $X=[x_{i,j}]$ is the matrix of generators for $A_B^{(t)}$, and $$\begin{bmatrix} a&b\\c&d\end{bmatrix}$$ is any $2\times 2$ submatrix of $X$, then:

\begin{enumerate}
\item $ab=qba$, $cd=qdc$;
\item $ac=qca$, $bd=qdb$;
\item $bc=cb$;
\item $ad=\begin{cases} da, &\textnormal{if $d=x_{k,\ell}$ and $(k,\ell)>(r,s)$;}\\
da+(q-q^{-1})bc, & \textnormal{if $d=x_{k,\ell}$ and $(k,\ell)\leq(r,s)$.}\end{cases}$

\end{enumerate}
\end{thm}
\begin{proof}

First note that for any coordinates $(i,j)$ and $(i^\prime,j^\prime)$, 
\begin{align} 
x_{i,j}x_{i^\prime,j^\prime}& =  \sum_{\substack{P\in \Gamma_B^{(t)}(i,j),\\ Q\in \Gamma_B^{(t)}(i^\prime,j^\prime)}} w(P)w(Q)\nonumber \\
&=\sum_{\substack{P, Q\,: \\ P\cap Q=\emptyset}} w(P)w(Q) + \sum_{\substack{P, Q\,:\\ P\cap Q\neq\emptyset}} w(P)w(Q).\label{eq1}
\end{align}

Let $$\begin{bmatrix} a&b\\c&d\end{bmatrix}=\begin{bmatrix} x_{i,j} & x_{i,\ell} \\ x_{k,j} & x_{k,\ell} \end{bmatrix}$$ be a $2\times 2$ submatrix of $X$. 

First, consider $x_{i,j}$ and $x_{i,\ell}$. In this case the first sum in Equation~\eqref{eq1} is necessarily empty, since any pair $(P,Q)\in \Gamma_B^{(t)}(i,j)\times \Gamma_B^{(t)}(i,\ell)$ have row vertex $i$ in common. Part 1 of Theorem~\ref{pathcommutation} shows that for any such pair, there is a unique pair $(\tilde{P},\tilde{Q})\in\Gamma_B^{(t)}(i,j)\times \Gamma_B^{(t)}(i,\ell)$ such that $w(P)w(Q)=qw(\tilde{Q})w(\tilde{P})$. Hence, Equation~\eqref{eq1} implies $x_{i,j}x_{i,\ell} =  qx_{i,\ell}x_{i,j}$
The relations between: $x_{k,j}$ and $x_{k,\ell}$; $x_{i,j}$ and $x_{k,j}$; $x_{i,\ell}$ and $x_{k,\ell}$; and $x_{i,j}$ and $x_{k,j}$ are all obtained similarly.

Now consider $x_{i,j}$ and $x_{k,\ell}$. If $(r,s)< (k,\ell)$, then $$\Gamma_B^{(t)}(k,\ell) =\{Q=(k,(k,\ell),\ell)\}$$ and any $P\in\Gamma_B^{(t)}(i,j)$ is disjoint from $Q$ by definition of $\Gamma_B^{(t)}(i,j)$. Hence $x_{i,j}x_{k,\ell}=x_{k,\ell}x_{i,j}$ by Part 4a of Theorem~\ref{pathcommutation}.
If $(k,\ell)\leq (r,s)$, then by Equation~\eqref{eq1} and Part 4b of Theorem~\ref{pathcommutation}, we obtain
\begin{align*} x_{i,j}x_{k,\ell}& =  qx_{i,\ell}x_{k,j}+\sum_{\substack{P\in \Gamma_B^{(t)}(i,j),\,Q\in \Gamma_B^{(t)}(i,j):\\ P\cap Q=\emptyset}} w(P)w(Q).\end{align*}
Since the weights of disjoint paths commute by Part 4a of Theorem~\ref{pathcommutation}, it follows that $x_{i,j}x_{k,\ell}-x_{k,\ell}x_{i,j} = (q-q^{-1})x_{i,\ell}x_{k,j}.$

\end{proof}

The intuition behind these algebras is that one obtains $A_B^{(t)}$ from $A_B^{(t-1)}$ by ``allowing more paths.'' To be more precise, let $[x_{i,j}]$  be the matrix of generators for $A_B^{(t)}$, and $[y_{i,j}]$ that of $A_B^{(t-1)}$. As elements of $\qtorus$ we have \begin{align}x_{i,j} &= y_{i,j}+\sum w(P),\label{eqn5}\end{align} where the sum is over all paths $P\colon i\to j$ for which $(r,s)$ is a $\reflectbox{L}$-turn in $P$. 
If $i\geq r$, $j\geq s$, or $(r,s)\in B$, then no such $P$ exists and $$x_{i,j} = y_{i,j}.$$ On the other hand, if $(r,s)\not\in B$ and both $i<r$ and $j<s$, suppose $P\colon i\to j$ is a path with a \reflectbox{L}-turn at $(r,s)$.  Consider $w(P)w(Q)$, where $Q=(r,(r,s),s)$. As in the proof of Theorem~\ref{pathcommutation}, we may form paths $\tilde{P}\colon i\to s$ and $\tilde{Q}\colon r\to j$ by ``switching tails'' at $(r,s)$. Since $w(P)w(Q)=qw(\tilde{Q})w(\tilde{P})$, multiplying Equation~\eqref{eqn5} through by $y_{r,s}=x_{r,s}=w(Q)$ gives \begin{align*} x_{i,j}x_{r,s} &= y_{i,j}y_{r,s}+\sum w(P)y_{r,s} \\
&= y_{i,j}y_{r,s} + q y_{i,s}y_{r,j}.\end{align*}

One may easily check that $t_{r,s}=x_{r,s}=y_{r,s}$ generates a left and right Ore set for $A_B^{(t)}$ and $A_B^{(t-1)}$. (For $x_{r,s}$, this follows from the observation that $x_{i,j}x_{r,s}^{m+1}=x_{r,s}^m a$ for some $a\in A_B^{(t)}$ when $x_{i,j}\neq 0$ and $(i,j)$ is northwest of $(r,s)$.) Hence, we have just proved Parts 1 and 2 of the following result. Part 3 follows from these, and Part 4 is trivial. 

\begin{thm}[cf. Proposition 5.4.2 in~\cite{cauchon1}]
The following hold.

\begin{enumerate}
\item If $(r,s)\not\in B$, then $A_B^{(t-1)}$ is a subalgebra of $$A_B^{(t)}[x_{r,s}^{-1}]$$ where $$ y_{i,j}= \begin{cases} x_{i,j} - x_{i,s}\left(x_{r,s}\right)^{-1}x_{r,j}, & \textnormal{ if $i<r$ and $j<s$;} \\  x_{i,j} & \textnormal{ otherwise.}\end{cases}$$

\item If $(r,s)\not\in B$, then $A_B^{(t)}$ is a subalgebra of $$A_B^{(t-1)}[y_{r,s}^{-1}]$$ where $$ x_{i,j} = \begin{cases} y_{i,j} + y_{i,s}\left(y_{r,s}\right)^{-1}y_{r,j}, & \textnormal{ if $i<r$ and $j<s$;} \\  y_{i,j} & \textnormal{ otherwise.}\end{cases}$$

\item If $(r,s)\not\in B$, then $A^{(t)}_B[x_{r,s}^{-1}]= A_B^{(t-1)}[y_{r,s}^{-1}].$
\item If $(r,s)\in B$, then $A_B^{(t)}=A_B^{(t-1)}$.\qed

\end{enumerate}
\end{thm}

In view of Theorem~\ref{ddtheorem}, we conclude the following when $B=\emptyset$.

\begin{cor} \label{isocor}
For every $t\in [mn]$ we have $R^{(t)}\simeq A^{(t)}$, where $R^{(t)}$ are the algebras of Definition~\ref{Rtdef}, and where the standard generator of $R^{(t)}$ with coordinate $(i,j)$ maps to the generator of $A^{(t)}$ with coordinate $(i,j)$.
\end{cor}

Hence, $A^{(1)}\simeq \qaffine$, $A^{(mn)}\simeq\qmatrix$ and both the deleting derivations and $\C{H}$-stratification theories apply to $A^{(t)}$. Moreover, we follow the arrow notation introduced in Section~\ref{DDsection} to distinguish a generator $x_{i,j}$ of $A_B^{(t)}$ from its image $\overleftarrow{x_{i,j}}$ in $A_B^{(t-1)}$, and a generator $y_{i,j}$ of $A_B^{(t-1)}[y_{r,s}^{-1}]$ from its image $\overrightarrow{y_{i,j}}$ in $A_B^{(t)}[x_{r,s}^{-1}]$.

\subsection{$\C{H}$-Primes as Kernels}\label{kernelsection}

Fix $t\in[mn]$ and a Cauchon diagram $B$.  Denote the matrix of generators for $A^{(t)}$ by $[x_{i,j}]$ and the matrix of generators for $A_B^{(t)}$ by $[x_{i,j}^B]$.

\begin{defn} \label{phidef}
For $t\in [mn]$ and a Cauchon diagram $B$, let $\sigma_B^{(t)}:A^{(t)}\to A_B^{(t)}$ be defined on the standard generators by $$\sigma_B^{(t)}(x_{i,j})=x_{i,j}^B.$$
\end{defn}

The content of Section 3.1 of~\cite{cauchon2} imply the following two results.

\begin{prop}\label{welldefinedprop}
The map $\sigma_B^{(t)}$ extends to a well-defined, surjective homomorphism.
\end{prop}


%



\begin{thm} \label{kerthm}
One has $$\ker\left(\sigma_B^{(t)}\right) \in \hspec\left(A^{(t)}\right).$$ Moreover, if $t>1$, $$\ker\left(\sigma_B^{(t-1)}\right)=\phi_t\left(\ker\left(\sigma_B^{(t)}\right)\right),$$ where $\phi_t$ is as in Theorem~\ref{Cauchonmap}.
\end{thm}

We conclude this short section with a technical lemma. For $M\in\matnonneg$, write $M=M_0 + M_1$, where $$(M_0)_{i,j} =\begin{cases} (M)_{i,j} & \textnormal{ if $(i,j)\leq (r,s)$;}\\ 0 & \textnormal{ if $(i,j)> (r,s)$,}\end{cases}$$ and $M_1 = M-M_0$. Now, let $K_t=a\in\ker\left(\sigma_B^{(t)}\right)$. Let $\C{M}$ denote the set of $M\in\matnonneg$ for which $\vect{x}^M$ is a lex term of $a$. Hence, for some $\alpha_M\in\B{K}^*$, we have \begin{align*} a&=\sum_{M\in\C{M}} \alpha_M\vect{x}^M \\ 
&=\sum_{M\in\C{M}} \alpha_M \vect{x}^{M_0}\vect{x}^{M_1} \\
&=
\sum_{N\in\mat} \left(\sum_{\substack{M\in\C{M}:\\ M_1=N_1}} \alpha_M \vect{x}^{M_0}\right)\vect{x}^{N_1}. \end{align*}

Consider \begin{align} \sigma_B^{(t)}(a) = \sum_{N\in\mat} \left(\sum_{\substack{M\in\C{M}: \\M_1=N_1}} \alpha_M \sigma_B^{(t)}(\vect{x}^{M_0})\right)\sigma_B^{(t)}(\vect{x}^{N_1}) = 0\label{Neqn} \end{align}

Let $N\in\mat$. If there is a coordinate $(i,j)>(r,s)$ with both $(i,j)\in B$ and $(N)_{i,j}\geq 1$, then $\vect{x}^{N_1}\in K_t$ since $x_{i,j}=t_{i,j}$ and $\sigma_B^{(t)}(x_{i,j})=0$. Otherwise, $\vect{x}^{N_1}\neq 0$, and the coefficient of $\sigma_B^{(t)}(\vect{x}^{N_1})$ must be $0$ by Proposition~\ref{LID}, i.e., that $$\sum_{\substack{M\in\C{M}:\\ M_1=N_1}} \alpha_M \vect{x}^{M_0}\in K_t.$$ 

\begin{lem}\label{splittingup}
 
With notation as in the preceding two paragraphs, we have that if $a\in K_t$, then $$a=a^\prime+ \sum_{\substack{N\in\mat, \\ \vect{x}^{N_1}\not\in K_t}} a_N\vect{x}^{N_1},$$ where in the second summand each $a_N\in K_t$, and $a^\prime\in K_t$ has the property that every lex term $\vect{x}^L$ of $a^\prime$ satisfies $\vect{x}^{L_1}\in K_t$, i.e., $(L)_{i,j}\geq 1$ for some $(i,j)>(r,s)$ and $(i,j)\in B$. 

%
%
%
%
%
\end{lem}

\section{Generators of $\C{H}$-Primes} \label{gensection}

The goal of this section is the proof of Theorem~\ref{genthm} where we show that an $\C{H}$-prime in $\hspec(\qmatrix)$ has, as a right ideal, a Gr\"obner basis consisting of the quantum minors it contains. That these elements also form a Gr\"obner basis as a left ideal can be shown similarly. 

We begin by defining quantum minors in Section~\ref{minorsection} and recall Theorem 4.4 in ~\cite{casteels} which shows that a $q$-analogue of Lindstr\"oms classic lemma~\cite{lind} holds in the context of Cauchon graphs. We follow this by reviewing the notions of Gr\"obner bases as applied to the algebras $A^{(t)}$, and finally prove the main result in Section~\ref{generatorsection}.

\subsection{Quantum Minors} \label{minorsection}

Throughout this section, we fix a Cauchon diagram $B$ and a $t\in [mn]$. Set $(r,s)$ to be the $t^\textnormal{th}$ smallest coordinate and $[x_{i,j}]$ to be the matrix of generators for $A_B^{(t)}$.

\begin{defn} \label{minordef}

Let $I=\{i_1<i_2<\cdots<i_k\}\subseteq [m]$ and $J=\{j_1<j_2<\cdots<j_k\}\subseteq [n]$ be nonempty subsets of the same cardinality. The \emph{quantum minor} associated to $I$ and $J$ is the element of $A^{(t)}_B$ defined by
\begin{eqnarray*} 
[\ImidJ]_B^{(t)}&= &\sum_{\sigma\in S_k} (-q)^{\ell(\sigma)} x_{i_1,j_{\sigma(1)}}\cdots x_{i_k,j_{\sigma(k)}}
\end{eqnarray*}
where $S_k$ is the set of permutations of $[k]$ and $\ell(\sigma)$ is the number of inversions of $\sigma\in S_k$, i.e., the number of pairs $i,i^\prime \in [k]$ with $i<i^\prime$ but $\sigma(i)>\sigma(i^\prime)$.
\end{defn}

\begin{rem} \label{ltremark}
The defining expression for $[\ImidJ]^{(t)}_B$ is its lexicographic expression. More precisely, for $\sigma\in S_k$, write $P_\sigma$ to be the $m\times n$ matrix whose submatrix indexed by $(I,J)$ equals the standard $k\times k$ permutation matrix corresponding to $\sigma$, and where all other entries of $P_\sigma$ are zero. We can then write \begin{align*}[\ImidJ]^{(t)}_B = \sum_{\sigma\in S_k} (-q)^{\ell(\sigma)} \vect{x}^{P_\sigma}. \end{align*}
\end{rem}

We will often write $[\ImidJ]^{(t)}$ for $[\ImidJ]^{(t)}_\emptyset$. However, for the remainder of this section, we write $[\ImidJ] = [\ImidJ]^{(t)}_B.$ For the remainder of this paper we shorten ``quantum minor'' to just ``minor.'' 

\begin{defn} \label{coordinatesdef}
For $I=\{i_1<i_2<\cdots<i_k\}\subseteq [m]$ and $J=\{j_1<j_2<\cdots<j_k\}\subseteq [n]$, each $(i_\ell,j_\ell)$ is called a \emph{diagonal coordinate} of $[\ImidJ]$. Moreover, $(i_k,j_k)$ is the \emph{maximum coordinate} of $[\ImidJ]$.
\end{defn}

As elements of $\qtorus$, each minor whose maximum coordinate is at most $(r,s)$ reduces to a particularly nice form via a $q$-analogue of Lindstr\"om's Lemma. To explain, we first need to set some notation. At this point, the reader may wish to recall some of the notation set in Section~\ref{pathsection}. 

\begin{defn} \label{pathsystem} 
Let $I=\{i_1,\ldots,i_k\}\subseteq [m]$ and $J=\{j_1,\ldots,j_k\}\subseteq [n]$ be such that $|I|=|J|=k$. 

\begin{enumerate} \item A \emph{vertex-disjoint path system} from the row vertices $I$ to the column vertices $J$ in $\fgraph$ is a set of $k$ mutually disjoint paths $(P_1,\ldots, P_k)$ where $P_r \in \Gamma_B^{(t)}(i_r,j_r)$ for each $r\in [k]$. We write $$ \Gamma_B^{(t)}(\ImidJ) = \{\textnormal{all vertex-disjoint path systems from $I$ to $J$ in $\fgraph$}\}.$$

\item If $\mathcal{P}=(P_1,\ldots, P_k)\in\Gamma_B^{(t)}(\ImidJ)$, then the \emph{weight} of $\C{P}$ is the product $$ w(\C{P})=w(P_1)w(P_2)\cdots w(P_k)\in\qtorus.$$
\end{enumerate}
\end{defn}

\begin{notn}
If we wish to explicitly write out the elements of $I$ and $J$ in either $[\ImidJ]$ or $\Gamma_B^{(t)}(\ImidJ)$, we will omit the braces. For example, we write $$[\ImidJ]=[\{i_1,\ldots,i_k\}\,|\,\{j_1,\ldots,j_k\}] = [i_1,\ldots,i_k\,|\, j_1,\ldots,j_k].$$

\end{notn}
\begin{ex}\label{pathsystemexample}
For the Cauchon graph of Figure~\ref{pathsystemex}, the path system $\C{P}=(P_1,P_2,P_3)$ where \begin{align*}P_1&=(1,(1,3),(1,2),(2,2),(4,2),(4,1),1),\\ P_2 &= (2,(2,3), (3,3),(4,3),3),\\ P_3&= (4,(4,4),4)\end{align*} is a vertex-disjoint path system in $\Gamma^{(16)}_B(1,2,3\,|\,1,3,4)$. In fact, it is the unique such vertex-disjoint path system and $$w(\C{P})=(t_{1,2}t_{4,2}^{-1}t_{4,1})(t_{2,3})(t_{3,4}).$$

The reader may verify that the set $\Gamma^{(16)}_B(1,2\,|\,1,2)$ is empty.

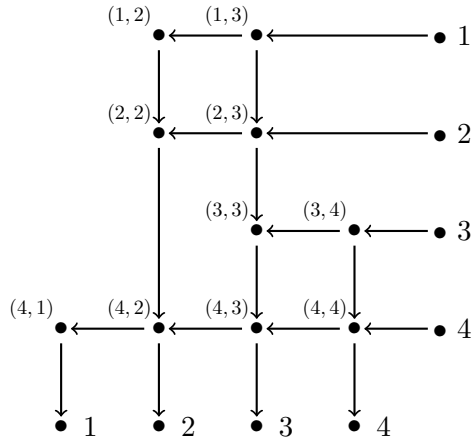
\begin{figure}[htbp]
\begin{center}

\begin{tikzpicture}[xscale=1.3, yscale=1.3]

\node at (0,0) {$\bullet$};
\node[scale=0.7] at (-0.3, 0.2) {$(4,1)$};
\node at (1,0) {$\bullet$};
\node[scale=0.7] at (0.7, 0.2) {$(4,2)$};
\node at (2,0) {$\bullet$};
\node[scale=0.7] at (1.7, 0.2) {$(4,3)$};
\node at (3,0) {$\bullet$};
\node[scale=0.7] at (2.7, 0.2) {$(4,4)$};

\node at (2,1) {$\bullet$};
\node[scale=0.7] at (1.7, 1.2) {$(3,3)$};
\node at (3,1) {$\bullet$};
\node[scale=0.7] at (2.7, 1.2) {$(3,4)$};

\node at (1,2) {$\bullet$};
\node[scale=0.7] at (0.7, 2.2) {$(2,2)$};
\node at (2,2) {$\bullet$};
\node[scale=0.7] at (1.7, 2.2) {$(2,3)$};
\node at (1,3) {$\bullet$};
\node[scale=0.7] at (0.7, 3.2) {$(1,2)$};
\node at (2,3) {$\bullet$};
\node[scale=0.7] at (1.7, 3.2) {$(1,3)$};

\node at (0,-1) {$\bullet$};
\node at (1,-1) {$\bullet$};
\node at (2,-1) {$\bullet$};
\node at (3,-1) {$\bullet$};
\node at (0.3,-1) {$1$};
\node at (1.3,-1) {$2$};
\node at (2.3,-1) {$3$};
\node at (3.3,-1) {$4$};
\node at (4,0) {$\bullet$ $4$};
\node at (4,1) {$\bullet$ $3$};
\node at (4,2) {$\bullet$ $2$};
\node at (4,3) {$\bullet$ $1$};

\draw [->, thick, black] (3.75,0)--(3.1,0);
\draw [->, thick, black] (3.75,1)--(3.1,1);
\draw [->, thick, black] (3.75,2)--(2.1,2);
\draw [->, thick, black] (3.75,3)--(2.1,3);

\draw [->, thick, black] (2.85,0)--(2.1,0);
\draw [->, thick, black] (2.85,1)--(2.1,1);
\draw [->, thick, black] (1.85,2)--(1.1,2);
\draw [->, thick, black] (1.85,3)--(1.1,3);
\draw [->, thick, black] (1.85,0)--(1.1,0);
\draw [->, thick, black] (0.85,0)--(0.1,0);

\draw [->, thick, black] (0,-0.15)--(0,-0.9);
\draw [->, thick, black] (1,-0.15)--(1,-0.9);
\draw [->, thick, black] (2,-0.15)--(2,-0.9);
\draw [->, thick, black] (3,-0.15)--(3,-0.9);

\draw [->, thick, black] (2,0.85)--(2,0.1);
\draw [->, thick, black] (3,0.85)--(3,0.1);

\draw [->, thick, black] (1,1.85)--(1,0.1);
\draw [->, thick, black] (2,1.85)--(2,1.1);
\draw [->, thick, black] (1,2.85)--(1,2.1);
\draw [->, thick, black] (2,2.85)--(2,2.1);

\end{tikzpicture}
\caption{A Cauchon graph.}
\label{pathsystemex}
\end{center}
\end{figure}

\end{ex}

The following is the $q$-analogue of a special case of Lindstr\"om's Lemma. 

\begin{thm}[\cite{casteels}, Theorem 4.4] \label{lindstrom}

If $[\ImidJ]$ has maximum coordinate at most $(r,s)$, then, as an element of $\qtorus$, $$[\ImidJ]=\sum_{\C{P}\in \Gamma_B^{(t)}(\ImidJ)} w(\C{P}).$$ \qed \end{thm}

The proof in~\cite{casteels} deals with the case $t=mn$ and uses a technique similar to the ``tail-switching'' method of Theorem~\ref{pathcommutation}. The same proof is valid here due to the assumption that the maximum coordinate of the minor is at most $(r,s)$.

\begin{ex} \label{pathsystemexample2}
In the Cauchon graph of Figure~\ref{pathsystemex}, say with $t=16$, there is no vertex-disjoint path system from $\{1,2\}$ to $\{1,2\}$. Theorem~\ref{lindstrom} tells us that $[1,2\,|\,1,2] = 0$. This may be verified directly: \begin{align*} [1,2\,|\,1,2] & = x_{1,1}x_{2,2}-qx_{1,2}x_{2,1} \\ &= (t_{1,2}t_{4,2}^{-1}t_{4,1}+ t_{1,3}t_{2,3}^{-1}t_{2,2}t_{4,2}^{-1}t_{4,1} + t_{1,3}t_{4,3}^{-1}t_{4,1})(t_{2,2}+t_{2,3}t_{4,3}^{-1}t_{4,2}) \\ &- q(t_{1,2}+t_{1,3}t_{2,3}^{-1}t_{1,3}t_{4,3}^{-1}t_{4,2})(t_{2,2}t_{4,2}^{-1}t_{4,1}+t_{2,3}t_{4,3}^{-1}t_{4,1})\\ &= 0\end{align*}

Similarly, if one so wishes, it may be checked that \begin{align*}[1,2,3\,|\,1,3,4] &= x_{1,1}x_{2,3}x_{3,4} - q x_{1,1}x_{2,4}x_{3,3} - qx_{1,3}x_{2,1}x_{3,4} - q^3x_{1,4}x_{2,3}x_{3,1} \\ &+ q^2x_{1,3}x_{2,4}x_{3,1}+q^2x_{1,4}x_{2,1}x_{3,3}\\&= w(P_1)w(P_2)w(P_3) \\ &= (t_{1,2}t_{4,2}^{-1}t_{4,1})(t_{2,3})(t_{3,4}),\end{align*} where $P_1,P_2$ and $P_3$ are as in Example~\ref{pathsystemexample}. 
\end{ex}

Before moving on, a quick application of Theorem~\ref{lindstrom} is worth mentioning: the well-known fact that in $\C{O}_q(\C{M}_{n,n}(\B{K}))$ the \emph{quantum determinant} $$D_q=[1,2\ldots,n\,|\,1,2,\ldots,n]$$ is central. Indeed, it is easy to see that there is exactly one vertex-disjoint path system from $[n]$ to $[n]$ in $G^{n\times n}_\emptyset$, namely $\C{P}=(P_1,\ldots,P_n),$ where $P_i=(i,(i,i),i)$ for each $i\in [n]$. Hence, $$D_q =t_{1,1}t_{2,2}\cdots t_{n,n}.$$ Centrality of $D_q$ follows from the observation that the right hand side commutes with every generator $t^{\pm 1}_{i,j}$ of $\qtorus$.

The next result was given as Theorem 4.5 in~\cite{casteels}, but under the additional assumption that $q$ is transcendental over $\B{Q}$. We here provide a proof for when $q$ is a nonzero, non-root of unity. 

\begin{thm} \label{zerocondition}
A quantum minor $[\ImidJ]$ with maximum coordinate at most $(r,s)$ equals zero if and only if there does \emph{not} exist a vertex-disjoint path system from $I$ to $J$, i.e., if and only if $\Gamma_B^{(t)}(\ImidJ)=\emptyset$.
\end{thm}

\begin{proof}
If $\Gamma_B^{(t)}(\ImidJ)=\emptyset$, then Theorem~\ref{lindstrom} implies that $[\ImidJ]=0$.

Now suppose $\Gamma_B^{(t)}(\ImidJ)\neq \emptyset$, i.e., there is at least one vertex-disjoint path system from $I$ to $J$. The weight of a vertex-disjoint path system $\C{P}$ is equal to $q^\alpha \vect{t}^{M_\C{P}}\in\qtorus$ for some integer $\alpha$, where $$(M_\C{P})_{i,j} =\begin{cases} 1 & \textnormal{ if there is a path in $\C{P}$ with a $\Gamma$-turn at $(i,j)$;} \\
-1 & \textnormal{ if there is a path in $\C{P}$ with a $\reflectbox{L}$-turn at $(i,j)$;}\\
0 & \textnormal{ otherwise.}\end{cases}$$

Therefore, if for any distinct $\C{P},\C{Q}\in\Gamma_B^{(t)}(\ImidJ)$ one has $M_\C{P}\neq M_\C{Q}$, then by Theorem~\ref{lindstrom} and Proposition~\ref{LID}, we may conclude that $[\ImidJ]\neq 0$.

Suppose $\C{P}=(P_1,\ldots,P_k)$ and $\C{Q}=(Q_1,\ldots,Q_k)$ are two vertex-disjoint path systems from $I$ to $J$ and that $M_\C{P}=M_\C{Q}$, i.e., a path in $\C{P}$ has a $\Gamma$-turn (respectively $\reflectbox{L}$-turn) at $(i,j)$ if and only if a path in $Q$ does. We aim to show that $\C{P}=\C{Q}$. First, consider the paths $P_k$ and $Q_k$. Let $(i_k,\ell)$ be the first vertex where $P_k$ turns, and $(i_k,\ell^\prime)$ be the first vertex where $Q_k$ turns. If $\ell>\ell^\prime$, then $Q_k$ goes straight through $(i_k, \ell^\prime)$. However, since $\C{Q}$ contains some path $Q$ that turns at $(i,\ell)$, this implies (since $B$ is a Cauchon diagram) that $Q$ and $Q_k$ intersect, contradicting the choice of $\C{Q}$ as a vertex-disjoint path system. The symmetric case shows that $\ell\not<\ell^\prime$ and hence $\ell=\ell^\prime$. A similar argument can then be applied to the remainder of the turning vertices (if any) in $P_k$ and $Q_k$, from which we conclude that $P_k=Q_k$. Repeating the argument with $P_{k-1}$ and $Q_{k-1}$, etc., we see that $\C{P}=\C{Q}$, as desired.

\end{proof}

\begin{cor} \label{lindstromcor}
(Recall the map $\sigma_B^{(t)}: A^{(t)}\to A_B^{(t)}$ of Section~\ref{kernelsection}.) A quantum minor $[\ImidJ]^{(t)}\in A^{(t)}$ with maximum coordinate at most $(r,s)$ is in $\ker(\sigma_B^{(t)})$ if and only if there does not exist a vertex-disjoint path system from $I$ to $J$ in $\fgraph$, i.e., $\Gamma_B^{(t)}(\ImidJ)=\emptyset$. \qed
\end{cor}

We conclude this section by showing how one may construct new vertex-disjoint path systems from $I$ to $J$ from old. First, suppose $i$ is a row vertex and $j$ is a column vertex in $\fgraph$, and consider two paths $P\colon i\to j$ and $Q\colon i\to j$. Let $(i=v_0,\ldots,v_k=j)$ be the subsequence of all vertices that $P$ and $Q$ have in common. For each $a\in[k]$, let $P_a$ (respectively $Q_a$) denote the sub-path of $P$ (respectively $Q$) starting at $v_{a-1}$ and ending at $v_a$. If $P_a\neq Q_a$, then the first edge of $P_a$ is perpendicular to the first edge of $Q_a$. If the first edge of $P_a$ is horizontal, let us say that $P_a$ is \emph{above} $Q_a$, otherwise $P_a$ is \emph{below} $Q_a$. Now consider the paths $$U_a = \begin{cases} P_a & \textnormal{ if $P_a=Q_a$,} \\ P_a &\textnormal{ if $P_a$ is above $Q_a$,} \\
Q_a &\textnormal{if $Q_a$ is above $P_a$,}\end{cases}$$ 
and 
$$L_a = \begin{cases} P_a & \textnormal{ if $P_a=Q_a$,} \\ P_a &\textnormal{ if $P_a$ is below $Q_a$,} \\
Q_a &\textnormal{ if $Q_a$ is below $P_a$.}\end{cases}$$ 

\begin{defn}\label{ULdef}
With notation as in the preceding paragraph, we let $U(P,Q)\colon i\to j$ be the path $$U(P,Q)=U_1\cup U_2\cup\cdots\cup U_k$$ and $L(P,Q)\colon i\to j$ be the path $$L(P,Q)=L_1\cup L_2\cup\cdots\cup L_k.$$ 
\end{defn}

\begin{ex} \label{Uex}
With respect to Figure~\ref{Uexample}, $U_1$ is the solid path from $i=v_0$ to $v_1$, $U_2$ is the dashed path from $v_1$ to $v_2$, $U_3$ is the solid path from $v_2$ to $v_3$, etc. On the other hand, $L_1$ is the solid path from $i=v_0$ to $v_1$, $L_2$ is the solid path from $v_1$ to $v_2$, $L_3$ is the solid path from $v_2$ to $v_3$, etc.
\begin{figure}[htbp]
\begin{center}

\begin{tikzpicture}[xscale=1.9, yscale=1.9]

\draw (0,0) rectangle (3,3);

\draw[color={rgb:black,1;white,10}, fill={rgb:black,1;white,8}] (2.35,2.45) rectangle (3.5,2.55);
\draw[color={rgb:black,1;white,10}, fill={rgb:black,1;white,8}] (2.35,1.8) rectangle (2.45,2.55);
\draw[color={rgb:black,1;white,10}, fill={rgb:black,1;white,8}] (1.95,1.75) rectangle (2.45,1.85);
\draw[color={rgb:black,1;white,10}, fill={rgb:black,1;white,8}] (1.95,1.05) rectangle (2.05,1.85);
\draw[color={rgb:black,1;white,10}, fill={rgb:black,1;white,8}] (0.95,1.05) rectangle (2.05,1.15);
\draw[color={rgb:black,1;white,10}, fill={rgb:black,1;white,8}] (0.95,-0.2) rectangle (1.05,1.15);

\node at (3.5, 2.5) {$\bullet$};
\node at (4, 2.5) {$i=v_0$};

\node at (1, -0.3) {$\bullet$};
\node at (1.4,-0.3) {$v_6=j$};

\node at (0.8, 1.2) {$P$};
\draw [thick, black] (3.5,2.5)--(2.8,2.5);
\node at (2.8, 2.5) {$\bullet$};
\node[scale=0.7] at (2.8, 2.6) {$v_1$};
\draw [thick, black] (2.8,2.5)--(2.8,1.8);
\node at (2.8, 2.15) {$\bullet$};
\node at (2.8, 1.8) {$\bullet$};
\node at (2.4, 1.8) {$\bullet$};
\draw [thick, black] (2.8,1.8)--(2,1.8);
\node at (2, 1.8) {$\bullet$};
\node[scale=0.7] at (2.5,1.9) {$v_2$};
\draw [thick, black] (2,1.8)--(2, 1.1);
\node at (2, 1.1) {$\bullet$};
\node[scale=0.7] at (2.1,1.2) {$v_3$};
\draw [thick, black] (2,1.1)--(1, 1.1);
\node at (1.33, 1.1) {$\bullet$};
\node[scale=0.7] at (1.67,1.2) {$v_4$};

\node at (1.67, 1.1) {$\bullet$};
\node at (1, 1.1) {$\bullet$};
\draw [->, thick, black] (1,1.1)--(1, -0.2);
\node at (1, 0.7) {$\bullet$};
\node at (1, 0.3) {$\bullet$};

\node[scale=0.7] at (0.85,0.3) {$v_5$};
\draw [->, thick, black] (1,1.1)--(1, -0.2);

\draw [thick, dashed] (3.5,2.52)--(2.8,2.52);
\node at (2.15, 2.6) {$Q$};
\draw [thick, dashed] (2.8,2.52)--(2.4,2.52);
\node at (2.4, 2.52) {$\bullet$};
\draw [thick, dashed] (2.4,2.52)--(2.4,1.8);
\draw [thick, dashed] (2.4,1.8)--(2.4,1.1);
\node at (2.4, 1.1) {$\bullet$};
\draw [thick, dashed] (2.4,1.08)--(2,1.08);
\draw [thick, dashed] (2,1.08)--(1.67,1.08);
\draw [thick, dashed] (1.67,1.08)--(1.67,0.3);
\node at (1.67, 0.3) {$\bullet$};
\node at (1.67, 0.7) {$\bullet$};
\draw [thick, dashed] (1.67,0.3)--(1.02,0.3);
\node at (1.33, 0.3) {$\bullet$};
\draw [->, thick, dashed] (1.02,0.3)--(1.02,-0.2);

\end{tikzpicture}
\caption{$P$ is the solid path; $Q$ is the dashed path; $U(P,Q)$ is the shadowed path.}
\label{Uexample}
\end{center}
\end{figure}
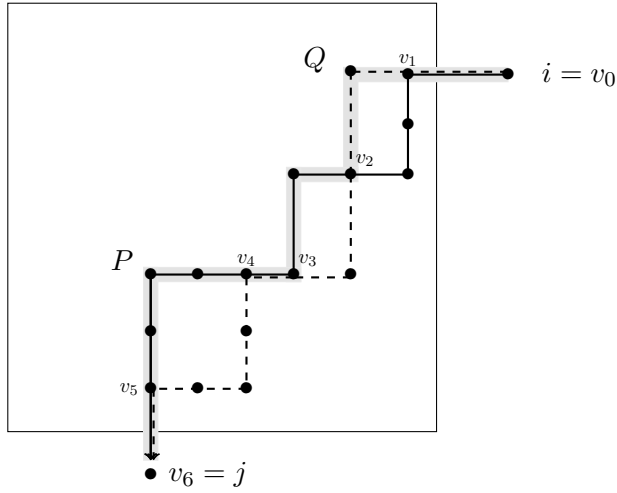

\end{ex}

The following lemma states the key property of $U(P,Q)$ that we require.

\begin{lem} \label{disjointU}
For a row vertex $i$ and column vertex $j$ in $\fgraph$, consider two paths $P\colon i\to j$ and $Q\colon i\to j$. Suppose that $R\colon i^\prime\to j^\prime$ is a path with $i^\prime>i$. If $R$ is disjoint from either $P$ or $Q$, then $R$ is disjoint from $U(P,Q)$. 
\end{lem}

\begin{proof}
With respect to $P$ and $Q$, we use the notation of the paragraph just prior to Example~\ref{Uex}. Without loss of generality, suppose $P$ and $R$ are disjoint. 

If $R$ and $U(P,Q)$ have a vertex $w$ in common, then $w\in Q$ and there exists an $a$ such that $w$ is in the subpath $Q_a$ of $Q$. Since $w\in U(P,Q)$, we have $U_a=Q_a$ for this $a$ and so $Q_a$ is above $P_a$. On the other hand, since $i^\prime>i$, $R$ must intersect the Jordan curve formed by $P_a$ and $Q_a$. Since $\fgraph$ is planar, the intersection occurs at a vertex of $P$, a contradiction.
\end{proof}

\begin{cor}\label{disjointUcor}
Let $i<i^\prime$ be two row vertices and $j<j^\prime$ be two column vertices in $\fgraph$. Suppose $P\colon i\to j$ and $P^\prime\colon i^\prime \to j^\prime$ are disjoint paths and $Q\colon i\to j$ and $Q^\prime\colon i^\prime \to j^\prime$ are disjoint paths. Then $U(P,Q)$ and $U(P^\prime,Q^\prime)$ are disjoint.
\end{cor}
\begin{proof}
By two applications of Lemma~\ref{disjointU}, $U(P,Q)$ is disjoint from both $P^\prime$ and $Q^\prime$. Since $U(P^\prime,Q^\prime)$ consists only of subpaths coming from either $P^\prime$ or $Q^\prime$, we have that $U(P,Q)$ and $U(P^\prime, Q^\prime)$ are disjoint as well.
\end{proof}
Repeated application of Corollary~\ref{disjointUcor} immediately gives the following result.
\begin{cor} \label{newvdps}
Let $\C{P}=(P_1,\ldots, P_k)$ and $\C{Q}=(Q_1,\ldots,Q_k)$ be vertex-disjoint path systems from $I$ to $J$. Then $$U(\C{P},\C{Q})=( U(P_1,Q_1),\ldots, U(P_k,Q_k))$$ is a vertex-disjoint path system from $I$ to $J$. \qed
\end{cor}

Now, if $\Gamma_B^{(t)}(\ImidJ)$ is non-empty, then repeated applications of Corollary~\ref{newvdps} to the finitely many path systems in $\Gamma_B^{(t)}(\ImidJ)$ shows that the next definition is sensible.
\begin{defn}\label{supdef}
If $\Gamma_B^{(t)}(\ImidJ)\neq\emptyset$, then the \emph{supremum} of $\Gamma_B^{(t)}(\ImidJ)$ is the (unique) vertex-disjoint path system $(Q_1,\ldots,Q_k) \in \Gamma_B^{(t)}(\ImidJ)$ such that for any $\C{P}=(P_1,\cdots,P_k)\in \Gamma_B^{(t)}(\ImidJ)$ one has, for each $i\in [k]$, $$U(Q_i,P_i)=Q_i.$$ 

\end{defn}

For $L(P,Q)$, it is clear that results similar to Lemma~\ref{disjointU}, Corollary~\ref{disjointUcor} and Corollary~\ref{newvdps} hold. We omit their explicit statements here, but note that the next definition is also sensible.

\begin{defn}\label{infdef}
If $\Gamma_B^{(t)}(\ImidJ)\neq\emptyset$, then the \emph{infimum} of $\Gamma_B^{(t)}(\ImidJ)$ is the (unique) vertex-disjoint path system $(Q_1,\ldots,Q_k) \in \Gamma_B^{(t)}(\ImidJ)$ such that for any $\C{P}=(P_1,\cdots,P_k)\in \Gamma_B^{(t)}(\ImidJ)$ one has, for each $i\in [k]$, $$L(Q_i,P_i)=Q_i.$$ 

\end{defn}

\begin{ex} Once again, consider the Cauchon graph of Figure~\ref{pathsystemex}. The supremum of $\Gamma_B^{(16)}(1,3\,|\,1,3)$ is the path system $(\tilde{Q}_1,\tilde{Q}_2)$ where \begin{align*}\tilde{Q}_1&=(1,(1,3),(1,2),(2,2),(4,2),(4,1),1),\\ \tilde{Q}_2 &= (3,(3,4),(3,3),(4,3),3),\end{align*} while the infimum of $\Gamma_B^{(16)}(1,3\,|\,1,3)$ is the path system $(Q_1,Q_2)$, where \begin{align*}Q_1&=(1,(1,3),(2,3),(2,2),(4,2),(4,1),1),\\ Q_2 &= (3,(3,4),(4,4), (4,3), 3).\end{align*}
\end{ex}

\subsection{Gr\"obner Bases}\label{grobnersection}

Gr\"obner basis theory is well-known in commutative algebra and fortunately many of its key aspects transfer easily to quantum matrices and the algebras $R^{(t)}\simeq A^{(t)}$. For a more general and detailed account of Gr\"obner basis theory for noncommutative algebras, we refer the reader to the book of Bueso, G\'omez-Torrecillas and Verschoren~\cite{grobner}. 

Throughout this section, we fix $t\in [mn]$, let $(r,s)$ be the $t^\textnormal{th}$ smallest coordinate, and denote the matrix of generators of $A^{(t)}$ by $[x_{i,j}]$.  We now define a total order of the lexicographic monomials in $A^{(t)}$.

\begin{defn} \label{revlex} The \emph{matrix lexicographic order} $\prec$ on $\mat$ is defined as follows. If $M\neq N\in \mat$, let $(k,\ell)$ be the least coordinate in which $M$ and $N$ differ. Then we set \begin{align*}M\prec N \Leftrightarrow (M)_{k,\ell}<(N)_{k,\ell}\end{align*} and say that ``$M\prec N$ \emph{at} $(k,\ell)$.''

If $M\prec N$ are both in $\matnonneg$, then the matrix lexicographic order induces a total order (that we also call matrix lexicographic) on the lexicographic monomials of $A^{(t)}$ by setting 
\begin{align*}\vect{x}^M\prec \vect{x}^{N} \Leftrightarrow M\prec N.\end{align*}
By allowing the $(r,s)$-entry in $M$ and $N$ to be negative, this terminology extends to a total order on the lexicographic monomials of $A^{(t)}[x_{r,s}^{-1}]$.

\end{defn}
For example, under the matrix lexicographic order, we have $$x_{i,j}\prec x_{k,\ell} \Leftrightarrow (i,j)>(k,\ell).$$ 

If $(i,j),(k,\ell)\leq (r,s)$, and $(i,j)$ is northwest of $(k,\ell)$, then we have the relation $$x_{k,\ell}x_{i,j}=x_{i,j}x_{k,\ell}-(q-q^{-1})x_{i,\ell}x_{k,j}.$$ On the other hand, we also have $$x_{i,\ell}x_{k,j}\prec x_{i,j}x_{k,\ell}.$$ Essentially by repeated application of these facts and the other relations amongst the standard generators, we obtain the following, which is a special case of the more general Proposition~2.4 in~\cite{grobner}.

\begin{prop} \label{straightening}
For $M,N\in \matnonneg$, the lexicographic expression of $\vect{x}^M\vect{x}^N$ is $$\vect{x}^M\vect{x}^N=q^\alpha\vect{x}^{M+N} + \sum_{L\in\matnonneg} \alpha_L\vect{x}^{L},$$ for some integer $\alpha$ and where for every $\alpha_L\neq 0$, one has $L\prec M+N$. \qed
\end{prop}

\begin{defn} \label{dividedef}
Let $M, N\in \matnonneg$. We say that $\vect{x}^M$ \emph{divides} $\vect{x}^N$ if $(M)_{i,j}\leq (N)_{i,j}$ for all $(i,j)\in [m]\times [n]$. 
\end{defn}

Using this terminology, we will use Proposition~\ref{straightening} in the following way.

\begin{cor}  \label{strcor}
Let $M,N\in\matnonneg$. If $\vect{x}^M$ divides $\vect{x}^N$, then there exists an integer $\alpha$,  matrices $L\prec N$, and scalars $\alpha_L\in\B{K}^*$  such that
 \begin{align*} \vect{x}^N & =  q^\alpha\vect{x}^M\vect{x}^{N-M} + \sum_L \alpha_L \vect{x}^{L}.\qed
\end{align*}
\end{cor}

\begin{rem}
Proposition~\ref{straightening}, Defintion~\ref{dividedef} and Corollary~\ref{strcor} extend to $A^{(t)}[x_{r,s}^{-1}]$ by allowing the $(r,s)$-entry in each matrix to be negative.
\end{rem}

\begin{defn}\label{leadingtermdef} Let $a\in A^{(t)}$ with lexicographic expression $$a=\sum_L\alpha_L\vect{x}^L.$$ The \emph{leading term} of $a$ is the maximum lex term of $a$ with respect to the matrix lexicographic order. We denote the leading term of $a$ by $\lt(a)$.
\end{defn}

We are now ready to give the definition of a Gr\"obner basis for a right ideal. 

\begin{defn} \label{grobnerdef}
Let $J$ be a right ideal of $A^{(t)}$, and let $$G=\{g_1,g_2,\ldots, g_k\}\subseteq J.$$ We say that $G$ is a \emph{Gr\"obner basis} for $J$ if for every $a\in J$ there exists a $g_i\in G$ such that $\lt(g_i)$ divides $\lt(a)$.
\end{defn}

If one has a Gr\"obner basis $\{g_1,g_2,\ldots, g_k\}$ for a right ideal $J$, then one may find an expression for any $a\in J$ as a combination of the $g_i$ recursively. If $\lt(a)$ is divided by $\lt(g_i)$, then by Corollary~\ref{strcor} we may write $$a=g_i a^\prime+ b$$ where $\lt(b)\prec \lt(a)$. Since $b\in J$, we can repeat the process if $b\neq 0$. As there are only finitely many lexicographic terms smaller than $\lt(a)$, this will end after finitely many steps. Thus, the elements of the Gr\"obner basis generate $J$.


We will eventually deal with quantum minors and in this context require the following, more refined version of Corollary~\ref{strcor}. 

\begin{lem} \label{strcor2}
Let $[\ImidJ]^{(t)}\in A^{(t)}$ be a minor with maximum coordinate $(i_k,j_k)$. Recalling Remark~\ref{ltremark}, if we write \begin{align*}[\ImidJ]^{(t)} = \sum_{\sigma\in S_k} (-q)^{\ell(\sigma)} \vect{x}^{P_\sigma}, \end{align*} then: 
\begin{enumerate}
\item One has $\lt([\ImidJ]^{(t)}) = \vect{x}^{P_{\textnormal{id}}},$ where $\textnormal{id}$ is the identity permutation;


\item If $\vect{x}^{P_{\textnormal{id}}}$ divides $\vect{x}^M$ for some $M\in\matnonneg$, then \begin{align} \vect{x}^M &= q^\alpha [\ImidJ]^{(t)}\vect{x}^{M-{P_{\textnormal{id}}}} + w, \label{minorrelation} \end{align} for some integer $\alpha$ and $w\in A^{(t)}$ where, if $\lt(w)=\vect{x}^K$, then $K\prec M$ at a coordinate northwest of $(i_k,j_k)$.
\end{enumerate}

\end{lem}

The first part of Lemma~\ref{strcor2} is a trivial observation. The justification for the second part is fairly technical, but its heart is the following auxiliary lemma. For this lemma we set $E_{k,\ell}$ to be the $m\times n$ matrix with a $1$ in coordinate $(k,\ell)$ and $0$ elsewhere.

\begin{lem}\label{strcor2lemma} 

If $(i,j)\in [m]\times [n]$ and $\vect{x}^M\in A^{(t)}$ is such that all entries of $M$ in coordinates larger than $(a,b)$ are zero, then we may write \begin{align*} \vect{x}^Mx_{i,j} &= q^\alpha x_{i,j}\vect{x}^M + w,\end{align*} where $\alpha\in\mathbb{Z}$, and if $w\neq 0$ and $\vect{x}^K$ is a lex term of $w$, then $M$ and $K$ are equal in all entries northeast of $(i,j)$. Moreover, if $\lt(w)=\vect{x}^L$, then $L\prec M+E_{i,j}$ at a coordinate northwest of $(i,j)$

\end{lem}

\begin{proof}

We proceed by induction on $j$, starting with the easy observation that for $j=1$, $x_{i,j}$ and $\vect{x}^M$ $q^*$-commute. 

Now, fix $j>1$.  Consider the process of commuting $x_{i,j}$ to the left of $\vect{x}^M$, where \emph{step $(a,b)$} is defined to be the point in this process just before we commute $x_{i,j}$ past $x_{a,b}^{(M)_{a,b}}$. For a given $(a,b)$, let $M_0\in\matnonneg$ be equal to $M$ in all entries with coordinate less than $(a,b)$ and let $M_1=M-M_0$. Suppose we are at step $(a,b)$ and we have an expression of the form \begin{align*}\vect{x}^Mx_{i,j}&=q^\alpha \vect{x}^{M_0}x_{a,b}^{(M)_{a,b}}x_{i,j}\vect{x}^{M_1}+w,\end{align*} where $\alpha\in\mathbb{Z}$ and $w\in A^{(t)}$ is such that $\lt(w)\prec M+E_{i,j}$ and if $w\neq 0$, and $\vect{x}^K$ is a lex term of $w$, then $M$ and $K$ are equal in all entries northeast of $(i,j)$. We claim that there is such an expression for step $(a,b)^-$. Note that, once proven, repeated applications of this claim proves the inductive step, and hence the lemma.

If $x_{a,b}$ and $x_{i,j}$ $q^*$-commute, then the claim is trivial, so suppose $x_{a,b}x_{i,j}=x_{i,j}x_{a,b} +(q-q^{-1})x_{i,b}x_{a,j}$. Thus $b<j$ and, as is easily shown by induction on $(M)_{a,b}$, there is a $c\in\B{K}$ such that \begin{align*}x_{a,b}^{(M)_{a,b}}x_{i,j}&=x_{i,j}x_{a,b}^{(M)_{a,b}} +cx_{i,b}x_{a,b}^{(M)_{a,b}-1}x_{a,j}.\end{align*} From this we obtain  \begin{align*}q^\alpha \vect{x}^{M_0}x_{a,b}^{(M)_{a,b}}x_{i,j}\vect{x}^{M_1}+w &= q^\alpha \vect{x}^{M_0}x_{i,j}x_{a,b}^{(M)_{a,b}}\vect{x}^{M_1}  \\&+ cq^\alpha \vect{x}^{M_0}x_{i,b}x_{a,b}^{(M)_{a,b}-1}x_{a,j}\vect{x}^{M_1}+w. \end{align*}
Note that the claim is established if we can show that any lex term $\vect{x}^{K}$ of $\vect{x}^{M_0}x_{i,b}x_{a,b}^{(M)_{a,b}-1}x_{a,j}\vect{x}^{M_1}$ is such that $K$ equals $M$ northeast of $(i,j)$. 

As $M_1$ is zero in all entries with coordinates less than $(a,b)$,  there is a $\beta\in\mathbb{Z}$ with $x_{a,b}^{(M)_{a,b}-1}x_{a,j}\vect{x}^{M_1} = q^\beta \vect{x}^{M_1^\prime}$, where $M_1^\prime=M_1+\left((M)_{a,b}-1\right)E_{a,b}+E_{a,j}.$ Since $b<j$, we apply the induction hypothesis for $b$ to obtain \begin{align*} x_{i,b}\vect{x}^{M_1^\prime} &= q^\gamma\vect{x}^{M_1^\prime+E_{i,b}} - w^\prime,\end{align*} for some integer $\gamma$ and $w^\prime\in A^{(t)}$, where any lex term $\vect{x}^{K^\prime}$ of $w^\prime$ is such that  $K^\prime\prec M_1^\prime$ and $K^\prime$ equals $M_1^\prime$ in all entries northeast of $(i,b)$, and so in particular northeast of $(i,j)$. Moreover, since $K^\prime\prec M_1^\prime$, we know that $K^\prime$ can only be zero in all entries with coordinate less than $(a,b)$. For this reason, $\vect{x}^{M_0}\vect{x}^{K^\prime}=\vect{x}^{M_0+K^\prime}$ where $M_0+K^\prime$ is equal to $M$ in all entries northeast of $(i,j)$. As $M_1^\prime+E_{i,b}$ also equals $M$ in all entries northeast of $(i,j)$, we have established the claimed expression at step $(a,b)^-$. 

Finally, from the above procedure we also get $L\prec M+E_{i,j}$ where $\lt(w)=\vect{x}^L$. Furthermore, since the commutation relations are homogeneous with respect to the grading introduced at the end of Section~\ref{Rtsection}, we in fact have that $L\prec M+E_{i,j}$ at a coordinate northwest of $(i,j)$. 
\end{proof}

Lemma~\ref{strcor2lemma} roughly says that as we commute $x_{i,j}$ to the left of $\vect{x}^M$ and find the lexicographic expression of any new terms, one never needs to ``create or destroy'' any generator with coordinate northeast of $(i,j)$. 

\begin{proof}[Proof of Lemma~\ref{strcor2}, Part 2]

By applying Lemma~\ref{strcor2lemma} to the generators corresponding to $\vect{x}^{P_{\textnormal{id}}}$ in $\vect{x}^M$, we find that there is an integer $\alpha$ and a $w\in A^{(t)}$ such that \begin{align*} \vect{x}^M &= q^\alpha \vect{x}^{P_\textnormal{id}} \vect{x}^{M-P_\textnormal{id}} + w^\prime, \end{align*} where $w^\prime \in A^{(t)}$ and if $\lt(w^\prime)=
\vect{x}^K$, then $K\prec M$ at a coordinate northwest of $(i_k,j_k)$. On the other hand, notice that if $\sigma\in S_k$ with $\sigma\neq\textnormal{id}$, then \begin{align*} \vect{x}^{P_\sigma}\vect{x}^{M-P_\textnormal{id}} &= \vect{x}^{M-P_\textnormal{id}+P_\sigma} + w^{\prime\prime},\end{align*} where $\vect{x}^{M-P_\textnormal{id}+P_\sigma}$ is the leading term of the right-side and $M-P_\textnormal{id}+P_\sigma\prec M$ at a coordinate northwest of $(i_k,j_k)$. Our desired equation 
\begin{align*} \vect{x}^M &= q^\alpha [\ImidJ]^{(t)}\vect{x}^{M-P_\textnormal{id}} + w, \end{align*} follows for some integer $\alpha$ and $w\in A^{(t)}$ where, if $\lt(w)=\vect{x}^K$, then $K\prec M$ at a coordinate northwest of $(i_k,j_k)$.

\end{proof}

\subsection{Adding Derivations and Lexicographic Expressions} \label{addingderivationssection}

Throughout this section, we fix $t\in [mn], t\neq 1$ and let $(r,s)$ be the $t^\textnormal{th}$ smallest coordinate. Let $[x_{i,j}]$ be the matrix of generators for $A^{(t)}$, and $[y_{i,j}]$ the matrix of generators for $A^{(t-1)}$.

The proof of the main theorem requires a somewhat detailed understanding of the effect of the adding derivations map on the lexicographic expressions of an element $a\in A^{(t)}$ and its image $\overleftarrow{a}\in A^{(t-1)}[y_{r,s}^{-1}]$. This short section provides this information. 

Recall from Section~\ref{DDsection} that the adding derivations map is the homomorphism $$\overleftarrow{\cdot}:A^{(t)}\to A^{(t-1)}[y_{r,s}^{-1}]$$ defined on the standard generators by \begin{align*} 
\overleftarrow{x_{i,j}}=\begin{cases} y_{i,j}+ y_{i,s}y_{r,s}^{-1}y_{r,j}, & \textnormal{if $(i,j)$ is northwest of $(r,s)$;} \\
y_{i,j}, &\textnormal{ otherwise,}\end{cases}
\end{align*}
or, equivalently, by 
\begin{align*} 
\overleftarrow{x_{i,j}}=\begin{cases} y_{i,j}+ qy_{i,s}y_{r,j}y_{r,s}^{-1}, & \textnormal{if $(i,j)$ is northwest of $(r,s)$;} \\
y_{i,j}, &\textnormal{ otherwise.}\end{cases}
\end{align*}

Let $\vect{x}^M\in A^{(t)}$ and write \begin{align*} \vect{x}^M =  x_{i_1,j_1}x_{i_2,j_2}\cdots x_{i_p,j_p}, \end{align*} where for each $k\in [p-1]$, $(i_k,j_k)\leq (i_{k+1},j_{k+1})$. Let $\C{D}$ be the set of all $k$ such that $(i_k,j_k)$ is northwest of $(r,s)$. Then we may write, 
\begin{align}
\overleftarrow{\vect{x}^M} &= \sum_{C\subseteq \C{D}} q^{|C|} \stackrel{C}{\overleftarrow{x_{i_1,j_1}}} \stackrel{C}{\overleftarrow{x_{i_2,j_2}}} \cdots \stackrel{C}{\overleftarrow{x_{i_p,j_p}}},\label{Cequation}
\end{align}
where, for a $C\subseteq  \C{D}$, \begin{align*}  \stackrel{C}{\overleftarrow{x_{i_k,j_k}}} = \begin{cases} y_{i_k,s}y_{r,j_k}y_{r,s}^{-1}, & \textnormal{ if $k\in C$;} \\ y_{i_k,j_k}, & \textnormal{ if $k\not\in C$.}\end{cases}\end{align*}

\begin{lem}  \label{technicallemmacrit}
With notation as in the preceding discussion, let $z\in A^{(t-1)}[y_{r,s}^{-1}]$ be a summand on the right side of Equation~(\ref{Cequation}), so that for some $C\subseteq D$, \begin{align*} z =~~ \stackrel{C}{\overleftarrow{x_{i_1,j_1}}} \stackrel{C}{\overleftarrow{x_{i_2,j_2}}} \cdots \stackrel{C}{\overleftarrow{x_{i_p,j_p}}}.\end{align*} Then in the lexicographic expression of $z$, written as \begin{align*} z&=\sum_{L_C \in \mat} \alpha_{L_C} \vect{y}^{L_C}\end{align*} where $\alpha_{L_C}\in \B{K}^*$, the following hold.
\begin{enumerate}
\item For each $L_C$, $$(L_C)_{r,s} = (M)_{r,s} - |C|.$$
\item If $C\neq \emptyset$, then for every $L_C$, we have $L_C\prec M$ at the least $(i_k,j_k)$ for which $k \in C$.
\item For each $L_C$ and for each $i\in [m]\setminus r$, $$(L_C)_{i,s} = (M)_{i,s}+ |\{k\in C\mid i_k=i\}|.$$
\item If $(i,j)$ is northwest of $(r,s)$ and if $$(L_C)_{i,j}>(M)_{i,j}-|\{k\in C\mid (i_k,j_k)=(i,j)\}|,$$ then there is a coordinate $(i,j^\prime)$ with $1\leq j^\prime<j$ such that $$(L_C)_{i,j^\prime}<(M)_{i,j^\prime}-|\{k\in C\mid (i_k,j_k)=(i,j^\prime)\}|.$$
\item For each $L_C$, the entries in coordinates \emph{not} north, west or northwest of $(r,s)$ are equal to the corresponding entries in $M$.
\end{enumerate}
\end{lem}

\begin{proof}
First, let us split the summand $z$ by row indices, i.e., write $$z=(\stackrel{C}{\overleftarrow{x_{1,j_{1,1}}}}\stackrel{C}{\overleftarrow{x_{1,j_{1,2}}}}\cdots \stackrel{C}{\overleftarrow{x_{1,j_{1,p_1}}}})\cdots (\stackrel{C}{\overleftarrow{x_{m,j_{m,1}}}}\stackrel{C}{\overleftarrow{x_{m,j_{1,2}}}}\cdots \stackrel{C}{\overleftarrow{x_{m,j_{1,p_m}}}}),$$ where, for each $i\in [m]$, the generators appearing in the monomial $$\stackrel{C}{\overleftarrow{x_{i,j_{i,1}}}}\stackrel{C}{\overleftarrow{x_{i,j_{i,2}}}}\cdots \stackrel{C}{\overleftarrow{x_{i,j_{i,p_i}}}}$$ have indices $$(a,b) \in \{ (i,j)\mid j\in [n]\}\cup\{(r,j)\mid j\in [s]\}.$$ Moreover, if $y_{r,j}$ appears with $j\neq s$, then $y_{r,j}$ is to the right of any $y_{i,j^\prime}$ with $j^\prime<j$. In other words, such a $y_{r,j}$ $q^*$-commutes with every generator appearing to its right. Also, in $A^{(t-1)}$, we have that $y_{r,s}$ actually $q^*$-commutes with every generator of $A^{(t-1)}$. Thus $y_{r,s}^{-1}$ $q^*$-commutes with every generator in $A^{(t-1)}[y_{r,s}^{-1}]$ and we may write  $$\stackrel{C}{\overleftarrow{x_{i,j_{i,1}}}}\stackrel{C}{\overleftarrow{x_{i,j_{i,2}}}}\cdots \stackrel{C}{\overleftarrow{x_{i,j_{i,p_1}}}} = q^\alpha\vect{y}^{M_i}\vect{y}^{R_i}y_{r,s}^{-\beta},$$ where $\alpha\in\B{Z}$, $\beta$ is the number of occurrences of $y_{r,s}^{-1}$ in the left monomial, $M_i\in\matnonneg$ is the matrix defined by $$(M_i)_{a,b} = 
\begin{cases} 0 & \textnormal{ if $a\neq i$;}\\ 
 (M)_{i,b}-|\{k\in C\mid (i_k,j_k)=(i,b)\}| & \textnormal{ if $a=i$ and $1\leq b<s$;}\\
 (M)_{i,s}+|\{k\in C\mid i_k=i\}| & \textnormal{ if $a=i$ and $b=s$;}\\
 (M)_{i,b} & \textnormal{ if $s<b\leq n$,}
 \end{cases}$$
 and $R_i$ is a matrix whose nonzero entries appear only in coordinates between $(r,1)$ and $(r,s-1)$.

%
%
It follows that we may write \begin{align} z=q^{\alpha^\prime}\vect{y}^{M_1}\vect{y}^{R_1}\vect{y}^{M_2}\vect{y}^{R_2}\cdots \vect{y}^{M_{r-1}}\vect{y}^{R_{r-1}}\vect{y}^{R_r}y_{r,s}^{-|C|}\vect{y}^L,\label{bigeqn}\end{align} for some $\alpha^\prime\in\B{Z}$, where the entries of $R_r$ equal those of $M$ at coordinates between $(r,1)$ and $(r,s-1)$ and are zero elsewhere, and where entries of $L$ equal those of $M$ at all coordinates greater than $(r,s)$. 

Next, let $y_{r,j}$ be a generator with $1\leq j<s$, and consider $y_{r,j}\vect{y}^{M_i}$ for some $1\leq i<r$. Recall that, for $j^\prime<j$, we have the relation \begin{align*} y_{r,j}y_{i,j^\prime} &= y_{i,j^\prime}y_{r,j} - (q-q^{-1})y_{i,j}y_{r,j^\prime}. \end{align*} Repeated applications of this relation imply that $$y_{r,j}\vect{y}^{M_i} = \vect{y}^{M_i}y_{r,j} + \sum_\ell \alpha_\ell \vect{y}^{M_i^\ell}\vect{y}^{R^\ell},$$ for nonzero scalars $\alpha_\ell$ and where:
\begin{enumerate}
\item Every $M_i^\ell \in \matnonneg$ satisfies $M_i^\ell\prec M_i$, and the entries of each $M_i^\ell$ differ from those in $M_i$ only between coordinates $(i,1)$ and $(i,s-1)$;
\item Each $R^\ell\in\matnonneg$ has nonzero entries only between coordinates $(r,1)$ and $(r,s-1)$.
\end{enumerate}
In particular, when finding the lexicographic expression of the monomial $z$ written in the form of Equation~\eqref{bigeqn}, we never create or destroy any of the generators $y_{i,s}$, $y_{r,s}^{\pm 1}$, nor any generator with coordinates \emph{not} north, west or northwest. Parts 1,3 and 5 of the lemma follow. It also follows that for every $L_C$ and $i\in [r-1]$, if the entries in $L_C$ and $M$ with coordinates between $(i,1)$ and $(i,s-1)$ differ, then the first different entry is smaller in $L_C$. This implies Part 4. Finally, Part 2 comes from the fact that each term in the lexicographic expression of $z$ must start with $y_{i_1,j_1}\cdots y_{i_{k-1},j_{k-1}}$ since no subsequent relation produces a generator $y_{a,b}$ with $(a,b)<(i_k,j_k)$.

\end{proof}
\begin{cor} \label{leadingterm}
If $a\in A^{(t)}$, and $\lt (a)=\vect{x}^M$, then $\lt(\overleftarrow{a})=\vect{y}^M.$
\end{cor}

\begin{proof}
If $C\neq\emptyset$, then each term $\vect{y}^{L_C}$ in the resulting lexicographic expression satisfies $\vect{y}^{L_C}\prec \vect{y}^M$ by Part 2 of Lemma~\ref{technicallemmacrit}. On the other hand, $$\vect{y}^M=\,\,\,\stackrel{\emptyset}{\overleftarrow{x_{i_1,j_1}}} \stackrel{\emptyset}{\overleftarrow{x_{i_2,j_2}}} \cdots \stackrel{\emptyset}{\overleftarrow{x_{i_p,j_p}}}.$$
\end{proof}

\subsection{Generators of $\C{H}$-primes} \label{generatorsection}

We come to the main theorem of this paper. It is fairly straightforward to modify the proof and some of the above definitions to obtain the analogous result for left ideals. We remind the reader that an appendix to this paper provides an index of terms and notation used in the following proof.

\begin{thm} \label{genthm}
Fix the following data: A Cauchon diagram $B$; $t\in [mn]$; $(r,s)$ the $t^\textnormal{th}$ smallest coordinate; $[x_{i,j}]$ the matrix of generators for $A^{(t)}$; and the sequence of $\C{H}$-primes $(K_1,\ldots,K_{mn})$, where $$K_t=\ker\left(\sigma_B^{(t)}\right).$$ 

Let $G_t$ be the set of all $x_{i,j}$ with $(i,j)>(r,s)$ and $(i,j)\in B$, together with all quantum minors in $K_t$ whose maximum coordinate is at most $(r,s)$. Then $G_t$ is a Gr\"obner basis for $K_t$ as a right ideal. 
\end{thm}

\begin{proof}

First, note that $B=\emptyset$ if and only if $K_1=\langle 0\rangle$. On the other hand, in view of Theorem~\ref{ddtheorem}, we have $K_t=\langle 0\rangle$ for some $t\in[mn]$ if and only if $K_t=\langle 0\rangle$ for every $t\in[mn]$. Since the empty set generates $\langle 0\rangle$, we are done in the case $B=\emptyset$. From now on, we suppose $B\neq\emptyset$ and proceed by induction on $t$.

If $t=1$, then the only minor in $A^{(1)}=\qaffine$ whose maximum coordinate is $(1,1)$ is $$[1\,
|\,1]^{(1)}=t_{1,1}.$$ Since $t_{1,1}\in K_1$ if and only if $(1,1)\in B$, we see that $G_1$ is precisely the set of generators $t_{i,j}$ with $(i,j)\in B$. On the other hand, these $t_{i,j}$ generate $K_1$ by Theorem~\ref{grobner1} and so Proposition~\ref{LID} implies $G_1$ is indeed a Gr\"obner basis.

So now suppose $t\neq 1$ and that $G_{t-1}$ is a Gr\"obner basis for $K_{t-1}$. Let $[y_{i,j}]$ be the matrix of generators for $A^{(t-1)}$. There are two cases to consider, according to whether or not $(r,s)\in B$.

If $(r,s)\in B$, then, as elements of $\qtorus$, we have for each coordinate $(i,j)$ that $$\sigma_B^{(t)}(x_{i,j})=\sigma_B^{(t-1)}(y_{i,j}).$$ Therefore,  $$a=\sum_L \alpha_L\vect{x}^L\in K_t$$ if and only if $$a^\prime=\sum_L \alpha_L\vect{y}^L\in K_{t-1}.$$ Hence, if $\vect{y}^M$ divides $\lt(a^\prime)$, then $\vect{x}^M$ divides $\lt(a)$. 

Now, the previous paragraph also implies that if $[\ImidJ]^{(t-1)}\in K_{t-1}$ with maximum coordinate at most $(r,s)^-$, then $[\ImidJ]^{(t)}\in K_t$ with maximum coordinate strictly less than $(r,s)$ so that $[\ImidJ]^{(t)}\in G_t$. Also, if $(i,j)>(r,s)$ is such that $(i,j)\in B$, then $x_{i,j}\in K_t$. Finally, since $(r,s)\in B$, $$[r\,|\, s]^{(t)}=x_{r,s}\in K_t.$$ It now follows that since $G_{t-1}$ is a Gr\"obner basis for $K_{t-1}$, $G_t$ is a Gr\"obner basis\footnote{In general we have actually shown that a subset of $G_t$ is a Gr\"obner basis for $K_t$, but nothing is lost by adding the extra minors in $K_t$ with maximum coordinate equal to $(r,s)$.} for $K_t$.

Now assume $(r,s)\not\in B$, i.e., $x_{r,s}\not\in K_t$ and that $G_{t-1}$ is a Gr\"obner basis for $K_{t-1}$. In the following we aim to verify that $G_t$ satisfies Definition~\ref{grobnerdef} for $K_t$, but this requires some effort. The strategy we employ is as follows. Suppose a nonzero $a\in K_t$ is chosen such that $\lt(a)=\vect{x}^M$ is not divisible by the leading term of a member of $G_t$. Using the full power of the paths viewpoint developed above, we deduce in Claims $1$ and $2$ some structural properties of $M$. Using the information so obtained, we then find a term $\vect{y}^{N_C}\in A^{(t-1)}$ that \emph{is not} divisible by the leading term of any member of $G_{t-1}$ (Claim $3$) yet \emph{is} the leading term of an element of $K_{t-1}$ (Claims $4$ and $5$). Of course, these opposing properties contradict the induction hypothesis.

Fix a nonzero, monic $a\in K_t$ with lexicographic expression $$a=\vect{x}^M + \sum_L \alpha_L \vect{x}^L,$$ where $\lt(a)=\vect{x}^M$. Furthermore, we may assume that $a$ is homogeneous with respect to the grading introduced at the end of Section~\ref{Rtsection}, i.e., that for each $i\in[m]$, the $i^\textnormal{th}$ row sum of every $L$ and $M$ are equal, and for every $j\in [n]$, the $j^\textnormal{th}$ column sum of $M$ and every $L$ are equal.

If there exists an $(i,j)\in B$ with $(i,j)>(r,s)$ and $(M)_{i,j}\geq 1$, then $x_{i,j}\in G_t$ divides $\lt(a)$, and we are done. So we may assume no such $(i,j)$ exists. In fact by Lemma~\ref{splittingup} we may further assume that $M$ and every $L$ have the same values in each coordinate $(i,j)>(r,s)$, and, without loss of generality, that these entries are all zero, i.e., $(M)_{i,j}=0=(L)_{i,j}$ for all $(i,j)>(r,s)$.

Since $(r,s)\not\in B$, we have $$K_t=\overrightarrow{K_{t-1}}[x_{r,s}^{-1}]\cap A^{(t)},$$ and so there exists a $b\in K_{t-1}$ and a nonnegative integer $h$ with $$a=\overrightarrow{b}x_{r,s}^{-h}.$$ Then     $b=\overleftarrow{a}y_{r,s}^h,$ and by Corollary~\ref{leadingterm}, $$\lt(b)=\vect{y}^{M}y_{r,s}^h.$$ 

We henceforth call a minor in $G_{t-1}$ whose leading term divides $\lt(b)$ \emph{critical}. Note that since the maximum coordinate of a critical minor is at most $(r,s)^-$, its leading term actually divides $\vect{y}^M$. By induction, there exists at least one critical minor. Now, if $[\ImidJ]^{(t-1)}$ is critical and $[\ImidJ]^{(t)}\in K_t$, then, since the maximum coordinate of $[\ImidJ]^{(t)}$ is strictly less than $(r,s)$, we have have found an element of $G_t$ whose leading term divides $\lt(a)$, and we are done. \emph{From now on, we assume that if $[\ImidJ]^{(t-1)}$ is critical, then $[\ImidJ]^{(t)}\not\in K_t$.}
 
\begin{claim2} \label{critminorsclaim}

If $[\ImidJ]^{(t-1)}$ is critical, where $I=(i_1<i_2<\cdots <i_k)$ and $J=(j_1<j_2<\cdots < j_k)$, then we may assume the following.

\begin{enumerate} 
\item The set $\Gamma_B^{(t)}(\ImidJ)$ is nonempty and every vertex-disjoint path system in it contains a path with a $\textnormal{\reflectbox{L}}$-turn at $(r,s)$.
\item If $(i_{k^\prime},j_{k^\prime})$ is the largest diagonal coordinate northwest of $(r,s)$, then $$[i_1,\ldots,i_{k^\prime}\,|\,j_1,\ldots, j_{k^\prime}]^{(t-1)}$$ is critical.
\item If $(i_k,j_k)$ is northwest of $(r,s)$, then for every $(i,j)$ with $i_k<i\leq r$ and $j_k<j\leq s$, one has $(M)_{i,j} =0$.

\end{enumerate}
\end{claim2}

\noindent \emph{Proof of Claim~\ref{critminorsclaim}}:

\noindent \emph{Part 1}: This is simply restating the assumption preceding the claim, since otherwise there is a vertex-disjoint path system in $\Gamma_B^{(t-1)}(\ImidJ)$, i.e., $$[\ImidJ]^{(t-1)}\not\in K_{t-1}.$$

\noindent \emph{Part 2}: By Part 1, there exists a $\reflectbox{L}$-turn at $(r,s)$ in any vertex-disjoint path system in $\Gamma_B^{(t)}(\ImidJ)$. Hence $r\not\in I$ (in particular, $i_k<r$), $s\not\in J$ and at least $(i_1,j_1)$ is northwest of $(r,s)$. Therefore $(i_k,j_k)$ is either northwest or northeast of $(r,s)$.

If $(i_k,j_k)$ is northwest of $(r,s)$, then there is nothing to prove, so suppose $(i_k,j_k)$ is northeast of $(r,s)$. If $[I\setminus i_k\,|\, J\setminus j_k]^{(t-1)} \in K_{t-1}$, then replace $[\ImidJ]^{(t-1)}$ with $[I\setminus i_k\,|\, J\setminus j_k]^{(t-1)}$ and restart this argument. So assume that $(i_k,j_k)$ is northeast of $(r,s)$ and $[I\setminus i_k\,|\, J\setminus j_k]^{(t-1)} \not\in K_{t-1}$, i.e, there exists a vertex-disjoint path system $$\C{P}=(P_1,\ldots, P_{k-1})\in \Gamma_B^{(t-1)}(I\setminus i_k\,|\,J\setminus j_k).$$

Let $$\C{Q}=(Q_1,\ldots,Q_k)\in \Gamma_B^{(t)}(\ImidJ).$$ From Part 1, there exists a $Q_\alpha:i_\alpha\rightarrow j_\alpha$ containing $(r,s)$ as a $\reflectbox{L}$-turn. Clearly, we must have $\alpha=k^\prime$, and $k^\prime \neq k$ since $(i_k,j_k)$ is northeast of $(r,s)$. Recalling Corollary~\ref{newvdps}, consider the vertex-disjoint path system $$\C{R}=U(\C{P},\C{Q}\setminus Q_k)\in \Gamma_B^{(t-1)}(I\setminus i_k\,|\,J\setminus j_k)$$ See Figure~\ref{part2figure}. Since $P_{k^\prime}$ does not contain a $\reflectbox{L}$-turn at $(r,s)$, the path $U(P_{k^\prime},Q_{k^\prime})$ does not contain a $\reflectbox{L}$-turn at $(r,s)$. Moreover, by Corollary~\ref{disjointUcor}, $\C{R}$ is disjoint from $Q_k$. Hence, $\C{R}\cup Q_k$ is a vertex-disjoint path system in the empty set $\Gamma_B^{(t-1)}(\ImidJ)$, an impossibility.
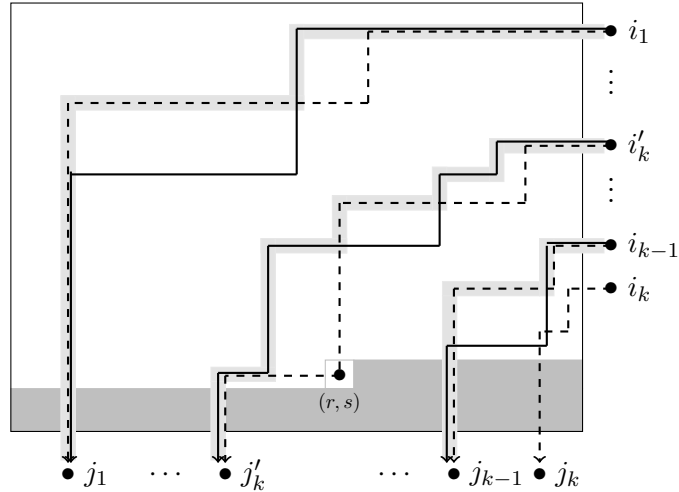
\begin{figure}[htbp]
\begin{center}

\begin{tikzpicture}[xscale=1.9, yscale=1.9]
\draw[color=lightgray, fill=lightgray] (0,0) rectangle (4,0.3);
\draw[color=lightgray, fill=lightgray] (2.4,0.3) rectangle (4,0.5);
\draw[color=lightgray] (2.2,0.3) rectangle (2.4,0.5);
\draw (0,0) rectangle (4,3);

\node at (2.3, 0.39) {$\bullet$};
\node[scale=0.7] at (2.3, 0.2) {$(r,s)$};

\node at (4.2,1) {$\bullet$};  
\node at (4.4,1) {$i_k$};
\node at (3.7,-0.3) {$\bullet$};
\node at (3.9,-0.3) {$j_k$};

\draw [thick, dashed] (4.2,1)--(3.9,1);
\draw [thick, dashed] (3.9,1)--(3.9,0.7);
\draw [thick, dashed] (3.9,0.7)--(3.7,0.7);
\draw [->,thick, dashed] (3.7,0.7)--(3.7,-0.22);

\node at (3.1,-0.3) {$\bullet$};
\node at (3.4,-0.3) {$j_{k-1}$};
\node at (2.7,-0.3) {$\cdots$};

\node at (4.2,1.3) {$\bullet$};
\node at (4.5, 1.3) {$i_{k-1}$};
\draw[color={rgb:black,1;white,10}, fill={rgb:black,1;white,8}] (4.15, 1.25) rectangle (3.7,1.35);
\draw[color={rgb:black,1;white,10}, fill={rgb:black,1;white,8}] (3.8, 0.95) rectangle (3.7,1.35);
\draw[color={rgb:black,1;white,10}, fill={rgb:black,1;white,8}] (3.02, 0.95) rectangle (3.7,1.05);
\draw[color={rgb:black,1;white,10}, fill={rgb:black,1;white,8}] (3.02, 0.95) rectangle (3.12,-0.2);

\draw [thick, dashed] (4.2,1.3)--(3.8,1.3);
\draw [thick, dashed] (3.8,1.3)--(3.8,1);
\draw [thick, dashed] (3.8,1)--(3.1,1);
\draw [->,thick, dashed] (3.1,1)--(3.1,-0.22);
\draw [thick, black] (4.2,1.32)--(3.75,1.32);
\draw [thick, black] (3.75,1.3)--(3.75,0.6);
\draw [thick,black] (3.75,0.6)--(3.05,0.6);
\draw [->,thick, black] (3.05,0.6)--(3.05,-0.22);

\node at (4.2,1.75) {$\vdots$};

\node at (4.2,2) {$\bullet$};
\node at (4.4,2) {$i_k^\prime$};

\draw[color={rgb:black,1;white,10}, fill={rgb:black,1;white,8}] (4.15,2.05) rectangle (3.35,1.95);
\draw[color={rgb:black,1;white,10}, fill={rgb:black,1;white,8}] (3.35, 1.75) rectangle (3.45,2.05);
\draw[color={rgb:black,1;white,10}, fill={rgb:black,1;white,8}] (3.35, 1.75) rectangle (2.95,1.85);
\draw[color={rgb:black,1;white,10}, fill={rgb:black,1;white,8}] (2.95, 1.55) rectangle (3.05,1.85);
\draw[color={rgb:black,1;white,10}, fill={rgb:black,1;white,8}] (2.25, 1.55) rectangle (2.95,1.65);
\draw[color={rgb:black,1;white,10}, fill={rgb:black,1;white,8}] (2.25, 1.25) rectangle (2.35,1.65);
\draw[color={rgb:black,1;white,10}, fill={rgb:black,1;white,8}] (1.75, 1.25) rectangle (2.35,1.35);
\draw[color={rgb:black,1;white,10}, fill={rgb:black,1;white,8}] (1.75, 0.35) rectangle (1.85,1.35);
\draw[color={rgb:black,1;white,10}, fill={rgb:black,1;white,8}] (1.4, 0.35) rectangle (1.85,0.45);
\draw[color={rgb:black,1;white,10}, fill={rgb:black,1;white,8}] (1.4, 0.35) rectangle (1.5,-0.2);

\draw [thick, dashed] (4.2,2)--(3.6,2);
\draw [thick, dashed] (3.6,2)--(3.6,1.6);
\draw [thick, dashed] (3.6,1.6)--(2.3,1.6);
\draw [thick, dashed] (2.3,1.6)--(2.3,0.39);
\draw [thick, dashed] (2.3,0.39)--(1.5,0.39);
\draw [->, thick, dashed] (1.5,0.39)--(1.5,-0.22);
\draw [thick, black] (4.2,2.03)--(3.4,2.03);
\draw [thick, black] (3.4,2.03)--(3.4,1.8);
\draw [thick, black] (3.4,1.8)--(3,1.8);
\draw [thick, black] (3,1.8)--(3,1.3);
\draw [thick, black] (3,1.3)--(1.8,1.3);
\draw [thick, black] (1.8,1.3)--(1.8,0.41);
\draw [thick, black] (1.8,0.41)--(1.45,0.41);
\draw [->,thick, black] (1.45,0.41)--(1.45,-0.22);

\node at (4.2,2.5) {$\vdots$};
\node at (4.2,2.8) {$\bullet$};
\node at (4.4,2.8) {$i_1$};

\draw[color={rgb:black,1;white,10}, fill={rgb:black,1;white,8}] (2,2.75) rectangle (4.15,2.85);
\draw[color={rgb:black,1;white,10}, fill={rgb:black,1;white,8}] (1.95,2.3) rectangle (2.05,2.85);
\draw[color={rgb:black,1;white,10}, fill={rgb:black,1;white,8}] (0.35,2.25) rectangle (2.05,2.35);
\draw[color={rgb:black,1;white,10}, fill={rgb:black,1;white,8}] (0.35,-0.2) rectangle (0.45,2.35);
\draw [thick, dashed] (2.5,2.8)--(2.5,2.3);
\draw [thick, dashed] (4.2,2.8)--(2.5,2.8);
\draw [thick, dashed] (2.5,2.3)--(0.4,2.3);
\draw [->,thick, dashed] (0.4,2.3)--(0.4,-0.22);
\draw [thick, black] (4.2,2.82)--(2,2.82);
\draw [thick, black] (2,2.82)--(2,1.8);
\draw [thick, black] (2,1.8)--(0.42,1.8);
\draw [->,thick, black] (0.42,1.82)--(0.42,-0.22);

\node at (0.4,-0.3) {$\bullet$};
\node at (0.6,-0.3) {$j_1$};

\node at (1.5,-0.3) {$\bullet$};
\node at (1.7,-0.3) {$j_k^\prime$};

\node at (1.1,-0.3) {$\cdots$};

\end{tikzpicture}
\caption{Illustration of the idea used to prove Part 2 of Claim 1. The dashed paths represent $\mathcal{Q}\in\Gamma^{(t)}_B(\ImidJ)$. The solid paths represent $\mathcal{P}\in\Gamma_B^{(t-1)}(I\setminus i_k\,|\,J\setminus j_k)$. The shaded paths represent $U(\mathcal{P},\mathcal{Q}\setminus Q_k)$.}
\label{part2figure}
\end{center}
\end{figure}

\noindent \emph{Part 3}: 
If $(i,j)=(r,s)$ and $(M)_{r,s}\geq 1$, then $[I\cup r\,|\, J\cup s]^{(t)}$ is a minor whose leading term divides $\vect{x}^M$ with maximum coordinate $(r,s)$. The only path in $\Gamma_B^{(t)}(r,s)$ is $(r,(r,s),s)$. Hence, if $\Gamma_B^{(t)}(I\cup r\,|\, J\cup s)$ is nonempty, then any path system in this set would have a sub-path system from $I$ to $J$ not using $(r,s)$. But this is a vertex-disjoint path system in the empty set $\Gamma_B^{(t-1)}(\ImidJ)$, an impossibility. Thus, $[I\cup r\,|\, J\cup s]^{(t)}\in G_t$ with leading term dividing $\vect{x}^M=\lt(a)$, and there is nothing left to prove. So we may assume $(M)_{r,s}=0$.

If $(i,j)\neq (r,s)$ but $(M)_{i,j}\geq 1$, then the leading term of $[I\cup i\,|\, J\cup j]^{(t-1)}$ divides $\vect{y}^M$. Since $[\ImidJ]^{(t-1)}\in K_{t-1}$,  there is no vertex-disjoint path system in $\Gamma_B^{(t-1)}(\ImidJ)$ and so certainly no vertex-disjoint path system in $\Gamma_B^{(t-1)}(I\cup i\,|\,J\cup j)$. Thus, $[I\cup i \,|\, J\cup j]^{(t-1)}$ is critical and so there exists a $\C{P} \in \Gamma_B^{(t)}(I\cup i\,|\, J\cup j)$. By Part 1 and vertex-disjointness, the path $P: i \to j \in \C{P}$ is necessarily the path with a $\reflectbox{L}$-turn at $(r,s)$. But then $\C{P}\setminus\{P\}$ is a vertex-disjoint path system in the empty set $\Gamma_B^{(t-1)}(\ImidJ)$, an impossibility. This completes the proof of Claim 1.


We now say that a coordinate $(i,j)$ is \emph{critical} if $(i,j)$ is northwest of $(r,s)$ and there exists a critical minor with $(i,j)$ as its maximum coordinate. 

\begin{claim2} \label{criticalclaim}
If $(i,j)$ critical, then every $(i,j^\prime)$ for $j<j^\prime<s$ with $(M)_{i,j^\prime}\geq 1$ is critical, and every $(i^\prime, j)$ for $i<i^\prime<r$ with $(M)_{i^\prime,j}\geq 1$ is critical.
\end{claim2}

\noindent \emph{Proof of Claim~\ref{criticalclaim}}: 
Suppose $[\ImidJ]^{(t-1)}$ is a critical minor whose maximum coordinate is $(i,j)$. Notice that the leading term of $$[I\,|\,J\setminus j \cup j^\prime]^{(t-1)}$$ divides $\vect{y}^M$ and its maximum coordinate is $(i,j^\prime)$, so it remains to show that this minor is in $K_{t-1}$.

Since $[\ImidJ]^{(t-1)}$ is critical, we may consider the supremum $\C{P}\in \Gamma_B^{(t)}(\ImidJ)\neq\emptyset$, which, by Part 1 of Claim 1, contains a path $P:i\to j$ with a $\reflectbox{L}$-turn at $(r,s)$. Notice that $P$ must have a horizontal subpath from $(r,s)$ to $(r,j)$, followed by a $\Gamma$-turn at $(r,j)$, and then vertically down to the column vertex $j$. In particular, $(r,j)$ is a white vertex. See Figure~\ref{claim2figure}.
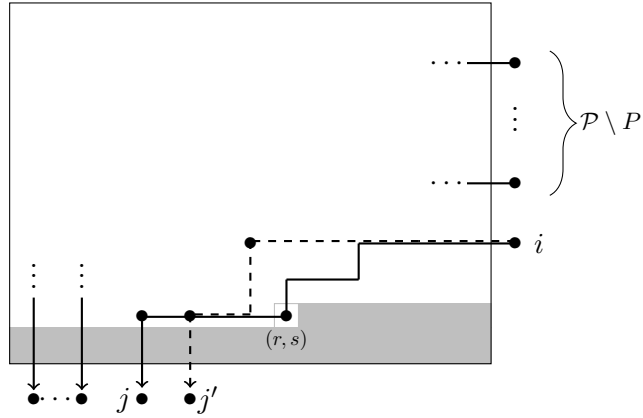
\begin{figure}[htbp]
\begin{center}

\begin{tikzpicture}[xscale=1.6, yscale=1.6]
\draw[color=lightgray, fill=lightgray] (0,0) rectangle (4,0.3);
\draw[color=lightgray, fill=lightgray] (2.4,0.3) rectangle (4,0.5);
\draw[color=lightgray] (2.2,0.3) rectangle (2.4,0.5);
\draw (0,0) rectangle (4,3);

\node at (2.3, 0.39) {$\bullet$};
\node[scale=0.7] at (2.3, 0.2) {$(r,s)$};

\node at (4.2,1) {$\bullet$};  
\node at (4.4,1) {$i$};
\node at (4.2, 1.5) {$\bullet$};  
\draw [thick, black] (4.2,1.5)--(3.8,1.5);
\node at (3.65, 1.5) {$\cdots$};  

\node at (4.2, 2.1) {$\vdots$};  
\node at (4.2, 2.5) {$\bullet$};  
\node at (3.65, 2.5) {$\cdots$}; 
\draw [thick, black] (4.2,2.5)--(3.8,2.5);
\usetikzlibrary{decorations.pathreplacing}
\draw [decorate,decoration={brace,amplitude=10pt,mirror,raise=4pt},yshift=0pt]
(4.4,1.4) -- (4.4,2.6) node [black,midway,xshift=0.95cm] {\footnotesize
$\C{P}\setminus P$};

\draw [thick, black] (4.2,1)--(2.9,1);
\draw [thick, black] (2.9,1)--(2.9,0.7);
\draw [thick, black] (2.9,0.7)--(2.3,0.7);
\draw [thick, black] (2.3,0.7)--(2.3,0.39);
\draw [thick, black] (2.3,0.39)--(1.1,0.39);
\draw [->, thick, black] (1.1,0.39)--(1.1,-0.2);

\draw [thick, dashed] (4.2,1.02)--(3.4,1.02);
\draw [thick, dashed] (3.4,1.02)--(2,1.02);
\draw [thick, dashed] (2,1.02)--(2,0.41);
\draw [thick, dashed] (2,0.41)--(1.5,0.41);

\node at (2,1) {$\bullet$};
\draw [->, thick, dashed] (1.5,0.41)--(1.5,-0.2);

\node at (1.1,-0.3) {$\bullet$};
\node at (0.95,-0.3) {$j$};
\node at (1.1,0.39) {$\bullet$};

\node at (1.5,-0.3) {$\bullet$};
\node at (1.65,-0.3) {$j^\prime$};
\node at (1.5,0.39) {$\bullet$};

\node at (0.2,-0.3) {$\bullet$}; 
\draw [<-, thick, black] (0.2,-0.22)--(0.2,0.55);
\node at (0.2,0.8) {$\vdots$};
 \node at (0.4,-0.3) {$\cdots$}; 
\node at (0.6,-0.3) {$\bullet$}; 
\node at (0.6,0.8) {$\vdots$};
\draw [<-, thick, black] (0.6,-0.22)--(0.6,0.55);

\end{tikzpicture}
\caption{Illustration of the idea used in proving Claim 2.  In the notation of that proof, the dashed line represents $Q$ and the solid line represents $P$. The other vertices and partial paths represent $\C{P}\setminus P= \C{Q}\setminus Q$. }
\label{claim2figure}
\end{center}
\end{figure}

Suppose that $[I\,|\,J\setminus j \cup j^\prime]^{(t-1)}\not\in K_{t-1}$, i.e., there exists a vertex-disjoint path system $\C{Q}$ from $I$ to $J\setminus j\cup j^\prime$ in $\Gamma_B^{(t-1)}(I\,|\,J\setminus j\cup j^\prime)$. Therefore, the path $Q: i\to j^\prime$ in $\C{Q}$ does \emph{not} use vertex $(r,s)$. By considering the appropriate supremums, we may assume without loss of generality that $\C{Q}\setminus Q = \C{P}\setminus P.$ Now, since $j^\prime>j$, $Q$ must intersect $P$ in order to end at $j^\prime$. Since $Q$ cannot have a \reflectbox{L}-turn at a $(r,s)$ or any larger vertex, the Cauchon condition implies that $(r,j^\prime)$ is a white vertex. On the other hand, $\C{P}\setminus P$ is disjoint from both $Q$ and $P$. If we let $R$ be the path starting at $i$, equal to $Q$ up to $(r,j^\prime)$, then equal to $P$ until the column vertex $j$, then $R$ is a path from $i$ to $j$ that does not contain $(r,s)$. Now $(\C{P}\setminus P)\cup R$ is a vertex-disjoint path system in $\Gamma_B^{(t-1)}(\ImidJ)$, a contradiction. That a coordinate $(i^\prime, j)$ with $i<i^\prime<r$ with $(M)_{i^\prime,j}\geq 1$ is critical is proven similarly. This completes the proof of Claim 2.\\

To summarize the discussion so far, we have shown that it suffices to assume the following. 
\begin{itemize}
\item If $[\ImidJ]^{(t-1)}$ is a critical minor, then $\Gamma_B^{(t)}(\ImidJ)\neq \emptyset$ and every vertex-disjoint path system contains a path with a \reflectbox{L}-turn at $(r,s)$ (by Part 1 of Claim 1).
\item Every critical minor contains a critical coordinate (by Part 2 of Claim 1).
\item For each critical coordinate $(i,j)$, there is a critical minor whose maximum coordinate is $(i,j)$ (by definition).
\item For each critical coordinate $(i,j)$ (of which there exists at least one), $(M)_{k,\ell}=0$ for all $i<k\leq r$ and $j<\ell\leq s$ (by Part 3 of Claim 1). In particular, no critical coordinate is northwest of another critical coordinate and so any critical minor contains a \emph{unique} critical coordinate. See Figure~\ref{Mstructure}.
\item If $(i,j)$ is northwest of $(r,s)$ and $(i,j)$ is \emph{not} a critical coordinate, then no coordinate above or to its left is critical (by Claim 2). 
\end{itemize}

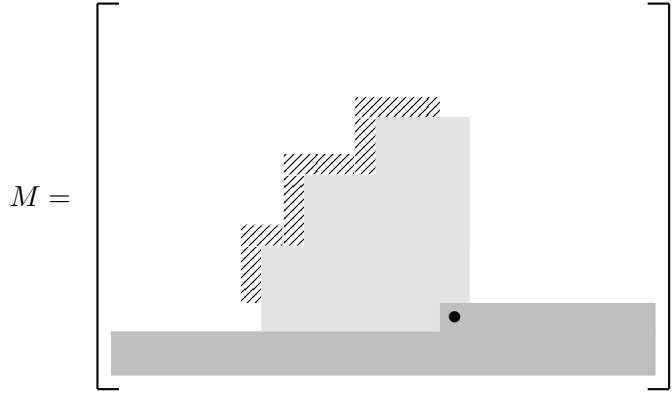
\begin{figure}[htbp]
\begin{center}
\begin{tikzpicture}[xscale=1.9, yscale=1.9]

\draw [thick, black] (0,0)--(0,2.7);
\draw [thick, black] (0,0)--(0.15,0);
\draw [thick, black] (0,2.7)--(0.15,2.7);

\draw [thick, black] (4,0)--(4,2.7);
\draw [thick, black] (4,0)--(3.85,0);
\draw [thick, black] (4,2.7)--(3.85,2.7);

\node at (2.5,0.5) {$\bullet$};  

\draw[color=lightgray, fill=lightgray] (0.1,0.1) rectangle (3.9,0.4);
\draw[color=lightgray, fill=lightgray] (2.4,0.4) rectangle (3.9,0.6);

\node at (2.5,0.5) {$\bullet$}; 
\node at (-0.4,1.35) {$M=$}; 
\usetikzlibrary{patterns};
\draw [color=white, step=0.5cm, pattern=north east lines] (1,0.6) rectangle (1.15,1);
\draw [color=white, step=0.5cm, pattern=north east lines] (1,1) rectangle (1.3,1.15);
\draw [color=white, step=0.5cm, pattern=north east lines] (1.3,1) rectangle (1.45,1.5);
\draw [color=white, step=0.5cm, pattern=north east lines] (1.3,1.5) rectangle (1.8,1.65);
\draw [color=white, step=0.5cm, pattern=north east lines] (1.8,1.5) rectangle (1.95,1.9);
\draw [color=white, step=0.5cm, pattern=north east lines] (1.8,1.9) rectangle (2.4,2.05);
\draw[color={rgb:black,1;white,8}, fill={rgb:black,1;white,8}] (1.15, 0.4) rectangle (2.4,1);
\draw[color={rgb:black,1;white,8}, fill={rgb:black,1;white,8}] (2.4, 0.6) rectangle (2.6,1.9);
\draw[color={rgb:black,1;white,8}, fill={rgb:black,1;white,8}] (1.45, 1) rectangle (2.6,1.5);
\draw[color={rgb:black,1;white,8}, fill={rgb:black,1;white,8}] (1.95, 1.5) rectangle (2.6,1.9);
\draw[color=lightgray, fill=lightgray] (0.1,0.1) rectangle (3.9,0.4);
\draw[color=lightgray, fill=lightgray] (2.4,0.4) rectangle (3.9,0.6);

\node at (2.5,0.5) {$\bullet$};

\end{tikzpicture}
\caption{Structure of $M$: The bullet represents coordinate $(r,s)$. All critical coordinates lie in the striped region. All entries in the two regions shaded solid gray are $0$.}
\label{Mstructure}
\end{center}
\end{figure}

The remainder of this proof will show that the above list of assumptions leads to a contradiction to the induction hypothesis.

Recalling the notation in Section~\ref{addingderivationssection}, let $$\lt(a)=\vect{x}^M=x_{i_1,j_1}x_{i_2,j_2}\cdots x_{i_p,j_p}, $$ and set $$C=\{k\in[p]\mid \textnormal{$(i_k,j_k)$ is critical}\}, $$ where $C$ is nonempty (since, by induction, there exists at least one critical minor, which in turn contains a critical coordinate). 
Consider the monomial \begin{align*} \stackrel{C}{\overleftarrow{x_{i_1,j_1}}} \stackrel{C}{\overleftarrow{x_{i_2,j_2}}} \cdots \stackrel{C}{\overleftarrow{x_{i_p,j_p}}}y_{r,s}^h.
\end{align*}
By the assumptions just established, Lemma~\ref{technicallemmacrit} and Proposition~\ref{straightening}, the lexicographic expression of this monomial equals
\begin{align}
q^\alpha \vect{y}^{N_C} + \sum_{L_C\in\mat} \alpha_{L_C} \vect{y}^{L_C}, \label{NCdefinition} 
 \end{align}
for some integer $\alpha$ and with every $L_C\prec N_C$ where
 \begin{align*} 
 (N_C)_{i,j} &= \begin{cases} 
 0, & \textnormal{ if $(i,j)$ is critical;}\\
 (M)_{i,j}, & \textnormal{ if $i\neq r$, $j\neq s$ and $(i,j)$ not critical;} \\
  (M)_{i,s}+\sum_{j^\prime} (M)_{i,j^\prime}, & \textnormal{ if $i\neq r$ and $j=s$;}\\
   (M)_{r,j}+\sum_{i^\prime} (M)_{i^\prime,j}, & \textnormal{ if $i=r$ and $j\neq s$;} \\
      h-|C|, & \textnormal{ if $i=r$ and $j=s$,} 
      \end{cases}
\end{align*}
and where the sum in the case that $i\neq r$ and $j=s$ is over all $j^\prime$ with $(i,j^\prime)$ critical, and the sum in the case that $i=r$ and $j\neq s$ is over all $i^\prime$ with $(i^\prime,j)$ critical. With respect to Figure~\ref{Mstructure}, the entries in the striped region are 0 in $N_C$, while entries above $(r,s)$ (respectively to the left of $(r,s)$) may become nonzero if there is a critical coordinate to the left (respectively above).

\begin{claim2} \label{notleadingtermclaim}
The term $\vect{y}^{N_C}$ is not divisible by the leading term of any element of $G_{t-1}$. Consequently, $\vect{y}^{N_C}$ is not the leading term of any element of $K_{t-1}$.
\end{claim2}
\noindent \emph{Proof of Claim~\ref{notleadingtermclaim}:} 
To the contrary, suppose that $\vect{y}^{N_C}$ is divisible by the leading term of some element in $G_{t-1}$. Since $(N_C)_{i,j}=(M)_{i,j}=0$ for every $(i,j)\geq (r,s)$, this element is a minor $$[\ImidJ]^{(t-1)},$$ where, say, $$I=(i_1<\cdots<i_z)  \textnormal{\hspace{0.2cm} and\hspace{0.2cm}   }  J=(j_1<\cdots<j_z).$$

Now, $[\ImidJ]^{(t-1)}$ does not contain a critical coordinate since $(N_C)_{i,j}=0$ for all critical coordinates $(i,j)$. Moreover, we may in this way conclude that $\vect{y}^M$ is not divisible by the leading term of $[\ImidJ]^{(t-1)}$. By the structure of the entries of $N_C$ compared to $M$, we then must have that $[\ImidJ]^{(t-1)}$ contains a coordinate $(i_k,j_k)$ in which $(N_C)_{i_k,j_k}>0$ while $(M)_{i,j}=0$, and so there are only two possibilities: either $(i_k,j_k)=(i_k,s)$ where $(i_k,j_k^\prime)$ is critical for some $j_k^\prime$, or $(i_k,j_k)=(r,j_k)$ where $(i_k^\prime,j_k)$ is critical for some $i_k^\prime$. We here show that the former possibility leads to a contradiction. The latter case is dealt with similarly.

Before we begin, we simplify our presentation slightly by further assuming that $(i_k,j_k)=(i_k,s)$ is the maximum coordinate of $[\ImidJ]^{(t-1)}$, i.e., that $z=k$. The general case is obtained by simply adding in $i_{k+1},\ldots,i_z$ and $j_{k+1},\ldots, j_z$ to the respective index sets of every minor we consider below.

As $\vect{y}^M$ is divisible by the leading term of $[I\setminus i_k\,|\, J\setminus s]^{(t-1)}$ (a minor with no critical coordinate), we have $[I\setminus i_k\,|\, J\setminus s]^{(t-1)}\not\in K_{t-1}$. So it is well-defined to set $$\tilde{\C{Q}}=(\tilde{Q}_1,\tilde{Q}_2,\ldots, \tilde{Q}_{k-1})$$ to be the supremum and $$\C{Q}=(Q_1,Q_2,\ldots, Q_{k-1})$$ to be the infimum of $\Gamma_B^{(t-1)}(I\setminus i_k\,|\, J\setminus s)$. 

Since $(i_k,j_k^\prime)$ is critical for some $j_k^\prime$, there exists, by Claim 1, a critical quantum minor $[I^\prime\,|\,J^\prime]^{(t-1)}$ where, for a (possibly nonpositive) integer $\alpha$, we write $$I^\prime=(i^\prime_\alpha<i^\prime_{\alpha+1}<\cdots <i^\prime_k=i_k) \textnormal{\hspace{0.2cm} and\hspace{0.2cm}   }J^\prime=(j^\prime_\alpha<j^\prime_{\alpha+1}<\cdots<j_k^\prime).$$ Set $$\tilde{\C{P}}=(\tilde{P}_\alpha,\ldots,\tilde{P}_k)$$ to be the supremum and $$\C{P}=(P_\alpha,\ldots,P_k)$$ to be the infimum of $\Gamma_B^{(t)}(I^\prime\,|\, J^\prime)$. By Claim 1, $P_k$ is a path from $i_k^\prime$ to $j_k^\prime$ in which $(r,s)$ is a \reflectbox{L}-turn.

The constructions to follow will show that if $\alpha\leq 1$, then we can construct a vertex-disjoint path system $$\C{R}_1\in \Gamma_B^{(t-1)}(\ImidJ),$$ or, if $\alpha>1$, a vertex-disjoint path system $$\C{R}^\prime_\alpha\in \Gamma_B^{(t-1)}(I^\prime\,|\, J^\prime).$$ As both $\Gamma_B^{(t-1)}(\ImidJ)$ and $\Gamma_B^{(t-1)}(I^\prime\,|\, J^\prime)$ were assumed to be empty sets, either case will establish a contradiction and so complete the proof of Claim~\ref{notleadingtermclaim}. The construction is fairly intricate so we first give an indication on how we plan to proceed. For $\ell\in[k]$, let $I_\ell=(i_\ell<\cdots<i_k)$ and $J_\ell=(j_\ell<\cdots j_k)$. Define $I^\prime_\ell$ and $J^\prime_\ell$ for $\alpha\leq\ell\leq k$ similarly. The first step is to build a vertex-disjoint path system $\C{R}_k\in\Gamma_B^{(t-1)}(I_k\,|\,J_k)$ using $\C{Q}$. If $k=1$, then we are done. Otherwise, we use $\C{R}_k$ to build $\C{R}^\prime_k\in\Gamma_B^{(t-1)}(I^\prime_k\,|\,J^\prime_k)$. Again, if $\alpha=k$, then we are done. Now suppose we have found $\C{R}_{\ell+1}\in\Gamma_B^{(t-1)}(I_{\ell+1}\,|\,J_{\ell+1})$ and $\C{R}^\prime_{\ell+1}\in \Gamma_B^{(t-1)}(I_{\ell+1}\,|\,J_{\ell+1})$ and that $\ell+1>\max(1,\alpha)$. We will show how to construct $\C{R}_\ell \in \Gamma_B^{(t-1)}(I_{\ell}\,|\,J_{\ell})$ using $\C{R}_{\ell+1}$ and $\C{R}^\prime_{\ell+1}$. If $\ell=1$ we are done. Otherwise, we construct $\C{R}_\ell^\prime \in \Gamma_B^{(t-1)}(I^\prime_\ell\,|\,J^\prime_\ell)$ using $\C{R}^\prime_{\ell+1}$ and the just constructed $\C{R}_\ell$. If $\ell=\alpha$ we are then done, otherwise we repeat the above, eventually ending with the desired vertex-disjoint path systems.

Now we give the promised details of the previous paragraph, beginning with the construction $\C{R}_k$. Recall that $P_k\in \C{P}$ has a subpath starting at row vertex $i^\prime_k=i_k$ and ending at vertex $(r,s)$. Define $Q_k$ to be this subpath followed by the vertical path from $(r,s)$ to column vertex $s$. For the purposes of the construction, set $v_k^0 = i_k$, $v_k^1=(r,s)$, and note that $v_k^0$ is the first vertex that $P_k$ and $Q_k$ have in common, while $v_k^1$ is the last vertex they have in common. If one sets $R_k=Q_k$, then note that we (trivially) have: $R_k=Q_k$ from $i_k$ to $v_k^0$; $R_k =U(P_k,Q_k)$ from $v_k^0$ to $v_k^1$; and $R_k=Q_k$ from $v_k^1$ to $j_k=s$.  See Figure~\ref{claim4figure1}.

\begin{figure}[htbp]
\begin{center}

\begin{tikzpicture}[xscale=1.9, yscale=1.9]
\draw[color=lightgray, fill=lightgray] (0,0) rectangle (4,0.3);
\draw[color=lightgray, fill=lightgray] (2.4,0.3) rectangle (4,0.5);
\draw[color=lightgray, fill=lightgray] (2.2,0.3) rectangle (2.4,0.5);
\draw (0,0) rectangle (4,3);

\node at (2.3, 0.39) {$\bullet$};
\node[scale=0.7] at (2.7, 0.39) {$(r,s)=v_k^1$};

\node at (4.2,1) {$\bullet$};  
\node at (4.9,1) {$i_k=i_k^\prime=v_k^0$};

\draw [thick, black] (4.2,1)--(2.9,1);
\draw [thick, black] (2.9,1)--(2.9,0.7);
\draw [thick, black] (2.9,0.7)--(2.3,0.7);
\draw [thick, black] (2.3,0.7)--(2.3,0.39);
\draw [thick, black] (2.3,0.39)--(1.1,0.39);
\draw [->, thick, black] (1.1,0.39)--(1.1,-0.2);

\draw [thick, dashed] (4.2,1.02)--(2.88,1.02);
\draw [thick, dashed] (2.88,1.02)--(2.88,0.72);
\draw [thick, dashed] (2.88,0.72)--(2.28,0.72);
\draw [thick, dashed,->] (2.28,0.72)--(2.28,-0.2);

\node at (1.1,-0.3) {$\bullet$};
\node at (1.3,-0.3) {$j_k^\prime$};
\node at (1.1,0.39) {$\bullet$};

\node at (2.28,-0.3) {$\bullet$};
\node at (2.66,-0.3) {$s=j_k$};

\end{tikzpicture}
\caption{Construction of $Q_k$ (dashed) from $P_k$ (solid) in the proof of Claim~\ref{notleadingtermclaim}.}
\label{claim4figure1}
\end{center}
\end{figure}
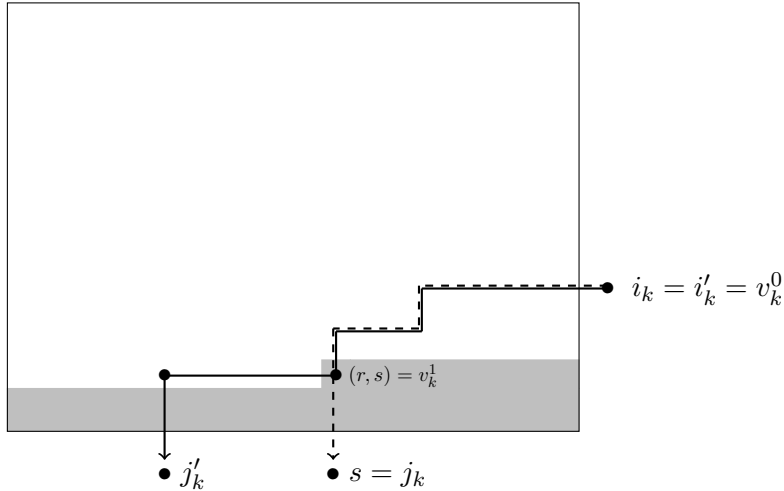

Set $\C{R}_k = (R_k)$. Of course, $\C{R}_k$ is a vertex-disjoint path system from $i_k$ to $j_k$ in $\Gamma_B^{(t-1)}(I_k\,|\,J_k)$. If $k=1$, then we are done, so we may assume $k>1$. 

In order to construct $\C{R}_k^\prime$, we first need to prove that $j_{k-1}\geq j_{k}^\prime$. To the contrary, suppose $j_{k-1}<j_k^\prime$, and consider $$[I \,|\, J\setminus s\cup j_k^\prime]^{(t-1)}.$$ If $[I \,|\, J\setminus s\cup j_k^\prime]^{(t-1)}\in K_{t-1}$, then it is critical and so there exists a vertex-disjoint path system from $I$ to $J\setminus s\cup j_k^\prime$ with the path from $i_k$ to $j_k^\prime$ containing a $\reflectbox{L}$-turn at $(r,s)$. But just as in the construction of $Q_k$ above, we may replace this path with a path from $i_k$ to $s$, thereby producing a vertex-disjoint path system from $I$ to $J$ in the empty set $\Gamma_B^{(t-1)}(\ImidJ)$, which is absurd. Next, suppose $[I \,|\, J\setminus s\cup j_k^\prime]^{(t-1)}\not\in K_{t-1}$, so that there does exist a vertex-disjoint path system from $I$ to $J\setminus s\cup j_k^\prime$ where the path $Q^\prime: i_k\to j_k^\prime$ does not contain a \reflectbox{L}-turn at $(r,s)$. We may take this path system to be $$(\tilde{Q}_1,\ldots,\tilde{Q}_{k-1}, Q^\prime).$$  

Now $\tilde{Q}_{k-1}$ is disjoint from $Q^\prime$, and so disjoint from $L(Q^\prime,P_k)$ by the lemma that is analogous to Lemma~\ref{disjointUcor}. But this latter path contains $(r,s)$ (since $P_k$ does) and so we may replace $Q^\prime$ with a path from $i_k$ to $s$, thereby again impossibly producing a vertex-disjoint path system in the empty set $\Gamma_B^{(t-1)}(\ImidJ)$. We can therefore conclude that $j_{k-1}\geq j_{k}^\prime$.

As $k>1$, consider $Q_{k-1}$, which, in particular, does not contain $(r,s)$. Now, $Q_{k-1}$ must intersect $Q_k$  at a vertex coming before $(r,s)$ on $Q_k$, as otherwise  $\C{Q}\cup Q_k\in \Gamma_B^{(t-1)}(\ImidJ)$. Let $w_k^0$ be the first such common vertex. On the other hand, since $j_{k-1}\geq j_k^\prime$ and $Q_{k-1}$ goes above $(r,s)$, $Q_{k-1}$ must also share with $P_k$ at least one vertex after $(r,s)$. Let $w_k^1$ be the last vertex that $Q_{k-1}$ and $P_k$ share. See Figure~\ref{claim4figure2}.

\begin{figure}[htbp]
\begin{center}

\begin{tikzpicture}[xscale=1.9, yscale=1.9]
\draw[color=lightgray, fill=lightgray] (0,0) rectangle (4,0.3);
\draw[color=lightgray, fill=lightgray] (2.4,0.3) rectangle (4,0.5);
\draw[color=lightgray, fill=lightgray] (2.2,0.3) rectangle (2.4,0.5);
\draw (0,0) rectangle (4,3);

\node at (2.3, 0.39) {$\bullet$};
\node[scale=0.85] at (2.8, 0.39) {$(r,s)=v_k^1$};

\node at (4.2,1) {$\bullet$};  
\node at (4.9,1) {$i_k=i_k^\prime=v_k^0$};

\node at (4.2, 2) {$\bullet$};
\node at (4.5, 2) {$i_{k-1}$};
\draw[color={rgb:black,1;white,10}, fill={rgb:black,1;white,8}] (4.15,0.95) rectangle (1.95,1.1);
\draw[color={rgb:black,1;white,10}, fill={rgb:black,1;white,8}] (1.95,1.1) rectangle (2.05,0.35);
\draw[color={rgb:black,1;white,10}, fill={rgb:black,1;white,8}] (1.05,0.45) rectangle (2.05,0.35);
\draw[color={rgb:black,1;white,10}, fill={rgb:black,1;white,8}] (1.05,0.45) rectangle (1.15,-0.2);

\draw [thick, dashed] (4.2,2)--(3.5,2);
\draw [thick, dashed] (3.5,2)--(3.5,1.5);
\draw [thick, dashed] (3.5,1.5)--(3.3,1.5);
\draw [thick, dashed] (3.3,1.5)--(3.3,1.04);
\node[scale=1.2] at (3.3,1) {$\bullet$};
\node[scale=0.85] at (3.3,0.8) {$w_{k}^0$};
\draw [thick, dashed] (3.3,1.04)--(2,1.04);
\draw [thick, dashed] (2,1.04)--(2,0.41);
\node[scale=1.2] at (1.6, 0.4) {$\bullet$};
\node[scale=0.85] at (1.6,0.6) {$w_{k}^1$};

\draw [thick, dashed] (2,0.41)--(1.6,0.41);
\draw [thick, dashed,->] (1.6,0.41)--(1.6,-0.2);
\node at (1.6,-0.3) {$\bullet$};
\node at (1.9,-0.3) {$j_{k-1}^\prime$};

\draw [thick, black] (4.2,1)--(2.9,1);
\draw [thick, black] (2.9,1)--(2.9,0.7);
\draw [thick, black] (2.9,0.7)--(2.3,0.7);
\draw [thick, black] (2.3,0.7)--(2.3,0.39);
\draw [thick, black] (2.3,0.39)--(1.1,0.39);
\draw [->, thick, black] (1.1,0.39)--(1.1,-0.2);


\node at (1.1,-0.3) {$\bullet$};
\node at (1.3,-0.3) {$j_k^\prime$};
\node at (1.1,0.39) {$\bullet$};


\end{tikzpicture}
\caption{$Q_{k-1}$ is the dashed path, $P_{k}$ is the solid path, $R_k^\prime$ is the shadowed path.}
\label{claim4figure2}
\end{center}
\end{figure}
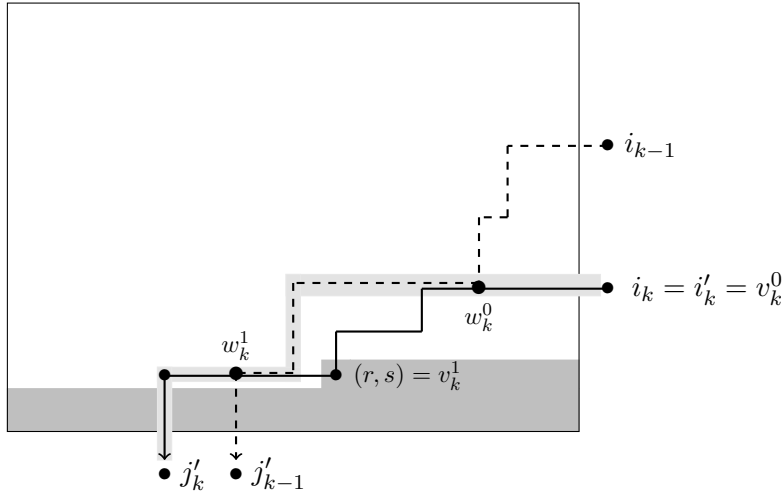

Define $R_k^\prime$ to be the path that equals $P_k$ from $i_k^\prime$ to $w_k^0$, then equals $U(Q_{k-1},P_k)$ from $w_k^0$ to $w_k^1$, and then equals $P_k$ from $w_k^1$ to $j_k^\prime$. Observe that $R_k^\prime$ does not contain $(r,s)$, so that $$\C{R}_k^\prime = (R_k^\prime)$$ is a vertex-disjoint path system in $\Gamma_B^{(t-1)}(I^\prime_k\,|\,J_k^\prime)$. If $k=\alpha$, then again we have obtained the desired contradiction, and so we may assume $\alpha<k$.

Now let $\ell$ be an integer with $\max(\alpha,1)\leq \ell <k$. Assume that $i_{\ell+1}\leq i_{\ell+1}^\prime$, $j_{\ell}\geq j_{\ell+1}^\prime$ and that we have the following data.
\begin{itemize} 

\item We have a $\C{R}_{\ell+1}=(R_{\ell+1},\ldots,R_k) \in \Gamma_B^{(t-1)}( I_{\ell+1}\,|\,J_{\ell+1})$. Moreover, there exists a vertex $v_{\ell+1}^0$ which is the first vertex that $P_{\ell+1}$ and $Q_{\ell+1}$ have in common, a vertex $v_{\ell+1}^1$ which is the last vertex that $P_{\ell+1}$ and $Q_{\ell+1}$ have in common, and $R_{\ell+1}$ equals $Q_{\ell+1}$ from $i_{\ell+1}$ to $v_{\ell+1}^0$, equals $U(P_{\ell+1},Q_{\ell+1})$ from $v_{\ell+1}^0$ to $v_{\ell+1}^1$, and equals $Q_{\ell+1}$ from $v_{\ell+1}^1$ to $j_{\ell+1}$. 

\item We have a $\C{R}^\prime_{\ell+1}=(R^\prime_{\ell+1},\ldots,R^\prime_k) \in \Gamma_B^{(t-1)}( I_{\ell+1}^\prime\,|\,J_{\ell+1}^\prime)$. Moreover, there exists a vertex $w_{\ell+1}^0$ which is the first vertex that $P_{\ell+1}$ and $Q_{\ell}$ have in common, a vertex $w_{\ell+1}^1$ which is the last vertex that $P_{\ell+1}$ and $Q_{\ell}$ have in common, and $R^\prime_{\ell+1}$ equals $P_{\ell+1}$ from $i_{\ell+1}^\prime$ to $w_{\ell+1}^0$, equals $U(P_{\ell+1},Q_{\ell})$ from $w_{\ell+1}^0$ to $w_{\ell+1}^1$, and equals $P_{\ell+1}$ from $w_{\ell+1}^1$ to $j^\prime_{\ell+1}$. 
\end{itemize}

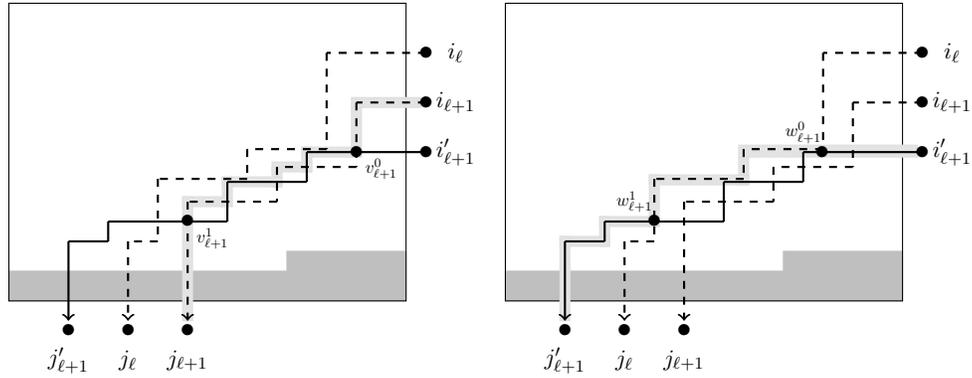
\begin{figure}[htbp]
\begin{center}

\begin{tikzpicture}[xscale=1.32, yscale=1.32]
\draw[color=lightgray, fill=lightgray] (0,0) rectangle (4,0.3);
\draw[color=lightgray, fill=lightgray] (2.8,0.3) rectangle (4,0.5);
\draw (0,0) rectangle (4,3);

\draw[color={rgb:black,1;white,7}, fill={rgb:black,1;white,8}] (3.45,1.95) rectangle (4.15,2.05);
\draw[color={rgb:black,1;white,7}, fill={rgb:black,1;white,8}] (3.45,1.45) rectangle (3.55,1.95);
\draw[color={rgb:black,1;white,7}, fill={rgb:black,1;white,8}] (2.95,1.45) rectangle (3.55,1.55);
\draw[color={rgb:black,1;white,7}, fill={rgb:black,1;white,8}] (2.95,1.3) rectangle (3.05,1.55);
\draw[color={rgb:black,1;white,7}, fill={rgb:black,1;white,8}] (2.65,1.3) rectangle (3.05,1.4);
\draw[color={rgb:black,1;white,7}, fill={rgb:black,1;white,8}] (2.65,1.15) rectangle (2.75,1.35);
\draw[color={rgb:black,1;white,7}, fill={rgb:black,1;white,8}] (2.15,1.15) rectangle (2.75,1.25);
\draw[color={rgb:black,1;white,7}, fill={rgb:black,1;white,8}] (2.15,0.95) rectangle (2.25,1.25);
\draw[color={rgb:black,1;white,7}, fill={rgb:black,1;white,8}] (1.75,0.95) rectangle (2.15,1.05);
\draw[color={rgb:black,1;white,7}, fill={rgb:black,1;white,8}] (1.75,-0.2) rectangle (1.85,1.05);

\node at (4.2, 1.5) {$\bullet$};
\node[scale=0.8] at (4.5, 1.5) {$i_{\ell+1}^\prime$};
\node at (4.2, 2) {$\bullet$};
\node[scale=0.8] at (4.5, 2) {$i_{\ell+1}$};
\node at (4.2, 2.5) {$\bullet$};
\node[scale=0.8] at (4.5, 2.5) {$i_{\ell}$};

\node at (0.6,-0.3) {$\bullet$};
\node[scale=0.8]  at (0.6,-0.6) {$j_{\ell+1}^\prime$};
\node at (1.2,-0.3) {$\bullet$};
\node[scale=0.8]  at (1.2,-0.6) {$j_{\ell}$};
\node at (1.8,-0.3) {$\bullet$};
\node[scale=0.8]  at (1.8,-0.6) {$j_{\ell+1}$};

\draw [thick, black] (4.2,1.5)--(3,1.5);
\draw [thick, black] (3,1.5)--(3,1.2);
\draw [thick, black] (3,1.2)--(2.2,1.2);
\draw [thick, black] (2.2,1.2)--(2.2,0.8);
\draw [thick, black] (2.2,0.8)--(1,0.8);
\draw [thick, black] (1,0.8)--(1,0.6);
\draw [thick, black] (1,0.6)--(0.6,0.6);
\draw [thick, black,->] (0.6,0.6)--(0.6,-0.2);

\draw [thick, dashed] (4.2,2)--(3.5,2);
\draw [thick, dashed] (3.5,2)--(3.5,1.35);
\draw [thick, dashed] (3.5,1.35)--(2.7,1.35);
\draw [thick, dashed] (2.7,1.35)--(2.7,1);
\draw [thick, dashed] (2.7,1)--(1.8,1);
\draw [thick, dashed,->] (1.8,1)--(1.8,-0.2);

\draw [thick, dashed] (4.2,2.5)--(3.2,2.5);
\draw [thick, dashed] (3.2,2.5)--(3.2,1.53);
\draw [thick, dashed] (3.2,1.53)--(2.4,1.53);
\draw [thick, dashed] (2.4,1.53)--(2.4,1.23);
\draw [thick, dashed] (2.4,1.23)--(1.5,1.23);
\draw [thick, dashed] (1.5,1.23)--(1.5,0.6);
\draw [thick, dashed] (1.5,0.6)--(1.2,0.6);
\draw [thick, dashed,->] (1.2,0.6)--(1.2,-0.2);

\node at (3.5,1.5) {$\bullet$};
\node[scale=0.6] at (3.75,1.33) {$v_{\ell+1}^0$};
\node at (1.8,0.8) {$\bullet$};
\node[scale=0.6] at (2.05,0.63) {$v_{\ell+1}^1$};

\draw[color=lightgray, fill=lightgray] (5,0) rectangle (9,0.3);
\draw[color=lightgray, fill=lightgray] (7.8,0.3) rectangle (9,0.5);
\draw (5,0) rectangle (9,3);

\draw[color={rgb:black,1;white,7}, fill={rgb:black,1;white,8}] (7.35,1.45) rectangle (9.15,1.57);
\draw[color={rgb:black,1;white,7}, fill={rgb:black,1;white,8}] (7.35,1.17) rectangle (7.45,1.45);
\draw[color={rgb:black,1;white,7}, fill={rgb:black,1;white,8}] (6.45,1.17) rectangle (7.45,1.27);
\draw[color={rgb:black,1;white,7}, fill={rgb:black,1;white,8}] (6.45,0.75) rectangle (6.55,1.27);
\draw[color={rgb:black,1;white,7}, fill={rgb:black,1;white,8}] (5.95,0.75) rectangle (6.55,0.85);
\draw[color={rgb:black,1;white,7}, fill={rgb:black,1;white,8}] (5.95,0.55) rectangle (6.05,0.85);
\draw[color={rgb:black,1;white,7}, fill={rgb:black,1;white,8}] (5.55,0.55) rectangle (6.05,0.65);
\draw[color={rgb:black,1;white,7}, fill={rgb:black,1;white,8}] (5.55,-0.2) rectangle (5.65,0.65);

\node at (9.2, 1.5) {$\bullet$};
\node[scale=0.8] at (9.5, 1.5) {$i_{\ell+1}^\prime$};
\node at (9.2, 2) {$\bullet$};
\node[scale=0.8] at (9.5, 2) {$i_{\ell+1}$};
\node at (9.2, 2.5) {$\bullet$};
\node[scale=0.8] at (9.5, 2.5) {$i_{\ell}$};

\node at (5.6,-0.3) {$\bullet$};
\node[scale=0.8]  at (5.6,-0.6) {$j_{\ell+1}^\prime$};
\node at (6.2,-0.3) {$\bullet$};
\node[scale=0.8]  at (6.2,-0.6) {$j_{\ell}$};
\node at (6.8,-0.3) {$\bullet$};
\node[scale=0.8]  at (6.8,-0.6) {$j_{\ell+1}$};

\draw [thick, black] (9.2,1.5)--(8,1.5);
\draw [thick, black] (8,1.5)--(8,1.2);
\draw [thick, black] (8,1.2)--(7.2,1.2);
\draw [thick, black] (7.2,1.2)--(7.2,0.8);
\draw [thick, black] (7.2,0.8)--(6,0.8);
\draw [thick, black] (6,0.8)--(6,0.6);
\draw [thick, black] (6,0.6)--(5.6,0.6);
\draw [thick, black,->] (5.6,0.6)--(5.6,-0.2);

\draw [thick, dashed] (9.2,2)--(8.5,2);
\draw [thick, dashed] (8.5,2)--(8.5,1.35);
\draw [thick, dashed] (8.5,1.35)--(7.7,1.35);
\draw [thick, dashed] (7.7,1.35)--(7.7,1);
\draw [thick, dashed] (7.7,1)--(6.8,1);
\draw [thick, dashed,->] (6.8,1)--(6.8,-0.2);

\draw [thick, dashed] (9.2,2.5)--(8.2,2.5);
\draw [thick, dashed] (8.2,2.5)--(8.2,1.53);
\draw [thick, dashed] (8.2,1.53)--(7.4,1.53);
\draw [thick, dashed] (7.4,1.53)--(7.4,1.23);
\draw [thick, dashed] (7.4,1.23)--(6.5,1.23);
\draw [thick, dashed] (6.5,1.23)--(6.5,0.6);
\draw [thick, dashed] (6.5,0.6)--(6.2,0.6);
\draw [thick, dashed,->] (6.2,0.6)--(6.2,-0.2);

\node at (8.19,1.5) {$\bullet$};
\node[scale=0.6] at (8,1.7) {$w_{\ell+1}^0$};
\node at (6.5,0.8) {$\bullet$};
\node[scale=0.6] at (6.3,1) {$w_{\ell+1}^1$};

\end{tikzpicture}
\caption{$R_{\ell+1}$ is shaded path on the left diagram; $R_{\ell+1}^\prime$ is shaded path on the right diagram.}
\label{claim4figure3}
\end{center}
\end{figure}

We will construct a path $R_\ell:i_\ell\to j_\ell$ disjoint from $R_{\ell+1}$, but first we need to show that $i_{\ell}\leq i_{\ell}^\prime$. Suppose that $i_{\ell}>  i_{\ell}^\prime$. Since $j_\ell\geq j_{\ell+1}^\prime > j_\ell^\prime$, we may consider the minor $$[I^{\prime\prime}\,|\,J^{\prime\prime}]^{(t-1)} = [i_\alpha^\prime,\ldots, i_\ell^\prime,i_{\ell},\ldots,i_{k-1}\,|\,j_\alpha^\prime,\ldots, j_\ell^\prime, j_{\ell},\ldots,j_{k-1}]^{(t-1)}.$$ Note that this minor does not contain a critical coordinate since $[\ImidJ]^{(t-1)}$ doesn't and $(i_k,j_k)$ is the unique critical coordinate in $[I^\prime\,|\,J^\prime]^{(t-1)}$. But as $\vect{y}^M$ is divisible by the leading term of $[I^{\prime\prime}\,|\,J^{\prime\prime}]^{(t-1)}$, we know that $[I^{\prime\prime}\,|\,J^{\prime\prime}]^{(t-1)} $ is not in $K_{t-1}$, i.e., $\Gamma_B^{(t-1)}(I^{\prime\prime}\,|\,J^{\prime\prime})$ is nonempty.

Indeed, $(\tilde{P}_1,\ldots,\tilde{P}_\ell, Q_\ell,\ldots, Q_{k-1})\in \Gamma_B^{(t-1)}(I^{\prime\prime}\,|\,J^{\prime\prime})$, since for any path system in $\Gamma_B^{(t-1)}(I^{\prime\prime}\,|\,J^{\prime\prime})$ we choose, the sub-path system from $\{i_1^\prime,\ldots,i_\ell^\prime\}$ to $\{j_1^\prime,\ldots,j_\ell^\prime\}$ may be replaced with the supremum of $$\Gamma_B^{(t-1)}(i_1^\prime,\ldots,i_\ell^\prime\,|\,j_1^\prime,\ldots,j_\ell^\prime),$$ and the sub-path system from $\{i_\ell,\ldots, i_{k-1}\}$ to $\{j_\ell,\ldots,j_{k-1}\}$ with the infimum of $$\Gamma_B^{(t-1)}(i_\ell,\ldots, i_{k-1}\,|\,j_\ell,\ldots,j_{k-1}).$$ These two sets are, of course, $(\tilde{P}_1,\ldots\tilde{P}_\ell)$ and $(Q_\ell,\ldots, Q_{k-1})$ respectively. In particular, this implies $\tilde{P}_\ell$ is disjoint from both $Q_\ell$. But $\tilde{P}_\ell$ is also disjoint from $P_{\ell+1}$. By the construction of $R_{\ell+1}^\prime$, it follows that $\tilde{P}_\ell$ and $R_{\ell+1}^\prime$ are also disjoint, so that $$\{\tilde{P}_1,\ldots,\tilde{P}_\ell\}\cup\C{R}_{\ell+1}^\prime$$ forms a vertex-disjoint path system in the empty set $\Gamma_B^{(t-1)}(\ImidJ)$. Since this is an impossibility, it must be the case that $i_{\ell}\leq i_{\ell}^\prime$.


Next, we construct $\C{R}_\ell$. Recall that $R_{\ell+1}^\prime$ has a first vertex $w_{\ell+1}^0$ that is common to $P_{\ell+1}$ and $Q_\ell$. On the other hand, since $P_\ell$ and $P_{\ell+1}$ are disjoint and $i_\ell \leq i^\prime_{\ell} < i_{\ell+1}^\prime$, it must be the case that $P_{\ell+1}$ intersects $Q_\ell$. Let $v_\ell^0$ be the first vertex they have in common and note that $v_\ell^0$ comes before $w_{\ell+1}^0$ on $Q_\ell$. See Figure~\ref{claim4figure4} for an example.

\begin{figure}[htbp]
\begin{center}

\begin{tikzpicture}[xscale=2.3, yscale=2.3]
\draw[color=lightgray, fill=lightgray] (0,0) rectangle (4,0.3);
\draw[color=lightgray, fill=lightgray] (2.8,0.3) rectangle (4,0.5);
\draw (0,0) rectangle (4,3);

\draw[color={rgb:black,1;white,7}, fill={rgb:black,1;white,8}] (3.45,1.95) rectangle (4.15,2.05);
\draw[color={rgb:black,1;white,7}, fill={rgb:black,1;white,8}] (3.45,1.45) rectangle (3.55,1.95);
\draw[color={rgb:black,1;white,7}, fill={rgb:black,1;white,8}] (2.95,1.45) rectangle (3.55,1.55);
\draw[color={rgb:black,1;white,7}, fill={rgb:black,1;white,8}] (2.95,1.3) rectangle (3.05,1.55);
\draw[color={rgb:black,1;white,7}, fill={rgb:black,1;white,8}] (2.65,1.3) rectangle (3.05,1.4);
\draw[color={rgb:black,1;white,7}, fill={rgb:black,1;white,8}] (2.65,1.15) rectangle (2.75,1.35);
\draw[color={rgb:black,1;white,7}, fill={rgb:black,1;white,8}] (2.15,1.15) rectangle (2.75,1.25);
\draw[color={rgb:black,1;white,7}, fill={rgb:black,1;white,8}] (2.15,0.95) rectangle (2.25,1.25);
\draw[color={rgb:black,1;white,7}, fill={rgb:black,1;white,8}] (1.75,0.95) rectangle (2.15,1.05);
\draw[color={rgb:black,1;white,7}, fill={rgb:black,1;white,8}] (1.75,-0.2) rectangle (1.85,1.05);

\draw[color={rgb:black,1;white,7}, fill={rgb:black,1;white,8}] (3.15,2.45) rectangle (4.15,2.55);
\draw[color={rgb:black,1;white,7}, fill={rgb:black,1;white,8}] (3.15,1.75) rectangle (3.25,2.55);
\draw[color={rgb:black,1;white,7}, fill={rgb:black,1;white,8}] (1.65,1.75) rectangle (3.25,1.85);
\draw[color={rgb:black,1;white,7}, fill={rgb:black,1;white,8}] (1.65,1.18) rectangle (1.75,1.85);
\draw[color={rgb:black,1;white,7}, fill={rgb:black,1;white,8}] (1.45,1.18) rectangle (1.75,1.28);
\draw[color={rgb:black,1;white,7}, fill={rgb:black,1;white,8}] (1.45,0.55) rectangle (1.55,1.28);
\draw[color={rgb:black,1;white,7}, fill={rgb:black,1;white,8}] (1.15,0.55) rectangle (1.55,0.65);
\draw[color={rgb:black,1;white,7}, fill={rgb:black,1;white,8}] (1.15,-0.2) rectangle (1.25,0.65);

\node at (4.2, 1.5) {$\bullet$};
\node[scale=0.8] at (4.4, 1.5) {$i_{\ell+1}^\prime$};
\node at (4.2, 2) {$\bullet$};
\node[scale=0.8] at (4.4, 2) {$i_{\ell+1}$};
\node at (4.2, 2.5) {$\bullet$};
\node[scale=0.8] at (4.4, 2.5) {$i_{\ell}$};
\node at (4.2,2.25) {$\bullet$};
\node[scale=0.8] at (4.4, 2.25) {$i_{\ell}^\prime$};

\node at (0.6,-0.3) {$\bullet$};
\node[scale=0.8]  at (0.6,-0.5) {$j_{\ell+1}^\prime$};
\node at (1.2,-0.3) {$\bullet$};
\node[scale=0.8]  at (1.2,-0.5) {$j_{\ell}$};
\node at (1.8,-0.3) {$\bullet$};
\node[scale=0.8]  at (1.8,-0.5) {$j_{\ell+1}$};
\node at (0.15,-0.3) {$\bullet$};
\node[scale=0.8]  at (0.15,-0.5) {$j_{\ell}^\prime$};

\draw [thick, black] (4.2,1.5)--(3,1.5);
\draw [thick, black] (3,1.5)--(3,1.2);
\draw [thick, black] (3,1.2)--(2.2,1.2);
\draw [thick, black] (2.2,1.2)--(2.2,0.8);
\draw [thick, black] (2.2,0.8)--(1,0.8);
\draw [thick, black] (1,0.8)--(1,0.6);
\draw [thick, black] (1,0.6)--(0.6,0.6);
\draw [thick, black,->] (0.6,0.6)--(0.6,-0.2);

\draw [thick, dashed] (4.2,2)--(3.5,2);
\draw [thick, dashed] (3.5,2)--(3.5,1.35);
\draw [thick, dashed] (3.5,1.35)--(2.7,1.35);
\draw [thick, dashed] (2.7,1.35)--(2.7,1);
\draw [thick, dashed] (2.7,1)--(1.8,1);
\draw [thick, dashed,->] (1.8,1)--(1.8,-0.2);

\draw [thick, dashed] (4.2,2.5)--(3.2,2.5);
\draw [thick, dashed] (3.2,2.5)--(3.2,1.53);
\draw [thick, dashed] (3.2,1.53)--(2.4,1.53);
\draw [thick, dashed] (2.4,1.53)--(2.4,1.23);
\draw [thick, dashed] (2.4,1.23)--(1.5,1.23);
\draw [thick, dashed] (1.5,1.23)--(1.5,0.6);
\draw [thick, dashed] (1.5,0.6)--(1.2,0.6);
\draw [thick, dashed,->] (1.2,0.6)--(1.2,-0.2);

\draw [thick, black] (4.2,2.25)--(3.6,2.25);
\draw [thick, black] (3.6,2.25)--(3.6,1.8);
\draw [thick, black] (3.6,1.8)--(1.7,1.8);
\draw [thick, black] (1.7,1.8)--(1.7,1.1);
\draw [thick, black] (1.7,1.1)--(0.15,1.1);
\draw [thick, black,->] (0.15,1.1)--(0.15,-0.2);

\node at (3.2,1.8) {$\bullet$};
\node[scale=0.7] at (3.1,1.9) {$v_{\ell}^0$};

\node at (1.5,1.1) {$\bullet$};
\node[scale=0.7] at (1.4,1.2) {$v_{\ell}^1$};

\node at (3.19,1.5) {$\bullet$};
\node[scale=0.7] at (3.04,1.65) {$w_{\ell+1}^0$};
\node at (1.5,0.8) {$\bullet$};
\node[scale=0.7] at (1.34,0.95) {$w_{\ell+1}^1$};

\end{tikzpicture}
\caption{Constructing $R_\ell$ (upper shaded path). Note that it is disjoint from $R_{\ell+1}$ (lower shaded path).}
\label{claim4figure4}
\end{center}
\end{figure}
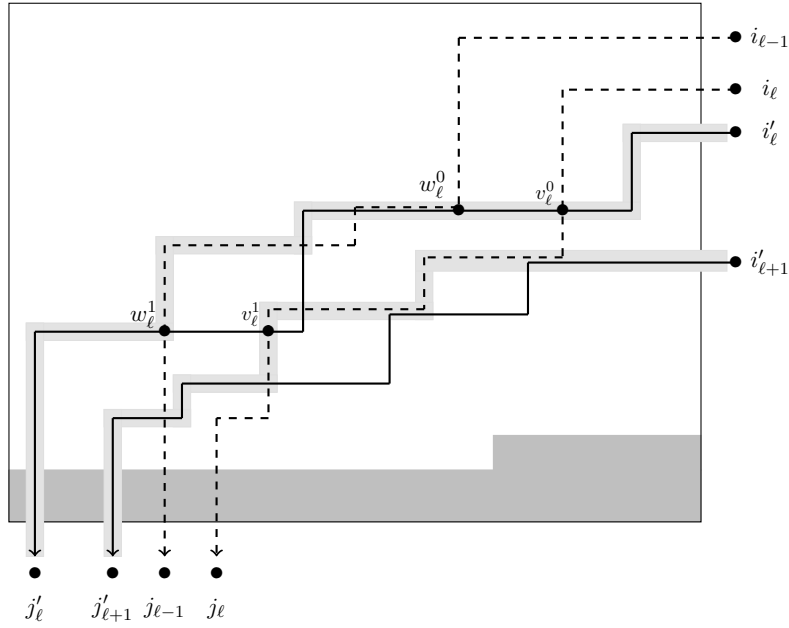
Next, observe that $P_\ell$ must also intersect $Q_\ell$ at a vertex coming after $w_{\ell+1}^0$. This is the case since otherwise, $P_\ell$ is disjoint from $R_{\ell+1}^\prime$ after $w_{\ell+1}^0$. But by the construction of $R_{\ell+1}^\prime$, we would then have $(P_1,\ldots, P_\ell)\cup\C{R}_\ell^\prime$, a vertex-disjoint path system in the empty set $\Gamma_B^{(t-1)}(I^\prime\,|\,J^\prime).$ So, let $v_{\ell}^1$ be the last vertex that $Q_\ell$ and $P_\ell$ have in common. Define $R_\ell$ as the path equal to $Q_\ell$ from $i_\ell$ to $v_\ell^0$, equal to $U(P_\ell,Q_\ell)$ from $v_\ell^0$ to $v_\ell^1$, and then equal to $Q_\ell$ from $v_{\ell}^1$ to $j_\ell$. Since $Q_\ell$ is disjoint from $Q_{\ell+1}$ up to $v_\ell^0$ and after $v_\ell^1$, and $U(P_\ell,Q_\ell)$ is disjoint from $U(P_{\ell+1},Q_{\ell+1})$, we see that $R_\ell$ is disjoint from $R_{\ell+1}$, and so $$\C{R}_\ell = \C{R}_{\ell+1}\cup R_\ell\in  \Gamma_B^{(t-1)}( i_{\ell},\ldots, i_k\,|\,j_{\ell},\ldots,j_k).$$ If $\ell=1$, then we have obtained the required path system completing the proof of this claim.

Assume $\ell>1$. To construct $\C{R}_\ell^\prime$, we first must show that $j_{\ell-1}\geq j^\prime_\ell$. To the contrary, suppose that $j_{\ell-1}<j^\prime_\ell$ Now, $i_{\ell-1}<i_\ell\leq i_\ell^\prime$, so we may consider the minor $$[I^{\prime\prime\prime}\,|\,J^{\prime\prime\prime}]^{(t-1)}= [i_1,\ldots,i_{\ell-1},i_\ell^\prime,\ldots, i_{k}^\prime\,|\,j_1,\ldots,j_{\ell-1},j_\ell^\prime,\ldots, j_{k}^\prime]^{(t-1)}.$$
Since $\vect{y}^M$ is divisible by the leading term of $[I^{\prime\prime\prime}\,|\,J^{\prime\prime\prime}]^{(t-1)}$, there are two possibilities. If $[I^{\prime\prime\prime}\,|\,J^{\prime\prime\prime}]^{(t-1)}$ is in $K_{t-1}$, then it is a critical minor, and so there is a vertex-disjoint path system in $$\Gamma_B^{(t)}(i_1,\ldots,i_{\ell-1},i_\ell^\prime,\ldots, i_{k}^\prime\,|\,j_1,\ldots,j_{\ell-1},j_\ell^\prime,\ldots, j_{k}^\prime),$$ which we may take to be $$(\tilde{Q}_1,\ldots,\tilde{Q}_{\ell-1}, P_\ell,\ldots,P_{k}).$$ Therefore, $\tilde{Q}_{\ell-1}$ is disjoint from both $P_\ell$ and $Q_\ell$, and so disjoint from $R_\ell$ by the latter path's construction. Hence, $(\tilde{Q}_1,\ldots,\tilde{Q}_{\ell-1})\cup\C{R}_\ell$ is a vertex-disjoint path system in the empty set $\Gamma_B^{(t-1)}(\ImidJ)$, an impossibility. The other possibility is that $[I^{\prime\prime\prime}\,|\,J^{\prime\prime\prime}]^{(t-1)}$ is not in $K_{t-1}$. This possibility is dealt with in a manner similar to the above when we justified the inequality $j_{k-1}\geq j_k^\prime$.  It follows that $j_{\ell-1}\geq j^\prime_\ell$. 

We now describe the construction of $R_\ell^\prime$. Since $\ell>1$, consider $Q_{\ell-1}$. This path is disjoint from $Q_\ell$. If $Q_{\ell-1}$ does not intersect $P_\ell$ at a vertex between $v_\ell^0$ and $v_\ell^1$, then $Q_{\ell-1}$ is disjoint from $R_\ell$ so that $(Q_1,\ldots,Q_{\ell-1})\cup \C{R}_\ell$ is a vertex-disjoint path system in the empty set $\Gamma_B^{(t-1)}(\ImidJ)$, an impossibility. So we may let $w_\ell^0$ be the first vertex that $Q_{\ell-1}$ shares with $P_\ell$. Now, since $j_\ell^\prime\leq j_{\ell-1}<j_\ell$, and the two subpaths of $P_\ell$ and $Q_\ell$ starting at $v^1_\ell$, together with the line from $j_\ell^\prime$ to $j_\ell$ is a closed curve in the plane, $Q_{\ell-1}$ must intersect $P_\ell$ at a vertex after $v_\ell^1$. Let $w_\ell^1$ be their last common vertex after $v_\ell^1$. We now take $R_\ell^\prime$ to be the path equal to $P_\ell$ from $i_\ell^\prime$ to $w_\ell^0$; equal to $U(P_\ell,Q_{\ell-1})$ from $w_\ell^0$ to $w_\ell^1$; and equal to $P_\ell$ from $w_\ell^1$ to $j_\ell^\prime$. See Figure~\ref{claim4figure5} for an example. That $R_{\ell}^\prime$ is disjoint from $R_{\ell+1}^\prime$ is seen similarly as when we showed that $R_\ell$ and $R_{\ell+1}$ are disjoint. 
\begin{figure}[htbp]
\begin{center}

\begin{tikzpicture}[xscale=2.3, yscale=2.3] 
\draw[color=lightgray, fill=lightgray] (5,0) rectangle (9,0.3);
\draw[color=lightgray, fill=lightgray] (7.8,0.3) rectangle (9,0.5);

\draw (5,0) rectangle (9,3);

\draw[color={rgb:black,1;white,7}, fill={rgb:black,1;white,8}] (7.35,1.45) rectangle (9.15,1.57);
\draw[color={rgb:black,1;white,7}, fill={rgb:black,1;white,8}] (7.35,1.17) rectangle (7.45,1.45);
\draw[color={rgb:black,1;white,7}, fill={rgb:black,1;white,8}] (6.45,1.17) rectangle (7.45,1.27);
\draw[color={rgb:black,1;white,7}, fill={rgb:black,1;white,8}] (6.45,0.75) rectangle (6.55,1.27);
\draw[color={rgb:black,1;white,7}, fill={rgb:black,1;white,8}] (5.95,0.75) rectangle (6.55,0.85);
\draw[color={rgb:black,1;white,7}, fill={rgb:black,1;white,8}] (5.95,0.55) rectangle (6.05,0.85);
\draw[color={rgb:black,1;white,7}, fill={rgb:black,1;white,8}] (5.55,0.55) rectangle (6.05,0.65);
\draw[color={rgb:black,1;white,7}, fill={rgb:black,1;white,8}] (5.55,-0.2) rectangle (5.65,0.65);
\draw[color={rgb:black,1;white,7}, fill={rgb:black,1;white,8}] (8.55,2.2) rectangle (9.15,2.3);
\draw[color={rgb:black,1;white,7}, fill={rgb:black,1;white,8}] (8.55,1.75) rectangle (8.65,2.3);
\draw[color={rgb:black,1;white,7}, fill={rgb:black,1;white,8}] (6.65,1.75) rectangle (8.65,1.85);
\draw[color={rgb:black,1;white,7}, fill={rgb:black,1;white,8}] (6.65,1.55) rectangle (6.75,1.85);
\draw[color={rgb:black,1;white,7}, fill={rgb:black,1;white,8}] (5.85,1.55) rectangle (6.75,1.65);
\draw[color={rgb:black,1;white,7}, fill={rgb:black,1;white,8}] (5.85,1.05) rectangle (5.95,1.65);
\draw[color={rgb:black,1;white,7}, fill={rgb:black,1;white,8}] (5.1,1.05) rectangle (5.95,1.15);
\draw[color={rgb:black,1;white,7}, fill={rgb:black,1;white,8}] (5.1,-0.2) rectangle (5.2,1.15);

\node at (9.2, 1.5) {$\bullet$};
\node[scale=0.8] at (9.4, 1.5) {$i_{\ell+1}^\prime$};

\node at (9.2, 2.5) {$\bullet$};
\node[scale=0.8] at (9.4, 2.5) {$i_{\ell}$};
\node at (9.2, 2.8) {$\bullet$};
\node[scale=0.8] at (9.4, 2.8) {$i_{\ell-1}$};

\draw [thick, dashed] (9.2,2.8)--(7.6,2.8);
\draw [thick, dashed] (7.6,2.8)--(7.6,1.82);
\draw [thick, dashed] (7.6,1.82)--(7,1.82);
\draw [thick, dashed] (7,1.82)--(7,1.6);
\draw [thick, dashed] (7,1.6)--(5.9,1.6);
\draw [thick, dashed,->] (5.9,1.6)--(5.9,-0.2);

\node at (5.9,1.1) {$\bullet$};
\node[scale=0.8] at (5.78,1.2) {$w_\ell^1$};

\node at (7.6,1.8) {$\bullet$};
\node[scale=0.8] at (7.45,1.95) {$w_\ell^0$};

\node at (5.6,-0.3) {$\bullet$};
\node[scale=0.8]  at (5.6,-0.5) {$j_{\ell+1}^\prime$};
\node at (5.9,-0.3) {$\bullet$};
\node[scale=0.8]  at (5.9,-0.5) {$j_{\ell-1}$};
\node at (6.2,-0.3) {$\bullet$};
\node[scale=0.8]  at (6.2,-0.5) {$j_{\ell}$};

\draw [thick, black] (9.2,1.5)--(8,1.5);
\draw [thick, black] (8,1.5)--(8,1.2);
\draw [thick, black] (8,1.2)--(7.2,1.2);
\draw [thick, black] (7.2,1.2)--(7.2,0.8);
\draw [thick, black] (7.2,0.8)--(6,0.8);
\draw [thick, black] (6,0.8)--(6,0.6);
\draw [thick, black] (6,0.6)--(5.6,0.6);
\draw [thick, black,->] (5.6,0.6)--(5.6,-0.2);

\draw [thick, dashed] (9.2,2.5)--(8.2,2.5);
\draw [thick, dashed] (8.2,2.5)--(8.2,1.53);
\draw [thick, dashed] (8.2,1.53)--(7.4,1.53);
\draw [thick, dashed] (7.4,1.53)--(7.4,1.23);
\draw [thick, dashed] (7.4,1.23)--(6.5,1.23);
\draw [thick, dashed] (6.5,1.23)--(6.5,0.6);
\draw [thick, dashed] (6.5,0.6)--(6.2,0.6);
\draw [thick, dashed,->] (6.2,0.6)--(6.2,-0.2);

\node at (9.2,2.25) {$\bullet$};
\node[scale=0.8] at (9.4, 2.25) {$i_{\ell}^\prime$};
\node at (5.15,-0.3) {$\bullet$};
\node[scale=0.8]  at (5.15,-0.5) {$j_{\ell}^\prime$};

\draw [thick, black] (9.2,2.25)--(8.6,2.25);
\draw [thick, black] (8.6,2.25)--(8.6,1.8);
\draw [thick, black] (8.6,1.8)--(6.7,1.8);
\draw [thick, black] (6.7,1.8)--(6.7,1.1);
\draw [thick, black] (6.7,1.1)--(5.15,1.1);
\draw [thick, black,->] (5.15,1.1)--(5.15,-0.2);

\node at (8.2,1.8) {$\bullet$};
\node[scale=0.7] at (8.1,1.9) {$v_{\ell}^0$};

\node at (6.5,1.1) {$\bullet$};
\node[scale=0.7] at (6.4,1.2) {$v_{\ell}^1$};

\end{tikzpicture}
\caption{Constructing $R_\ell^\prime$ (upper shaded path). Note that it is disjoint from $R_{\ell+1}^\prime$ (lower shaded path).}
\label{claim4figure5}
\end{center}
\end{figure}

Of course, we now take $$\C{R}_\ell^\prime = \C{R}_{\ell+1}^\prime\cup R_{\ell}^\prime\in \Gamma_B^{(t-1)}(i_\ell^\prime,\ldots,i_k^\prime\,|\,j_\ell^\prime,\ldots,j_k^\prime).$$
If $\ell=\alpha$, then we are done. Otherwise continue as above. As this process ends when $\ell=\max(\alpha,1)$, we eventually construct a vertex-disjoint path system in either the empty set $\Gamma_B^{(t-1)}(\ImidJ)$ or the empty set $\Gamma_B^{(t-1)}(I^\prime\,|\, J^\prime)$. This contradiction completes the proof of Claim~\ref{notleadingtermclaim}.

\begin{claim2} \label{uniquenessclaim}
The term $\vect{y}^{N_C}$ from Expression~\ref{NCdefinition} is a lex term of $b=\overleftarrow{a}y_{r,s}^h$. 
\end{claim2}

\noindent \emph{Proof of Claim~\ref{uniquenessclaim}:} 
Recall that a lexicographic term is said to be a \emph{lex term} of an element of $A^{(t-1)}$ or $A^{(t)}$ if it has a nonzero coefficient in the lexicographic expression of that element.

We have already seen that $\vect{y}^{N_C}$ is a lex term of \begin{align*} \stackrel{C}{\overleftarrow{x_{i_1,j_1}}} \stackrel{C}{\overleftarrow{x_{i_2,j_2}}} \cdots \stackrel{C}{\overleftarrow{x_{i_p,j_p}}}y_{r,s}^h.
\end{align*}
We will show that this is, in fact, the unique appearance of $\vect{y}^{N_C}$ in (the lexicographic expression of) any summand of $$b = \overleftarrow{a}y_{r,s}^h = \overleftarrow{\vect{x}^M}y_{r,s}^h + \sum_L \alpha_L\overleftarrow{\vect{x}^L}y_{r,s}^h,$$ and so is a lex term of $b$.

To start, consider in $\overleftarrow{\vect{x}^M}y_{r,s}^h$ the lexicographic expression of some \begin{align*} \stackrel{C^\prime}{\overleftarrow{x_{i_1,j_1}}} \stackrel{C^\prime}{\overleftarrow{x_{i_2,j_2}}} \cdots \stackrel{C^\prime}{\overleftarrow{x_{i_p,j_p}}}y_{r,s}^h &= \sum_{L_{C^\prime} \in \mat} \alpha_{L_{C^\prime}} \vect{y}^{L_{C^\prime}},\end{align*} where $C^\prime\neq C$. Suppose $C^\prime$ is chosen so that there is an $L_{C^\prime}$ equal to $N_C$.

Now, by Lemma~\ref{technicallemmacrit}, each term $\vect{y}^{L_{C^\prime}}$ satisfies $(L_{C^\prime})_{r,s}=h-|C^\prime|$. Since $(N_C)_{r,s}=h-|C|$, we must have if $|C^\prime|=|C| >0$. But, since $C\neq C^\prime$, there must exist $k\in C^\prime$ such that $(i_k,j_k)$ is not a critical coordinate. Since $(i_k,j_k)$ is not critical, we should have \begin{align*} (L_{C^\prime})_{i_k,j_k}&=(N_C)_{i_k,j_k}\\&=(M)_{i_k,j_k}\\&>(M)_{i_k,j_k}-|\{k^\prime\in C^\prime\mid (i_{k^\prime},j_{k^\prime})=(i_k,j_k)\}|.\end{align*}
By Part 4 of Lemma~\ref{technicallemmacrit}, there is a coordinate $(i_k,j)$ with $j<j_k$ and \begin{align*} (L_{C^\prime})_{i_k,j} &<  (M)_{i_k,j} \\ &= (N_C)_{i_k,j},\end{align*} where the equality follows from the fact that since $(i_k,j_k)$ is not critical, neither is $(i_k,j)$ by Claim 2. Hence, $L_{C^\prime}$ cannot be equal to $N_C$ since their entries differ in coordinate $(i_k,j)$. This is a contradiction and so we conclude that $\vect{y}^{N_C}$ is a lex term of $\overleftarrow{\vect{x}^M}y_{r,s}^{h}.$

Next, suppose $$\vect{x}^{L}=x_{a_1,b_1}\cdots x_{a_t,b_t},$$ appears in $a$, where $(a_k,b_k)\leq (a_{k+1},b_{k+1})$ for each $k\in[t-1]$, and where $L\prec M$ at coordinate $(i,j)$. With the notation of Section~\ref{addingderivationssection}, consider $$\overleftarrow{\vect{x}^L}y_{r,s}^h = \sum_{D} q^{|D|} \stackrel{D}{\overleftarrow{x_{a_1,b_1}}} \stackrel{D}{\overleftarrow{x_{a_2,b_2}}} \cdots \stackrel{D}{\overleftarrow{x_{a_t,b_t}}}y_{r,s}^h.$$ Suppose that $\vect{y}^{N_C}$ appears in $$\stackrel{D}{\overleftarrow{x_{a_1,b_1}}} \stackrel{D}{\overleftarrow{x_{a_2,b_2}}} \cdots \stackrel{D}{\overleftarrow{x_{a_t,b_t}}}y_{r,s}^h = \sum_{L_D}\alpha_{L_D}\vect{y}^{L_D}.$$

By Lemma~\ref{technicallemmacrit}, Part 5, every entry in an $L_D$ with coordinates not northwest, north or west of $(r,s)$ must equal the corresponding entry in $L$. Since we also require $L_D=N_C$ for some $D$, this implies that those entries are equal to the corresponding entry in $M$ as well. Thus, $(i,j)$ can only be north, west or northwest of $(r,s)$. On the other hand, if $j=s$, then all entries in $L$ and $M$ in row $i$ except coordinate $(i,j)$ are equal. By homogeneity, this means that we must also have $(L)_{i,j}=(M)_{i,j}$, a contradiction. Hence $(i,j)$ is not north of $(r,s)$, and by similar reasoning $(i,j)$ is not west of $(r,s)$. Therefore, we may assume that $L\prec M$ at a coordinate $(i,j)$ northwest of $(r,s)$.

There are two cases to consider. First, suppose $(i,j)$ is not a critical coordinate. In this case, $$(N_C)_{i,j}=(M)_{i,j}>(L)_{i,j},$$ and so we may proceed as above by applying Part 4 of Lemma~\ref{technicallemmacrit} to see that in order to have $(L_D)_{i,j}=(N_C)_{i,j}$, we would require an entry with coordinate $(i,j^\prime)$ with $j^\prime<j$ to satisfy \begin{align*} (L_D)_{i,j^\prime}&<(L)_{i,j^\prime}\\&=(M)_{i,j^\prime}\\&=(N_C)_{i,j^\prime}.\end{align*} Hence we cannot have $N_C=L_D$ in this case.

%


Next, suppose $(i,j)$ is critical. Let $(i,j_0)$ be the least critical coordinate in row $i$. Notice that no $(i,j^\prime)=(a_k,b_k)$ with $j^\prime<j_0$ has $k\in D$, for reasons similar to the previous paragraph. Now, consider $j^\prime$ where $j_0< j^\prime\leq s$.  By Part 3 of Claim 1 applied to $(i,j_0)$, we know that every entry of $M$ south of $(i,j^\prime)$ is equal to zero. Hence, the sum of the entries in column $j^\prime$ of $M$ is equal to $\sum_{i^\prime=1}^i (M)_{i^\prime,j^\prime}$. By homogeneity, this is equal to the sum of the entries in column $j^\prime$ of $L$. On the other hand, the entries north of $(i,j^\prime)$ in $L$ are equal to the corresponding entries in $M$. Since all entries of $L$ are nonnegative, we see that $$(L)_{i,j^\prime}\leq (M)_{i,j^\prime},$$ for every $j_0<j^\prime\leq s$. Also, since the entries of $L$ and $M$ are equal prior to $(i,j_0)$ and $L\prec M$, we must also have $(L)_{i,j_0}\leq (M)_{i,j_0}$. But, since we know that $(L)_{i,j}<(M)_{i,j}$, applying Part 3 of Lemma~\ref{technicallemmacrit} gives  \begin{align*}(L_D)_{i,s} &= (L)_{i,s}+ |\{k\in D\mid i_k=i\}| \\ &\leq (L)_{i,s} + \sum_{j^\prime=j_0}^s (L)_{i,j^\prime} \\ &< (M)_{i,s} + \sum_{j^\prime=j_0}^s (M)_{i,j^\prime}  \\ &= (N_C)_{i,s}.\end{align*}
Hence, we cannot have $L_D=N_C$ in this case either, and so this completes the proof of Claim~\ref{uniquenessclaim}.\\


\begin{claim2} \label{leadingtermclaim}
There exists an element of $K_{t-1}$ for which $\vect{y}^{N_C}$ is the leading term.
\end{claim2}

Note that Claims~\ref{notleadingtermclaim} and~\ref{leadingtermclaim} are incompatible, thus providing the required contradiction to the assumptions on the entries of $M$ and completing the proof of Theorem~\ref{genthm}. \\

\noindent \emph{Proof of Claim~\ref{leadingtermclaim}:} 
By Lemma~\ref{splittingup}, we may write $$b= \sum_{i=0}^\infty b_i y_{r,s}^i,$$ where finitely many $b_i\neq 0$ and each $b_i\in K_{t-1}$ with lexicographic expression using only generators with coordinates less than $(r,s)$.

By Claim~\ref{uniquenessclaim}, $\vect{y}^{N_C}$ is a lex term of $b$ and so, since $(N_C)_{r,s}=h-|C|$, it is a lex term of $$z_0 = b_{h-|C|}y_{r,s}^{h-|C|}.$$

Suppose, for a positive integer $k$, that we have constructed an element $z_{k-1} \in K_{t-1}$ in which $\vect{y}^{N_C}$ is a lex term. Moreover, suppose any lex term of $z_{k-1}$ that is greater than $\vect{y}^{N_C}$, also is a lex term of $z_0$. If $\lt(z_{k-1})=\vect{y}^{N_C}$, then we have found the required element of $K_{t-1}$. Otherwise, we construct below an element $z_k\in K_{t-1}$ with the same properties as $z_{k-1}$, but in which there are fewer lex terms greater than $\vect{y}^{N_C}$. Since there are only finitely many lex terms of $z_0$ that are greater than $\vect{y}^{N_C}$, this process must end after finitely many steps, resulting in an element of $K_{t-1}$ whose leading term is $\vect{y}^{N_C}$, as required.

Let $$\lt(z_{k-1})=\vect{y}^L\succ \vect{y}^{N_C},$$ so that for some $\gamma_L, \gamma_{N_C}\in \B{K}^*$ we may write $$z_{k-1} = \gamma_L \vect{y}^L + \gamma_{N_C}\vect{y}^{N_C} + z_{k-1}^\prime.$$ In particular, observe that in $z_{k-1}^\prime$, there are fewer lex terms greater than $\vect{y}^{N_C}$ than in $z_{k-1}$. Also, $\vect{y}^L\prec\vect{y}^{M}y_{r,s}^{h}$ since the latter term is the leading term of $b$ but $\vect{y}^L \in b_{h-|C|}y_{r,s}^{h-|C|}\neq b_hy_{r,s}^h$ since $|C|>0$. Finally, for $i\in[r-1]$, let $C_i$ denote the critical coordinates in row $i$. 

Let $i_0$ be the least index such that $C_{i}=C_{i_0}$ is non-empty. Let $(c_0,d_0)$ be the least coordinate in $C_{i_0}$. Since $\vect{y}^{N_C}\prec\vect{y}^L\prec\vect{y}^{M}y_{r,s}^{h}$ and the entries of $N_C$ and $M$ at coordinates prior to $(c_0,d_0)$ are equal, we have that the entries of $L$, $M$ and $N_C$ are equal prior to $(c_0,d_0)$ as well. 

Suppose $(c_0,d)\in C_{i_0}$ is such that $(L)_{c_0,d}>0$. In this case, we proceed as follows. Since $(c_0,d)$ is a critical coordinate, there is a critical minor $[\ImidJ]^{(t-1)}\in K_{t-1}$ with maximum coordinate $(c_0,d)$ whose leading term divides $\vect{y}^M$, and so divides $\vect{y}^L$ by the previous paragraph. By Lemma~\ref{strcor2}, we have \begin{align*} \vect{y}^L &= q^\alpha[\ImidJ]^{(t-1)}\vect{y}^{L-P_\textnormal{id}} + w,\end{align*} where $w\in A^{(t-1)}$ has the property that if $\lt(w)=\vect{y}^K$, then $K\prec L$ at an entry northwest of $(c_0,d)$. Since all entries of $L$ northwest of $(c_0,d)$ are equal to those of $N_C$ and $M$, we have that $\lt(w)\prec \vect{y}^{N_C}$ as well. 

Hence, \begin{align*} z_{k-1} &= \gamma_L \vect{y}^L + \gamma_{N_C}\vect{y}^{N_C} + z_{k-1}^\prime \\&= \gamma_L(q^\alpha[\ImidJ]^{(t-1)}\vect{y}^{L-P_\textnormal{id}} + w) + \gamma_{N_C}\vect{y}^{N_C} + z_{k-1}^\prime,\end{align*} so that if we define \begin{align*} z_k &= z_{k-1} - \gamma_Lq^\alpha[\ImidJ]^{(t-1)}\vect{y}^{L-P_\textnormal{id}} \\ &= \gamma_{N_C}\vect{y}^{N_C} + \gamma_Lw + z_{k-1}^\prime,\end{align*} then we have $z_k\in K_{t-1}$ satisfying the desired properties described above.

Now, suppose each coordinate $(c_0,d)\in C_{i_0}$ is such that $(L)_{c_0,d}=0$. Thus, $L$ and $N_C$ are equal in all entries prior to $(c_0,s)$. Also, since $\vect{y}^L$ is a lex term of $b$, there must be a lex term $\vect{x}^{L^\prime}$  of $a$ so that $\vect{y}^{L}$ is a lex term of $\overleftarrow{\vect{x}^{L^\prime}}y_{r,s}^{h}$. We also have $\vect{x}^{L^\prime}\preceq \vect{x}^M$, and it follows by Part 2 of Lemma~\ref{technicallemmacrit}, that the entries in $L^\prime$ and $M$ are equal prior to $(c_0,d_0)$.

Now, as in the proof of Claim 4, we may apply homogeneity to conclude that $(L^\prime)_{c_0,d}\leq (M)_{c_0,d}$ for each $(c_0,d)\in C_{i_0}$, and if any of these inequalities are strict, then $(L)_{i_0,s}<(N_C)_{i_0,s}$, contradicting the assumption that $N_C\prec L$.  Hence, $L^\prime$ and $M$ have equal entries prior to $(c_0,s)$.

Now, let $i_1$ be the second least index such that $C_{i_1}$ is nonempty, and consider coordinates from $(c_0,s)$ to $(c_1,d_1)^-$, where $(c_1,d_1)$ is the least coordinate in $C_{i-1}$. Since $\vect{y}^{N_C}\prec \vect{y}^L$, we know that if any entry in $L$ and $N_C$ in these coordinates differ, then the first differing entry is larger in $L$ than in $N_C$. On the other hand, the entries of $N_C$ and $M$ are equal in this range of coordinates. Thus, if the first differing entry is larger in $L$ than in $N_C$, then this entry in $L^\prime$ is larger than in $M$, yet every entry prior in $L^\prime$ is equal to that in $M$, implying that $\vect{y}^M\prec \vect{y}^{L^\prime}$, a contradiction. Hence, the entries in this range of coordinates are equal in $N_C, M, L$ and $L^\prime$. 

Since all entries north-west of a critical coordinate are equal in $M, N_C, L$ and $L^\prime$, we may now repeat the above arguments with the coordinates in $C_{i_1}$, and subsequent $C_i$'s if necessary. Eventually we must find a critical coordinate with a positive entry in $L$, as otherwise we would find that $N_C=L$, contradicting the assumption that $\vect{y}^{N_C}\prec \vect{y}^L$. Hence, we can always construct the required $z_k$ and, eventually, an element of $K_{t-1}$ with leading term $\vect{y}^{N_C}$. This completes the proof of Claim 5 and the theorem.

\end{proof}

\subsection{Conclusions}
The motivating goal of this work was to demonstrate the conjecture of Goodearl and Lenagan that when $q\in \B{K}^*$ is a non-root of unity, an $\C{H}$-prime of $\qmatrix$ is generated by the set of quantum minors it contains. That this is true is already immediate corollary of the $t=mn$ case of our Theorem~\ref{genthm}. However, the theorem actually implies a sharper result since we may consider a \emph{minimal} Gr\"obner basis for the $\C{H}$-prime. The idea here is simple: if $G$ is a Gr\"obner basis for an ideal and if $g_1,g_2\in G$ are such that $\lt(g_1)$ is divisible by $\lt(g_2)$, then $G\setminus g_1$ remains a Gr\"obner basis for the ideal. With respect to $\qmatrix$, this means the following. Suppose $[\ImidJ]^{(mn)}=[\ImidJ]$ is a minor with $I=\{i_1<i_2<\cdots <i_k\}$ and $J=\{j_1<j_2<\cdots j_k\}$. If $L\subsetneq [k]$, $I^\prime=I\cap\{i_\ell\mid \ell\in L\}$ and $J^\prime=J\cap\{j_\ell\mid \ell\in L\}$, then call $[I^\prime\,|\,J^\prime]$ a \emph{diagonal subminor} of $[\ImidJ]$. From the $t=mn$ case of Theorem~\ref{genthm} we find the following. 

\begin{cor} \label{maincor}
If $q\in \B{K}^*$ is a non-root of unity, then every $\C{H}$-prime $K$ of $\qmatrix$, is generated, as a right ideal, by those quantum minors in $K$ with no diagonal subminor in $K$. These quantum minors form a minimal Gr\"obner basis for $K$ with respect to the matrix lexicographic order. 
\end{cor}

In the statement of Corollary~\ref{maincor}, ``right ideal'' can be replaced by ``left ideal'' after proving the left ideal version of Theorem~\ref{genthm}.

\begin{ex} 
Let $K$ be the $\C{H}$-prime of $\C{O}_q(\C{M}_{3,4}(\B{K}))$ corresponding to the Cauchon diagram in Figure~\ref{minimalGBexample}. By using Corollary~\ref{lindstromcor}, we find that the quantum minors in $K$ are $$\{[123|123],[123|124],[12|12], [13|12],[23|12],[23|13],[23|23]\}.$$ Theorem~\ref{genthm} says that these form a Gr\"obner basis for $K$. However, $[12|12]$ is a diagonal subminor of $[123|123]$ and $[123|124]$. Therefore, $$\{[12|12], [13|12],[23|12],[23|13],[23|23]\}$$ is a minimal (in fact reduced) Gr\"obner basis for $K$.

\begin{figure}[htbp]
\begin{center}
\begin{tikzpicture}[xscale=1.1, yscale=1.1]

\draw [fill=black] (0,1) rectangle (0.5,1.5);
\draw [fill=black] (0,0.5) rectangle (1,1);
\draw [fill=black] (0.5,0.5) rectangle (0.5,1);
\draw [color=gray] (0,0) rectangle (0.5,0.5);
\draw [color=gray] (1,0) rectangle (1.5,0.5);
\draw [color=gray] (1.5,0) rectangle (2,0.5);
\draw [color=gray] (1.5,0.5) rectangle (2,1);
\draw [color=gray] (0,1) rectangle (0.5,1.5);
\draw [color=gray] (1,1) rectangle (1.5,1.5);
\draw [color=gray] (0.5,0) rectangle (1,0.5);
\draw [color=gray] (0.5,0.5) rectangle (1,1);
\draw [color=gray] (0.5,1) rectangle (1,1.5);
\draw [color=gray] (1.5,1) rectangle (2,1.5);

\end{tikzpicture}
\caption{}
\label{minimalGBexample}
\end{center}
\end{figure}
\end{ex}

\section*{Acknowledgements}
Most of this work was completed while the author was a visiting assistant professor at the University of California, Santa Barbara. He remains grateful for the hospitality he received there. The author also thanks Ken Goodearl for many helpful discussions, and St\'ephane Launois and Milen Yakimov for thoughtful comments. The anonymous referee made a careful, and very much appreciated, reading of this manuscript, and their suggestions resulted in a highly improved paper. 

\section*{Appendix}

To assist in the reading of this paper, in particular the proof of Theorem~\ref{genthm}, we below provide an index of some terms and notation used throughout this paper. 

\begin{description}
\item[Coordinates] Beginning of Section~\ref{basicssection}.
\item[Lexicographic order] Definition~\ref{lexorderdef}.
\item[$\boldsymbol{(r,s)^-}$] Definition~\ref{lexorderdef}.\\

\item[Cauchon Diagram] Definitions~\ref{diagramdef} and~\ref{cauchondiagramdef}.
\item[$\boldsymbol{\fgraph}$] (Cauchon graph) Definition~\ref{FactorGraph}.
\item[$\boldsymbol{\Gamma_B^{(t)}(\ImidJ)}$] Definition~\ref{pathsystem}.
\item[$\boldsymbol{U(P,Q)}$] Definition~\ref{ULdef}.
\item[$\boldsymbol{L(P,Q)}$] Definition~\ref{ULdef}.
\item[$\boldsymbol{U(\C{P},\C{Q})}$ (Supremum)] Definition~\ref{supdef}.
\item[$\boldsymbol{L(\C{P},\C{Q})}$ (Infimum)] Definition~\ref{infdef}.\\

\item[$\boldsymbol{A^{(t)},A_B^{(t)}}$] Definition~\ref{ASt}.
\item[$\vect{x}^N$] Notation~\ref{monomialnotation}.
\item[Lexicographic expression] Definition~\ref{lexexpressiondef}.
\item[Lex term of] Definition~\ref{lexexpressiondef}. \\

\item[$\boldsymbol{\sigma_B^{(t)}}$] Definition~\ref{phidef}.
\item[$\boldsymbol{\overrightarrow{a}}$] Theorem~\ref{ddtheorem}.
\item[$\boldsymbol{\overleftarrow{a}}$] Theorem~\ref{ddtheorem}.
\item[$\boldsymbol{\stackrel{C}{\overleftarrow{x_{i,j}}}}$] Lemma~\ref{technicallemmacrit} and preceding paragraph.\\

\item[(Quantum) Minor $\boldsymbol{ \lbrack \ImidJ\rbrack^{(t)}_B,\lbrack \ImidJ\rbrack^{(t)}, \lbrack \ImidJ\rbrack }$] Definition~\ref{minordef}.
\item[Diagonal coordinate (of a minor)] Definition~\ref{coordinatesdef}.
\item[Maximum coordinate (of a minor)] Definition~\ref{coordinatesdef}.\\

\item[$\boldsymbol{\prec} $] Definition~\ref{revlex}.
\item[$\boldsymbol{\lt(a)}$ (leading term of $a\in A^{(t)}$)]Definition~\ref{leadingtermdef}.
\item[Gr\"obner Basis] Definition~\ref{grobnerdef}.\\

\item[$\boldsymbol{N_C}$] See Expression~(\ref{NCdefinition}) just prior to Claim~\ref{notleadingtermclaim} in proof of Theorem~\ref{genthm}.
\item[Critical Minor] A minor in $K_{t-1}$ whose leading term divides $\lt(b)=\lt(\overleftarrow{a}y_{r,s}^h)$.
\item[Critical Coordinate] A coordinate $(i,j)$ that is northwest of $(r,s)$ such that there exists a critical minor with $(i,j)$ as its maximum coordinate.

%

\end{description}
\bibliography{casteels1}
\bibliographystyle{amsplain}
\end{document}